\documentclass[11pt]{article}
\usepackage{amsmath,amsthm,amsfonts,amssymb,amscd, amsxtra,mathrsfs,commath,url}
\usepackage{graphicx,float}
\usepackage{booktabs}
\usepackage{siunitx}
 \usepackage{soul,color}
\usepackage{color}
\usepackage[T1]{fontenc}
\usepackage[brazilian]{babel}
\usepackage{graphicx,float}
\usepackage{subfig} 
\usepackage[margin=2.6 cm,nohead]{geometry}
\usepackage{bbm}
\usepackage[ruled]{algorithm2e} 
\usepackage{cite}
\usepackage{threeparttablex,color,soul,colortbl,hhline,rotating,booktabs,longtable}
\usepackage[table]{xcolor}
\definecolor{lightgray}{gray}{0.95}
\newtheorem{theorem}{Teorema}[section]
\newtheorem{corollary}{Corolário} 
\newtheorem{lemma}{Lema}
\newtheorem{proposition}{Proposição}
\newtheorem{definition}{Definição}
\newtheorem{example}{Exemplo}
\newtheorem{remark}{Observação}

\newcommand\ddfrac[2]{\frac{\displaystyle #1}{\displaystyle #2}}

\def\min{\displaystyle\operatorname{min}}
\def\max{\operatorname{max}}

\def \T {{\scriptscriptstyle\mathrm{T}}}

 \usepackage{todonotes}
\newtheorem{step size}{step size}

\begin{document}

\title{Introdu\c c\~ao a otimiza\c c\~ao de portf\'olio\footnote{Disclaimer: Os estudos apresentados neste artigo tem apenas finalidade acadêmica e os autores não possuem  nenhum vínculo com as empresas responsáveis pelo ativos citado no texto.}}
\author{
O.  P. Ferreira  \thanks{Instituto de Matem\'atica e Estat\'istica, Universidade Federal de Goi\'as,  CEP 74001-970 - Goi\^ania, GO, Brazil, E-mails:  {\tt  orizon@ufg.br, gui.araujo.franca2002@gmail.com, max@ufg.br}. Os autores foram parcialmente financiados pelo CNPq processos  304666/2021-1 and 159341/2021-3.}
\and
G. A. Fran\c ca \footnotemark[2]
\and
M. V. Lemes   \footnotemark[2]
}
\maketitle

\maketitle

\maketitle
\vspace{.5cm}

\noindent
{\bf Abstract:}
Neste trabalho, introduzimos a Teoria Moderna do Portfólio utilizando conceitos básicos de álgebra linear, cálculo diferencial, estatística e otimização. Essa teoria nos permite medir o retorno e o risco de uma carteira de investimentos, servindo como base para a tomada de decisões no mercado financeiro. Como aplicação, apresentaremos quatro estratégias de investimento bastante simples, que visam minimizar o risco de investir em apenas dois ativos, superando o rendimento do CDI.

\noindent
{\bf Keywords:}  Teoria moderna do portfólio; retorno;  risco;   optimization.

\section{Introdução}
A {\it  teoria moderna do portfólio} foi introduzida por Harry Markowitz na década de 1950 (veja \cite{Markowitz1952}). Essa teoria oferece uma estrutura matemática para que investidores avessos ao risco possam montar uma carteira ou portfólio que maximize o retorno esperado para um determinado nível de risco, considerando um conjunto pré-definido de ativos. É importante destacar que o risco é uma parte inerente de uma recompensa mais alta, e pode ser quantificado medindo a variação diária dos ativos, através do desvio padrão da série temporal de seus retornos. Assim, o risco de um portfólio pode ser medido pela combinação dos riscos de cada ativo que o compõe. Segundo essa teoria, é possível montar uma ``fronteira eficiente" de carteiras ótimas que ofereçam o máximo retorno esperado para um nível de risco fixado. Neste artigo, revisitamos a teoria moderna do portfólio e descrevemos a fronteira eficiente, mostrando que ela forma um ramo de hipérbole nas variáveis risco e retorno (veja \cite{Merton1972}). Também calculamos o portfólio de risco mínimo e utilizamos esse conceito para apresentar e analisar uma estratégia de investimento que minimiza a relação entre o risco e o retorno.

Nosso principal objetivo com este artigo é tornar a Teoria Moderna do Portfólio acessível aos alunos dos primeiros anos de graduação. Para isso, apresentamos todos os conceitos necessários para uma boa compreensão do assunto, exigindo apenas conhecimentos básicos de Álgebra Linear, Estatística e Cálculo como pré-requisitos. É importante destacar que, embora seja uma teoria valiosa, ela nem sempre funciona bem na prática. Por exemplo, a teoria não impede a concentração de recursos em determinados ativos, o que pode resultar em um modelo com pesos desprezíveis em alguns ativos e um peso elevado em outros. Do ponto de vista do mercado, nenhuma dessas situações é desejada, pois pesos desprezíveis possuem custos proporcionais relevantes e pesos elevados aumentam o risco do portfólio ao concentrá-lo em poucos ativos. Para que a teoria de Markowitz possa ser utilizada na prática, são necessárias algumas hipóteses adicionais e ajustes mais refinados para evitar tais situações. Esses detalhes estão além dos objetivos deste artigo introdutório. No entanto, mostraremos que, com uma escolha adequada dos ativos, esses problemas podem ser mitigados, conduzindo a bons resultados.

A estrutura do artigo é a seguinte: Na Seção~\ref{se:pre}, apresentaremos alguns conceitos básicos de finanças e formularemos o nosso problema. Na Seção~\ref{se:notacoes}, apresentamos as notações e resultados básicos de Álgebra Linear, Estatística e Cálculo utilizados ao longo do artigo. Na Seção~\ref{se:tmp}, estudamos os conceitos de retorno e risco de uma carteira e suas propriedades básicas. Na Seção~\ref{sec:PRM}, apresentamos uma fórmula explícita para a fronteira eficiente de um conjunto de ativos e calculamos o portfólio de mínimo risco associado a esse conjunto. Na Seção~\ref{Sec:Aplicacao}, apresentamos e discutimos simulações de estratégias de alocação de capital. Por fim, concluímos o artigo na Seção~\ref{sec:conclusion} com algumas considerações finais.

\section{Preliminares} \label{se:pre}
Nesta seção apresentaremos  alguns conceitos básico de finanças. Em particular, mostramos de modo   intuitivo como o desvio padrão pode ser visto como uma  medida de risco. Também  formulamos o problema central que estamos interessados em discutir neste artigo.

\subsection{Risco em finanças}
Em finanças, {\it risco é a probabilidade de que os resultados reais de um investimento sejam diferentes dos resultados esperados, mesmo quando esses  resultados sejam positivos}. Em temos mais técnicos, o {\it risco é definido como a volatilidade dos retornos}, envolvendo  tanto retornos positivos quanto negativos. Assim, o uso adequado em finanças  de um   conceito envolvendo  {\it risco versus  retorno} deve significar que  ativos mais arriscados devem ter retornos esperados mais altos para compensar os investidores pela maior volatilidade. Por exemplo, os retornos obtidos ao investir em uma Startup recém criada ou em criptomoedas  podem alcançar retornos muito superiores que  investir em  Título de Renda Fixa, no entanto investir nesses ativos é bem mais arriscado que investir em Título do Tesouro\footnote{ Veja mais sobre isto em https://corporatefinanceinstitute.com/resources/knowledge/finance/risk/}.

\subsection{O Desvio padrão como medida de risco}
O {\it desvio padrão} é um conceito matemático que mede a dispersão dos  dados individuais em relação ao seu valor médio.  Ao avaliar um ativo  o investidor pode usar o desvio padrão dos retornos do ativo como medida de risco devido à sua capacidade de mostrar a volatilidade de uma negociação. Em outros termos, medir a probabilidade do retorno de um ativo  mover em uma certa direção,  o que pode resultar em ganhos ou perda ao investir no ativo.   Por exemplo, um ativo  que atinge altos níveis de retorno  e reverte fortemente  esta tendencia com muita frequência  possui  um desvio padrão  alto no período em consideração.  Isso significa que o ativo é altamente volátil e carrega um alto  grau  de risco com o qual um investidor avesso ao risco se sentirá desconfortável para investir nele. Por outro lado,  ativos  com um histórico de fornecer retornos com pequenas variações  tem  baixo risco, pois provavelmente permanecerão no mesmo intervalo de retorno por um longo tempo.  Assim,  ao usar o desvio padrão para medir o risco, os investidores  estão interessados em saber como os retornos estão espalhados e assim  determinar o nível de risco do investimento.
\begin{remark} \label{re:ddpmr1}
    É importante destacar que o desvio padrão como medida  de risco mostra apenas como os retornos  de um investimento em um determinado período  são distribuídos. No entanto,  isso  não significa que esta distribuição terá o mesmo comportamento no futuro. De fato, os  investimentos podem ser afetados por outros fatores não relacionados diretamente ao comportamento prévio do  ativo, como mudanças nas taxas de juros futuros e novas concorrência de mercado e etc, e assim o  retorno   pode ficar fora do intervalo previsto. Isso significa que o desvio padrão não deve ser usado como a ferramenta final de medição de risco, mas deve ser usado juntamente com outras funções de medição de risco.
\end{remark}
\begin{remark} \label{re:ddpmr2}
    Estudos mostram (veja em \cite{Roncalli2014}) que quando os retornos de determinado ativo possui distribuição normal, o desvio padrão é extremamente adequado para medir o risco. Isso significa que estamos assumindo que o ativo possui uma probabilidade uniforme de atingir valores acima ou abaixo da média. Esta hipótese  pode não se aplicável  a todos os tipos de ativos. Assim, a seleção prévia dos ativos é parte fundamental para a montagem de um portfólio vencedor.
\end{remark}
\subsection{Formulação do problema}{}
Na Tabela~\ref{tab:valor} listamos os preços   de fechamento ajustado diários dos  ativos  IVVB11, BOVA11 e BBAS3 negociados na bolsa brasileira entre os dia 09 à  16/04/2021, estes dados foram  obtidos usando o Google Finance.
\begin{table}[H]
    \begin{footnotesize}
        \begin{center}
            \begin{tabular}{lcccc}
                Período    & $IVVB11$ & $BOVA11$ & $BBAS3$ \\
                \hline
                09/04/2021 & $254,00$ & $113,01$ & $29,19$ \\
                12/04/2021 & $256,54$ & $114,40$ & $29,55$ \\
                13/04/2021 & $257,20$ & $114,67$ & $29,55$ \\
                14/04/2021 & $254,29$ & $115,60$ & $29,60$ \\
                15/04/2021 & $254,95$ & $116,20$ & $29,64$ \\
                16/04/2021 & $255,00$ & $116,46$ & $29,77$ \\
                \hline
            \end{tabular}
            \caption{\footnotesize  Preços de fechamento ajustado.} \label{tab:valor}
        \end{center}
    \end{footnotesize}
\end{table}
Usando os dados da tabela Tabela~\ref{tab:valor} podemos calcular os retornos diário  dos ativos entre os dia 12 à  16/04/2021. Então  podemos calcular o retorno médio e  desvio padrão de cada ativo neste período.  Na  Tabela~\ref{tab:retornoP1},  apresentamos  os retornos diários  dos ativos IVVB11, BOVA11 e BBAS3 entre os dia 12 à  16/04/2021.  Na última coluna desta tabela estão  os  retornos diários,  retorno médio e  desvio padrão do portifólio /portfólio com capital igualmente distribuídos,  ou seja,  composto por $1/3$  do capital aplicado no IVVB11, $1/3$ no BOVA11 e  $1/3$  no BBAS3, para facilitar a notação denotamos este portfólio por
$$
    P_1:=((1/3)\text{IVVB11}, (1/3)\text{BOVA11}, (1/3)\text{BBAS3}).
$$
\begin{table}[H]
    \begin{footnotesize}
        \begin{center}
            \begin{tabular}{lrrr|r}
                Período       & IVVB11    & BOVA11   & BBAS3    & Portfólio P1 \\
                \hline
                12/04/2021    & $0,0100$  & $0,0123$ & $0,0123$ & $0,0115$     \\
                13/04/2021    & $0,0026$  & $0,0024$ & $0,0000$ & $0,0016$     \\
                14/04/2021    & $-0,0113$ & $0,0081$ & $0,0017$ & $-0,0005$    \\
                15/04/2021    & $0,0026$  & $0,0052$ & $0,0014$ & $0,0023$     \\
                16/04/2021    & $0,0002$  & $0,0022$ & $0,0044$ & $0,0025$     \\
                \hline
                Retorno Médio & $0,0008$  & $0,0060$ & $0,0039$ & $0,0036$     \\
                Desvio Padrão & $0,0077$  & $0,0042$ & $0,0049$ & $0,0046$
            \end{tabular}
            \caption{\footnotesize  Retorno do portfólio P1.} \label{tab:retornoP1}
        \end{center}
    \end{footnotesize}
\end{table}
Vamos comparar o desempenho do  portfólio $P1$ com a carteira/portfólio  composto por $20\%$  do capital aplicado no IVVB11, $50\%$ no BOVA11 e  $30\%$  no BBAS3, o qual denotamos por
$$
    P_2:=(0.2\text{IVVB11}, 0.5\text{BOVA11}, 0.3\text{BBAS3}).
$$
Os retornos diários,  retorno médio e  desvio padrão do portfólio $P_2$ entre os dia 12 à  16/04/2021 esta apresentado na  última coluna da Tabela~\ref{tab:retornoP2}.
\begin{table}[H]
    \begin{footnotesize}
        \begin{center}
            \begin{tabular}{lrrr|r}
                Período       & IVVB11    & BOVA11   & BBAS3    & Portfólio P2 \\
                \hline
                12/04/2021    & $0,0100$  & $0,0122$ & $0,0123$ & $0,0118$     \\
                13/04/2021    & $0,0025$  & $0,0023$ & $0,0000$ & $0,0016$     \\
                14/04/2021    & $-0,0113$ & $0,0081$ & $0,0016$ & $0,0022$     \\
                15/04/2021    & $0,0025$  & $0,0051$ & $0,0013$ & $0,0035$     \\
                16/04/2021    & $0,0001$  & $0,0022$ & $0,0043$ & $0,0024$     \\
                \hline
                Retorno Médio & $0,0008$  & $0,0060$ & $0,0039$ & $0,0043$     \\
                Desvio Padrão & $0,0077$  & $0,0042$ & $0,0049$ & $0,0042$
            \end{tabular}
            \caption{\footnotesize  Retorno do portfólio P2.} \label{tab:retornoP2}
        \end{center}
    \end{footnotesize}
\end{table}
As Tabelas~\ref{tab:retornoP1} e \ref{tab:retornoP2} mostram que  o retorno médio do portfólio $P2$ é $21\%$ maior que o retorno médio do portfólio $P1$, sendo que  o risco do portfólio $P2$ é  $9\%$ menor que do portfólio $P1$.  Assim, o portfólio $P2$ possui um maior retorno com menor risco quando comparado com o portfólio $P1$.  Isso mostra que, uma vez escolhido os ativos,  a porcentagem   dos recursos alocada em cada um dos ativos, isto é, a montagem do portfólio,  pode impactar significativamente o retorno  e  risco do portfólio. Então o nosso problema é  otimizar a montagem de um portfólio, o qual pode ser descrito   da   seguinte forma:

{\bf Problema:} {\it Uma vez escolhido os ativos,  como devemos proceder para  montar uma  carteira ótima, isto é, uma carteira que ofereça o máximo retorno possível  para um determinado nível de risco fixado. Além disso,   como podemos  montar   uma carteira com o menor risco possível.}

\begin{remark}\label{obs:desempenho}
    Entre dois portfólios com o mesmo retorno podemos considerar aquele que teve o menor risco como o que teve o melhor desempenho entre os dois. Do mesmo modo se dois portfólios tiveram o mesmo risco será considerado o portfólio de melhor desempenho aquele que possuir o maior retorno. De uma forma geral, o quociente entre o retorno e o risco pode ser considerado uma \textbf{medida de desempenho} entre dois portfólios. Essa medida nos diz quanto de retorno o portfólio adiciona para cada $1\%$ de risco corrido. Logo, quanto menor for este quociente risco/retorno melhor será o desempenho do portfólio. Esta medida nos permite comparar o desempenho entre portfólios com diferentes níveis de risco e diferentes taxas de retorno.
\end{remark}

\section{Notações e resultados básicos} \label{se:notacoes}
Nesta seção apresentaremos notação e resultados básicos de Álgebra Linear,  Estatística e Cálculo Diferencial   usados através do artigo, mais detalhes sobre os assuntos discutidos aqui podem ser encontrados, por exemplo, em  \cite{Boldrini1980,Strang2013}.

\subsection{Matrizes e transforma\c c\~oes lineares} \label{sec:mtl}

Seja $n$ um número inteiro positivo.  Definimos o espaço vetorial $\mathbb{R}^n, $ como sendo o conjunto das $n$-úplas ordenadas, $u=(u_1, u_2, \ldots, u_n)$, em que cada $u_i \in \mathbb{R}$.  Sejam $u=(u_1, u_2, \ldots, u_n)$, $v=(v_1, v_2, \ldots, v_n)$ vetores em $\mathbb{R}^n$ e $\alpha \in \mathbb{R}$. Então definimos a {\it soma}  $u+v \in \mathbb{R}^n$ e o {\it produto por escalar} $\alpha u \in \mathbb{R}^n$ como sendo:
$ u+v=(u_1+v_1, u_2+v_2, \ldots, u_n+v_n)$ e  $\alpha u= (\alpha u_1, \alpha  u_2, \ldots, \alpha u_n)$, respectivamente. O {\it produto interno} de $u$ por $v$ é definido como sendo $\langle u,v \rangle=u_1v_1+u_2v_2+\ldots+u_nv_n$ e a {\it norma euclidiana}  associada  por $\|u\|=\sqrt{\langle u,u\rangle}.$   Assim,   $\|u + v\| \leq \|u\| + \|v\|$ e $\langle u,v \rangle \leq  \|u\| \|v\|$. Chamamos o espaço $\mathbb{R}^n$ dotado da norma $||\cdot||$ de {\it Espaço Euclidiano $n$-dimensional}. Definimos $e:=(1,1,\ldots,1) \in \mathbb{R}^n$,
$e^1:=(1,0,\ldots, 0), e^2:=(0,1,\ldots,0), \ldots e^n:=(0,0,\ldots,1)$. Portanto,  $e=e^1+e^2+\ldots+e^n$ e em geral  $u=u_1 e^1+u_2 e^2+\ldots+u_n e^n$.

O conjunto das matrizes de ordem $m\times n$ é denotado por $\mathbb{R}^{m \times n}$ e  uma matriz $A\in  \mathbb{R}^{m \times n}$ por
$$
    A=\begin{bmatrix}
        a_{11} & \dots  & a_{1n} \\
        \vdots & \ddots & \vdots \\
        a_{m1} & \dots  & a_{mn} \\
    \end{bmatrix},
$$
ou $A=(a_{ij})$.  A {\it matriz transposta} de $A$ é denotada por  $A^T \in \mathbb{R}^{n\times m}$.   Vamos adotar a convenção de que vetores em $\mathbb{R}^n$ são matrizes em $\mathbb{R}^{n \times 1}$, isto é, $\mathbb{R}^n\equiv \mathbb{R}^{n \times 1}$. Neste caso, temos  $\|u\|=u^Tu$. Definimos a {\it matriz identidade} de ordem $n$ como sendo a matriz quadrada $I \in \mathbb{R}^{n \times n}$, tal que os elementos de sua diagonal principal são iguais a $1$ e todos os outros elementos são $0$. Dadas  $A=(a_{ij}), B=(B_{ij}) \in \mathbb{R}^{m \times n}$, definimos a matriz soma $A+B:= ((a+b)_{ij})\in \mathbb{R}^{m \times n}$  tal que  $(a+b)_{ij}=a_{ij}+b_{ij}$. Dado  $\alpha \in \mathbb{R}$  definimos a matriz  produto por  $\alpha A:=((\alpha a)_{ij}) \in \mathbb{R}^{m \times n}$  tal que $(\alpha a)_{ij}=\alpha  a_{ij}$. Sejam  $A \in \mathbb{R}^{m\times n}$ e $B \in \mathbb{R}^{n \times p}$,  definimos o produto das matrizes $A$ por $B$ como sendo a matriz $A B \in \mathbb{R}^{m\times p}$ tal que $ (ab)_{ij}=\sum_{k=1}^na_{ik} b_{kj}$. Assim,  dados  matrizes $A \in \mathbb{R}^{m\times n}$ e $B,C \in \mathbb{R}^{n \times p}$, então valem  as seguintes  igualdades: $A(B+C)=AB+AC$, $(A+B)^T=A^T+B^T$, $(\alpha A)^T=\alpha A^T$, $(A^T)^T=A$ e  $(AB)^T=B^TA^T$.  Dado  $A \in \mathbb{R}^{n \times n}$,  a matriz inversa de $A$ é definida  como sendo a única matriz $A^{-1} \in \mathbb{R}^{n \times n}$ tal que $A A^{-1}=A^{-1} A=I$. Quando existir a matriz inversa, diremos que $A$ é uma matriz inversível e, caso contrário, diremos que $A$ é uma matriz singular.  Suponha que $A,B \in \mathbb{R}^{n\times n}$ são matrizes inversíveis. Então o produto $AB$ é uma matriz inversível e $(AB)^{-1}=B^{-1}A^{-1}$. Uma {\it norma de matrizes } em  $ \mathbb{R}^{n \times n}$ é uma função a valores reais $\| \cdot \|:  \mathbb{R}^{n \times n} \to [0, +\infty)$ com as seguintes propriedades: $\|A\| \geq 0$, $\|A\| = 0$ se e somente se  $A$ é a matriz nula $0$,   $\|\alpha A\| = |\alpha|\|A\|$, $\|A + B \| \leq \|A\| + \|B\|$,  $\|AB\| \leq \|A\| \|B\|$, para toda $\displaystyle{A}, \displaystyle{B} \in \mathbb{R}^{n \times n}$.
Seja $\| \cdot \|$ é uma norma de vetores em  $ \mathbb{R}^{n \times 1}$, então
$$
    \|A\| = \max_{\|\displaystyle{x}\| = 1} \|A\displaystyle{x}\|,
$$
é uma norma de matrizes chamada {\it norma induzida} ou norma de matrizes  associada com a norma de vetor. Podemos mostrar que dados $u\in  \mathbb{R}^{n \times 1}$ e  $A\in \mathbb{R}^{n \times n}$. Se  $\| \cdot \|$ é uma norma de matrizes induzida por uma  norma de vetores então
\begin{equation} \label{eq:rnvnm}
    \|Au\| \leq \|A\| \|u\|,
\end{equation}
veja \cite[7B, pág. 355]{Strang2013}. A seguir introduzimos uma definição que será fundamental na teoria que será desenvolvida na Seção~\ref{eq:mml} e  nas seções subsequentes. Dizemos que uma matriz  simétrica $A\in \mathbb{R}^{n \times n}$ é {\it positiva definida} se $x^TAx>0$ para todo  $x\in \mathbb{R}^{n \times 1}$ e $x\neq 0$. Toda matriz positiva definida $A$ é inversível,   isto é,     $\det A\neq0$, veja \cite[6B, pág. 318]{Strang2013}.

Um sistema de $m$ equações  $n$ variáveis $x_1, x_2, \ldots, x_. \in \mathbb{R}$ é escrito na seguinte forma
$$
    \left\{ \begin{array}{ccccccccc}
        a_{11}x_1 & +      & a_{12}x_2 & +      & \ldots & +      & a_{1n}x_n & =      & b_1    \\
        a_{21}x_1 & +      & a_{22}x_2 & +      & \ldots & +      & a_{2n}x_n & =      & b_2    \\
        \vdots    & \vdots & \vdots    & \vdots & \ddots & \vdots & \vdots    & \vdots & \vdots \\
        a_{m1}x_1 & +      & a_{m2}x_2 & +      & \ldots & +      & a_{mn}x_n & =      & b_m
    \end{array}
    \right.
$$
onde  $a_{i,j}, b_i \in \mathbb{R}$. Um vetor ${\bar x}=({\bar x}_1, {\bar x}_2, \ldots, {\bar x}_n) \in \mathbb{R}^n$ que satisfaz todas as $m$ equações simultaneamente é dito ser uma solução do sistema. Note que um sistema de $m$ equações nas $n$ variáveis $x_1, \ldots, x_n$ pode ser escrito em forma de uma equação matricial $Ax=b,$ onde
$$A=\begin{bmatrix}
        a_{11} & \dots  & a_{1n} \\
        \vdots & \ddots & \vdots \\
        a_{m1} & \dots  & a_{mn} \\
    \end{bmatrix}, \, \quad
    x=\begin{bmatrix}
        x_1    \\
        \vdots \\
        x_n
    \end{bmatrix}, \quad \,
    b=\begin{bmatrix}
        b_1    \\
        \vdots \\
        b_m
    \end{bmatrix}.
$$
Dado um sistema linear $Ax=b$ com  $A$  quadrada e inversível, então $x=A^{-1}b$ é sua  uma única solução.

Uma fun\c c\~ao $T:\mathbb{R}^n \to \mathbb{R}^m$ \'e chamada de {\it uma transforma\c c\~ao linear } se, para quaisquer $u, v \in \mathbb{R}^n$ e $\lambda \in \mathbb{R}$, s\~ao satisfeitas as seguintes propriedades: $T(u+v)=T(u) + T(v)$ e $T(\lambda u)=\lambda T(u)$. Sejam $A \in \mathbb{R}^{m \times n}$ e $T:\mathbb{R}^n \to \mathbb{R}^m$ definida por $Tx=Ax$.  \'E imediato verificar que a  fun\c c\~ao $T$ \'e uma transforma\c c\~ao linear.  Dada uma transforma\c c\~ao linear $T:\mathbb{R}^n \to \mathbb{R}^m$,  existe uma  matriz $A \in \mathbb{R}^{m \times n}$ que a representa,  i. e., $T$ pode ser colocada na forma $Tx=Ax$.  Neste caso, dizemos que $A$ \'e a \'unica matriz da transforma\c c\~ao linear $T$ com  respeito aos $n$ vetores can\^onicos de $\mathbb{R}^n$. Portanto, fixado os $n$ vetores can\^onicos de $\mathbb{R}^n$, est\'a estabelecida uma correspond\^encia biun\'{\i}voca entre as matrizes de $\mathbb{R}^{m \times n}$ e as transforma\c c\~ao linear $T:\mathbb{R}^n \to \mathbb{R}^m$. Em particular, uma transforma\c c\~ao linear $T: \mathbb{R}^n \to \mathbb{R}$ \'e chamada de {\it  funcional linear}. Ent\~ao,  existe um \'unico  vetor $a=(a_1, \cdots, a_n)^T\in \mathbb{R}^{ 1 \times n}$,  ou matriz $1\times n$,  tal que o funcional linear $T$ seja escrito na forma $Tv=a^Tv,$ onde $a_i=Te^i$ e $e^i=(0,\cdots, 1, \cdots, 0)^T$ \'e o i-\'esimo vetor can\^onico.

\subsection{ Diferenciabilidade de fun\c c\~oes reais de v\'arias vari\'aveis}
A  $i$-\'esima {\it derivada parcial} de $f:\mathbb{R}^{n} \to \!\!R$  no ponto $p\in \mathbb{R}^{n} $ \'e definida por
\begin{equation} \label{def:ddpar}
    \frac{\partial f}{\partial x_i}(p)=\lim_{t \to 0}\frac{f(p+te^i)-f(p)}{t}.
\end{equation}
\begin{definition} \label{Def:dif}
    Dizemos que $f:\mathbb{R}^{n} \to \!\!R$ \'e diferenci\'avel no ponto $p\in \mathbb{R}^{n} $ se existe um funcional linear $T: \mathbb{R}^n \to \mathbb{R}$  tal que, para todo $v\in \mathbb{R}^n$ com $p+v \in \mathbb{R}^{n} $, vale
    \begin{equation} \label{def:dif}
        f(p+v)=f(p)+Tv+ \|v\|\eta(v),\;\;\;\; \mbox{onde}\;\; \lim_{v\to 0}\eta(v)=0.
    \end{equation}
\end{definition}
\begin{remark} \label{obs:fl}
    De acordo com a Seção~\ref{sec:mtl}, existe um \'unico  vetor $a=(a_1, \cdots, a_n)\in \mathbb{R}^n$ tal que o funcional linear $T: \mathbb{R}^n \to \mathbb{R}$ seja escrito na forma $Tv=a^Tv,$ onde $a_i=Te^i$.
\end{remark}
Seja $f:\mathbb{R}^{n} \to \!\!R$ uma fun\c c\~ao diferenci\'avel no ponto $p\in \mathbb{R}^{n} $. Vamos agora calcular o funcional linear $T: \mathbb{R}^n \to \mathbb{R}$ da Defini\c c\~ao \ref{Def:dif}, i.e., o vetor $a$ que representa $T$ dado pela Observa\c c\~ao \ref{obs:fl}. Assim tomando $v=te^i$ em \eqref{def:dif} temos $f(p+te^i)=f(p)+tTe^i+ |t|\eta(te^i)$ o que implica
$$
    \frac{f(p+te^i)-f(p)}{t}=Te^i\pm\eta(te^i).
$$
Agora, fazendo $t\to 0$  na \'ultima igualdade e notando que $\lim_{t\to 0}\eta(te^i)=0$  temos  de \eqref{def:ddpar} que
\begin{equation} \label{eq:dp-1}
    \frac{\partial f}{\partial x_i}(p)=Te^i.
\end{equation}
A igualdade \eqref{eq:dp-1} vale para $i=1, \cdots, n$. Assim,  segue da Observa\c c\~ao~\ref{obs:fl} que o funcional linear $T: \mathbb{R}^n \to \mathbb{R}$ da Defini\c c\~ao \ref{Def:dif} \'e dado por
\begin{equation} \label{eq:dp-2}
    Tv=\sum_{i=1}^{n}\frac{\partial f}{\partial x_i}(p)v_i.
\end{equation}
Definimos o  {\it vetor gradiente} de $f$ no ponto $p$ por
$$
    \nabla f(p)=\left( \frac{\partial f}{\partial x_1}(p), \cdots,  \frac{\partial f}{\partial x_n}(p)\right)^T \in  \mathbb{R}^{n \times 1}.
$$
Portanto, quando $f$ \'e diferenci\'avel no ponto $p$, segue de \eqref{eq:dp-2} que o vetor gradiente $\nabla f(p) \in  \mathbb{R}^{n \times 1}$ determina o funcional linear $T$ da Defini\c c\~ao \ref{Def:dif}, o qual  denotamos de agora em diante por $df(p)$. Assim, quando $f$ \'e diferenci\'avel no ponto $p$,  a {\it diferencial} de $f$ nesse ponto \'e  definida com sendo o funcional linear $df(p):\mathbb{R}^n \to \mathbb{R}$ dado por
\begin{equation} \label{eq:dif-3}
    df(p)v=\nabla f(p)^T v.
\end{equation}
Note que n\~ao \'e necess\'ario que $f$ seja diferenci\'avel para definir o vetor gradiente \'e necess\'ario apenas a exist\^encia das derivadas parciais. Mas neste caso, o funcional linear associado ao vetor gradiente n\~ao ``merece o nome de diferencial''.
\begin{example} \label{ex:fq}
    Seja $Q \in \mathbb{R}^{n\times n}$. Defina $f:\mathbb{R}^n \to \mathbb{R}$ pondo $f(x)=x^TQx$. A diferencial e o vetor gradiente de $f$ no ponto $x$ s\~ao dados, respectivamente, por:
    \begin{equation} \label{gfq}
        df(x)v= v^T(Q+Q^T)x, \qquad  \nabla f(x)=(Q+Q^T)x.
    \end{equation}
    De fato, primeiro note que $f(x+v)=f(x)+v^T(Q+Q^T)x+ v^TQv$ ou equivalentemente
    $$
        f(x+v)=f(x)+v^T(Q+Q^T)x+ \|v\|\eta(v), \qquad  \eta(v):= \frac{1}{\|v\|}v^T Q v.
    $$
    Usando \eqref{eq:rnvnm} temos   $|\eta(v)|\leqslant \|Q\|\|v\|$. Assim,  $\lim_{v\to 0}\eta(v)=0$. Portanto $f$ \'e diferenci\'avel e as igualdades  \eqref{gfq} valem.
\end{example}
\begin{theorem} \label{th:rd}
    Sejam $f,\, g,\, h: \mathbb{R}^n  \to \mathbb{R}$ fun\c c\~oes diferenci\'aveis  para todo $x\in \mathbb{R}^{n} $. Ent\~ao $f+g : \mathbb{R}^{n} \to \mathbb{R}$ e  $fg : \mathbb{R}^{n} \to \mathbb{R}$ s\~ao diferencia\'veis e valem as seguintes propriedades:
    \begin{enumerate}
        \item[i)] $\nabla (f+g)(x)= \nabla f(x) +\nabla g(x)$;
        \item[ii)] $\nabla (fg)(x)= f(x) \nabla g(x)+g(x) \nabla f(x)$.
    \end{enumerate}
\end{theorem}
Para mais detalhes sobre o assunto tratado nesta seção veja \cite[Capítulo 2]{Avila1987}.
\subsubsection{Método dos multiplicadores de Lagrange} \label{eq:mml}
Nosso objetivo nesta seção é determinar, sobre hipóteses adequadas, a solução explicita de um problema quadrático homogêneo com restrições lineares. Para isso, vamos primeiro relembrar  o método dos multiplicadores de Lagrange para resolver  o seguinte problema
\begin{equation} \label{eq:gp}
    \begin{array}{cr}
        \text{Minimizar}_x ~ & f(x)     \\
        \text{tal que} ~     & h_1(x)=0 \\
                             & \vdots   \\
                             & h_s(x)=0
    \end{array}
\end{equation}
onde {\it $f:\mathbb{R}^{n\times 1}\to  \mathbb{R}$, $h_i:\mathbb{R}^{n\times 1}\to  \mathbb{R}$, para $i=1, \ldots, s$ são funções  diferenciáveis e os seguintes  vetores $\nabla h_1(x), \ldots, \nabla h_s(x)$  sejam linearmente independentes}.  A função  Lagrangeano $L: \mathbb{R}^{n\times 1}\times \mathbb{R}^{s\times 1}\to \mathbb{R}$ associada ao problema~\eqref{eq:gp}  é definida   por
\begin{equation} \label{eq:lagra}
    L(x, \lambda):=f(x)-\sum_{i=1}^s \lambda_ih_i(x).
\end{equation}
O método dos multiplicadores de Lagrange assegura que para  toda solução local $x^*\in \mathbb{R}^{n\times 1} $ do problema~\eqref{eq:gp}  existe $\lambda^*\in \mathbb{R}^{s\times 1}$  tal que $(x^*, \lambda^*)$ é ponto crítico do Lagrangeano  $L$, isto é, $(x^*, \lambda^*)$ é solução da equação  $\nabla L(x, \lambda)=0$, ou ainda, $(x^*, \lambda^*)$ é  solução das seguintes equações:
\begin{align*}
    \nabla f(x)-\sum_{i=1}^s \lambda_i \nabla h_i(x) & =0     \\
    h_1(x)                                           & =0     \\
                                                     & \vdots \\
    h_s(x)                                           & =0
\end{align*}
Para mais detalhes veja por exemplo  \cite[Capítulo 3]{Avila1987}.  As equações acima  são conhecidas como {\it condições  de otimalidade} do problema~\eqref{eq:gp}.  Para provar o próximo  teorema  vamos usar o método dos multiplicadores de Lagrange.
\begin{theorem} \label{th:sqp}
    Sejam $Q \in \mathbb{R}^{n\times n}$ uma matriz simétrica positiva definida, $A \in \mathbb{R}^{s\times n}$  e  $b \in \mathbb{R}^{s\times 1}$.  Suponhamos que $posto(A)=s<n$. Nessas condições, o  minimizador do problema quadrático
    \begin{equation} \label{eq;pq}
        \begin{array}{cr}
            \text{Minimizar}_x ~ & \frac{1}{2}x^TQx \\
            \text{tal que} ~     & Ax=b
        \end{array}
    \end{equation}
    é  $ x^*:=Q^{-1}A^T(AQ^{-1}A^T)^{-1}b\in \mathbb{R}^{n\times 1}$.
\end{theorem}
\begin{proof}
    Vamos aplicar o método dos  multiplicadores de Lagrange para calcular a solução do problema~\eqref{eq;pq}. Vamos primeiramente  fixar as algumas notações. Seja $A=(a_{ij})\in \mathbb{R}^{s\times n}$, então denote a $i-$esima linha de $A$ por $A_i=(a_{i1} \ldots  a_{in})\in  \mathbb{R}^{1\times n}$,  para $i=1, \ldots, s$,  e  $b=(b_1 \ldots b_s)^T\in \mathbb{R}^{s\times 1}$. Então  reescrevemos  o problema~\eqref{eq;pq} da seguinte forma:

    \begin{equation} \label{eq:pqr}
        \begin{array}{cr}
            \text{Minimizar}_x & \frac{1}{2}x^TQx \\
            \text{tal que} ~   & A_1x=b_1         \\
                               & \vdots           \\
                               & A_sx=b_s
        \end{array}
    \end{equation}
    Deste modo, o    Lagrangeano $L: \mathbb{R}^{n\times 1}\times \mathbb{R}^{s\times 1}\to \mathbb{R}$ associado ao problema quadrático ~\eqref{eq:pqr}  é dado   por
    \begin{equation} \label{eq:lagra}
        L(x, \lambda)= \frac{1}{2}x^TQx-\sum_{i=1}^s \lambda_i(A_ix-b_i),
    \end{equation}
    Assim, usando \eqref{eq:lagra} as condições de otimalidade do problema~\eqref{eq:pqr} são dadas por  $\nabla L(x, \lambda)=0$. Como $Q \in \mathbb{R}^{n\times n}$ é uma matriz simétrica, usando o  Exemplo~\ref{ex:fq} e o Teorema~\ref{th:rd},    a da última igualdade é escrita equivalentemente da seguinte forma:
    \begin{align*}
        Qx-\sum_{i=1}^s \lambda_iA_i & =0     \\
        A_1x-b_1                     & =0     \\
                                     & \vdots \\
        A_sx-b_s                     & =0
    \end{align*}
    Estas igualdades  podem ser reescrita    na forma $Qx-A^T\lambda=0$ e $Ax-b=0$,  ou ainda
    \begin{align}
        Qx & =A^T\lambda \label{eq:fiml} \\
        Ax & =b. \label{eq:siml}
    \end{align}
    Resta agora encontrar a solução das equações \eqref{eq:fiml} e \eqref{eq:siml}.  Para isso, primeiro note que sendo $Q \in \mathbb{R}^{m\times m}$  positiva definida  ela é inversível. Assim, de \eqref{eq:fiml} obtemos $x=Q^{-1}A^T\lambda$, que substituindo em \eqref{eq:siml} resulta  em $AQ^{-1}A^T\lambda=b$.  Como $posto(A)=s<n$,  temos que a matriz  $AQ^{-1}A^T$ é inversível. Desde modo,  $\lambda=(AQ^{-1}A^T)^{-1}b$, que  substituindo em $x=Q^{-1}A^T\lambda$  nós dá a solução,  que é  a igualdade deseja.
\end{proof}
\subsection{Média aritmética, desvio padrão e covariância}
Nesta seção vamos fixar algumas anotações e conceitos básico de estatística, a saber,  média  aritmética , desvio padrão e covariância.  Consideremos um conjunto de dados dispostos na forma vetorial $v=(v_1,\ldots, v_m)^T\in  \mathbb{R}^{m\times 1}$. A {\it média aritmética} dos dados $v_1, \ldots, v_m$ e dada por:
\begin{equation} \label{eq:defma}
    \mu_v=\ddfrac{1}{m}\sum_{i=1}^mv_i=\frac{v_1+\ldots+v_m}{m}.
\end{equation}
É fácil mostrar que a média aritmética é linear como função do vetor dados, isto é,
\begin{equation} \label{eq:mal}
    \mu_{v+u}=\mu_v+\mu_u, \qquad \quad \mu_{\alpha v}=\alpha\mu_v, \qquad  \forall v, u\in  \mathbb{R}^{m\times 1},~ \forall  \alpha \in \mathbb{R}.
\end{equation}
Em certo sentido,  média de um conjunto de dados é o valor que melhor representa esse conjunto.  Outra forma de caracterizar os conjuntos de dados é através de medidas de dispersão, tais como a variância e o desvio padrão. Seja o conjunto de dados $v=(v_1,\ldots, v_m)^T\in  \mathbb{R}^{m\times 1}$ e o vetor $e=(1,\ldots, 1)\in   \mathbb{R}^{m\times 1}$. Definimos  o {\it desvio padrão} $\sigma_v$  e  a  {\it variância} $\sigma_v^2$ dos dados dispostos na forma vetorial  $v=(v_1,\ldots, v_m)^T\in  \mathbb{R}^{m\times 1}$ como sendo
\begin{equation} \label{eq:defvari}
    \sigma_v:=\sqrt{\ddfrac{1}{m}\sum_{i=1}^m(v_i-\mu_v)^2}, \qquad \qquad \qquad \sigma_v^2:=\ddfrac{1}{m}\sum_{i=1}^m(v_i-\mu_v)^2.
\end{equation}
O desvio padrão  e  a  variância são medidas de dispersão  que indicam como os dados se agrupam ao redor da média. Se os dados se acumulam ao redor da média, então a variância e o desvio padrão são pequenos. Se os dados se encontram muito afastados da média (dispersos), então essas medidas são maiores.  Agora que temos medidas caracterizando conjuntos de dados, vamos definir agora medidas que nos permitem comparar diferentes conjuntos de dados: a covariância e a correlação. Sejam $v=(v_1,  \ldots, v_m)^T\in  \mathbb{R}^{m\times 1}$ e $u=(u_1, \ldots, u_m)^T\in  \mathbb{R}^{m\times 1}$ conjuntos de dados. Seja também o conjunto de dados $vu:=(v_1y_1, \ldots, v_mu_m)^T\in  \mathbb{R}^{m\times 1}$. Definimos a {\it covariância} $Cov(v,u)$ e a  {\it correlação} $ \rho(v,u)$ entre os dados $v$ e $u$ como sendo
\begin{equation} \label{eq:defcovcor}
    Cov(v,u)=\mu_{(v-e\mu_v)(u-e\mu_u)}:=\ddfrac{1}{m} \sum_{i=1}^m(v_i-\mu_v)(u_i-\mu_u), \qquad \qquad \rho(v,u):=\frac{Cov(v,u)}{\sigma_v\sigma_u}.
\end{equation}
Cálculos diretos mostram que as quantidades definidas em \eqref{eq:defvari} e  primeira igualdade em \eqref{eq:defcovcor} satisfazem as seguintes propriedades:
\begin{enumerate}
    \item[(i)] $Cov(v,u)=\mu_{vu}-\mu_{v}\mu_{u}$, para todo $v, u\in  \mathbb{R}^{m\times 1}$;
    \item[(ii)] $\sigma^2_{v+u}=\sigma^2_v+\sigma^2_u+2Cov(v,u)$, para todo $v, u\in  \mathbb{R}^{m\times 1}$;
    \item[(iii)] $\sigma^2_{\alpha v}=\alpha^2\sigma^2_v$,  para todo $v \in  \mathbb{R}^{m\times 1}$ e $\alpha \in \mathbb{R}$.
\end{enumerate}
Terminamos esta seção com  um teorema que nos fornece  uma propriedade muito importante da  correlação, na sua prova usamos  a seguinte notação: Para cada
$$
    v=(v_1,  \ldots, v_m)^T\in  \mathbb{R}^{m\times 1}, \qquad u=(u_1, \ldots, u_m)^T\in  \mathbb{R}^{m\times 1}
$$
conjuntos de dados dispostos  na forma vetorial,  definamos  $vu$ como sendo o seguinte conjunto de dados disposto na forma vetorial
\begin{equation} \label{eq:prodvet}
    vu:=(v_1u_1,\ldots, v_my_m)^T\in  \mathbb{R}^{m\times 1}.
\end{equation}
\begin{theorem} \label{th:corr}
    Para todo   $v, u\in  \mathbb{R}^{m\times 1}$, temos  $ -1 \leq \rho(v,u) \leq 1$.
\end{theorem}
\begin{proof}
    Inicialmente, definamos o operador média  $\mathbb{E}:\mathbb{R}^m \to \mathbb{R}$ por  $\mathbb{E}[x]=\mu_x$. Dados dois  conjuntos de dados $v, u \in  \mathbb{R}^{m\times 1}$,  usando a notação \eqref{eq:prodvet} definamos também a função real $f: \mathbb{R} \to \mathbb{R}$ com sendo
    $$
        f(t)=\mathbb{E}[((x-e\mu_x)+t(y-e\mu_y))^2].
    $$
    Usando  \eqref{eq:prodvet}, após alguns cálculos temos $(x-e\mu_x)^2+2t(x-e\mu_x)(y-e\mu_y)+t^2(y-e\mu_y)^2$. Como o operador média  $\mathbb{E}$ é linear obtemos
    $$
        f(t)=\mathbb{E}[(x-e\mu_x)^2]+2t\mathbb{E}[(x-e\mu_x)(y-e\mu_y)]+t^2\mathbb{E}[(y-e\mu_y)^2].
    $$
    Então, usando segunda igualdade em \eqref{eq:defvari}  e primeira em \eqref{eq:defcovcor},  concluímos que
    $$
        f(t)=\sigma^2_x+2tCov(v,u)+t^2\sigma^2_y,
    $$
    que é um polinômio de segundo grau. Como  $f(t) \geq 0$, para todo $t \in \mathbb{R}$, o discriminante da equação $f(t)=0$ deve ser $\Delta=4Cov(v,u)^2-4\sigma_x^2\sigma_y^2 \leq 0$,   o que implica
    $$
        \frac{Cov(v,u)^2}{\sigma_x^2\sigma_y^2}\leq 1.
    $$
    Portanto,  definição de $\rho(v,u)$ em \eqref{eq:defcovcor}  nós dá  $\rho(v,u)^2 \leq 1$,    que  é equivalente  ao  resultado desejado.
\end{proof}
Para mais detalhes  dos tópicos abordados nesta seção veja por exemplo  \cite{Bussab2010}.

\section{Retorno e risco} \label{se:tmp}
Nesta seção estudamos os conceitos  de retorno e risco de uma carteira e  suas propriedades básicas. Estes conceitos são importantes para o bom entendimento  da teoria moderna do  portfólio ou teoria de mínimo risco que foi introduzida em \cite{Markowitz1952}, a qual será estudada nas próximas seções.
\begin{definition}
    Seja um conjunto de $n$ ativos $A_1, \ldots, A_n$ e seja $a_{ij}$ o preço de fechamento ajustado do ativo $A_j$ no $i$-ésimo período considerado entre $m$ períodos totais. Podemos representar esses dados através da seguinte tabela:
    \begin{table}[H]
        \begin{footnotesize}
            \begin{center}
                \begin{tabular}{lccccc}
                    Período   & $A_1$    & \dots    & $A_j$    & \dots    & $A_n$                               \\
                    \hline
                    $1$º      & $a_{11}$ & \dots    & $a_{1j}$ & \dots    & $a_{1n}$                            \\
                    $2$º      & $a_{21}$ & \dots    & $a_{2j}$ & \dots    & $a_{2n}$ \\{}
                    \vdots    & $\vdots$ & $\vdots$ & $\vdots$ & $\vdots$ & $\vdots$                            \\
                    $i$-ésimo & $a_{i1}$ & \dots    & $a_{ij}$ & \dots    & $a_{in}$                            \\
                    \vdots    & $\vdots$ & $\vdots$ & $\vdots$ & $\vdots$ & $\vdots$                            \\
                    $m$-ésimo & $a_{m1}$ & \dots    & $a_{mj}$ & \dots    & $a_{mn}$                            \\
                    \hline
                \end{tabular}
                \caption{ \footnotesize Preços   de fechamento ajustado.} \label{tab:price}
            \end{center}
        \end{footnotesize}
    \end{table}
\end{definition}

\begin{remark} Os períodos na tabela acima podem ser divididos de diversas formas diferentes a depender do tipo de análise que será feita. Em geral, os dados são coletados diariamente, semanalmente ou mensalmente.
\end{remark}

Um {\it portfólio} composto pelos ativos $A_1, \ldots, A_n$  é definido pela alocação de um vetor de  capital $x=(x_1, \ldots, x_n)^T\in \mathbb{R}^{n\times 1}$ com $x_1+ \cdots + x_n=1$, onde a porcentagem $x_i$ do capital  é alocada no ativo $A_i$. Denotamos de agora em  diante  o {\it portfólio composto pelos ativos $A_1, \ldots, A_n$}  por
$$
    P_x:=(x_1A_1,\ldots, x_nA_n)
$$
onde  $x=(x_1, \ldots, x_n)^T\in \mathbb{R}^{n\times 1}$  sendo vetor de alocação de capital. O {\it retorno} $r_{ij}$  ou a {\it taxa de retorno} de um ativo $A_j$ sobre um determinado período $i$  é a porcentagem de crescimento (ou decrescimento) entre o período considerado e o período anterior, isto é,
$$
    r_{ij}:=\frac{a_{ij}-a_{(i-1)j}}{a_{(i-1)j}}, \qquad 1\leq i \leq m; \quad 1\leq j \leq n.
$$
Definimos  $\mu_j$ como sendo o {\it retorno médio}  do ativo  $A_j$ e   $\sigma_j$ como sendo o seu {\it risco}, isto é, o {\it desvio padrão}   sobre os períodos considerados, ou seja,
\begin{equation} \label{eq:muij}
    \mu_j:=\frac{1}{m}\sum_{i=1}^mr_{ij}, \qquad \qquad  \sigma_j:=\sqrt{\frac{1}{m}\sum_{i=1}^m(r_{ij}-\mu_j)^2}
\end{equation}
A seguir introduzimos o conceito de retorno de um portfólio com respeito a um conjunto de ativos.
\begin{definition} \label{eq:rmdp}
    Seja $P_x:=(x_1A_1,\ldots, x_nA_n)$ o portfólio com os ativos $A_1,\ldots, A_n$ tais que $x_1, \ldots, x_n$ são as frações do capital investido em cada um dos ativos, respectivamente. Seja $r_{ix}$ o retorno do portfólio $P_x$ no $i$-ésimo período, o qual é dado por
    \begin{equation} \label{eq:rix}
        r_{ix }:=\sum_{j=1}^nx_jr_{ij}
    \end{equation}
    Definimos também o retorno médio e desvio padrão do portfólio $P_x$ como sendo:
    \begin{equation} \label{eq:mxsx}
        \mu_x:=\frac{1}{m}\sum_{i=1}^mr_{ix } \qquad \qquad  \sigma_x:=\sqrt{\frac{1}{m}\sum_{i=1}^m(r_{ix}-\mu_x)^2}
    \end{equation}
\end{definition}
Na Tabela~\ref{tab:retur} listamos os dados de retornos, retornos médios e desvios padrões relativos ao portfólio $P_x$:
\begin{table}[H]
    \begin{footnotesize}
        \begin{center}
            \begin{tabular}{lccccc|c}
                Período       & $A_1$      & \dots    & $A_j$      & \dots    & $A_n$      & $P_x$       \\
                \hline
                $1$º          & $r_{11}$   & \dots    & $r_{1j}$   & \dots    & $r_{1n}$   & ${r}_{1 x}$ \\
                $2$º          & $r_{21}$   & \dots    & $r_{2j}$   & \dots    & $r_{2n}$   & ${r}_{2 x}$ \\
                \vdots        & $\vdots$   & $\vdots$ & $\vdots$   & $\vdots$ & $\vdots$   & $\vdots$    \\
                $i$-ésimo     & $r_{i1}$   & \dots    & $r_{ij}$   & \dots    & $r_{in}$   & ${r}_{i x}$ \\
                \vdots        & $\vdots$   & $\vdots$ & $\vdots$   & $\vdots$ & $\vdots$   & $\vdots$    \\
                $m$-ésimo     & $r_{m1}$   & \dots    & $r_{mj}$   & \dots    & $r_{mn}$   & ${r}_{mx}$  \\
                \hline
                Retorno Médio & $\mu_{1}$  & \dots    & $\mu_{j}$  & \dots    & $\mu_{n}$  & ${\mu}_{x}$ \\
                Desvio Padrão & $\sigma_1$ & \dots    & $\sigma_j$ & \dots    & $\sigma_n$ & $\sigma_x$
            \end{tabular}
            \caption{\footnotesize  Tabela dos retornos.} \label{tab:retur}
        \end{center}
    \end{footnotesize}
\end{table}
\noindent
Para referencia futura $M\in \mathbb{R}^{n\times 1}$ é o {\it vetor dos retornos médios dos ativos $A_1,\ldots,A_n$}, i.e.,
\begin{equation} \label{eq:rm}
    M^T:=(\mu_1,  \ldots, \mu_n).
\end{equation}
Definamos também $R_{* j} \in \mathbb{R}^{n\times 1}$ como sendo  o {\it vetor dos retornos dos ativos $A_j$}    como sendo
\begin{equation} \label{eq:rAj}
    R^{\T}_{* j}:=\left(r_{1j}, \ldots, r_{mj} \right).
\end{equation}
Definamos ainda $R_{i* } \in \mathbb{R}^{n\times 1}$ como sendo  o {\it vetor dos retornos dos ativos $A_1,\ldots,A_n$  no $i$-ésimo  período}  como sendo
\begin{equation} \label{eq:rix2}
    R_{i* }^T:=(r_{i1},r_{i2},\ldots,r_{in}).
\end{equation}
No próximo lema usamos $M$ para apresentar  uma fórmula para o  retorno do portfólio $P_x$.
\begin{lemma} \label{le:Mx}
    Sejam  $x=(x_1, \ldots, x_n)^T$  o vetor  das frações do capital investido e $M^T:=(\mu_1, \ldots, \mu_n)$  vetor dos retornos médios associados ao portfólio $P_x=(x_1A_1,  \ldots, x_nA_n)$. Então, o retorno do portfólio $P_x$ será dado por
    $$
        \mu_x=M^Tx.
    $$
\end{lemma}
\begin{proof}
    Combinando \eqref{eq:rix} com primeira igualdade em \eqref{eq:mxsx}   temos
    \begin{equation} \label{eq:maux}
        \mu_x=\frac{1}{m}\sum_{i=1}^m\sum_{j=1}^nx_jr_{ij}=\sum_{j=1}^nx_j\frac{1}{m}\sum_{i=1}^mr_{ij}
    \end{equation}
    Substituindo em  \eqref{eq:maux}  a primeira  igualdade em \eqref{eq:muij}    obtemos  $\mu_x=\sum_{j=1}^nx_j\mu_j=M^Tx$, o que conclui a prova.
\end{proof}
A seguir vamos deduzir uma fórmula para calcular o desvio padrão do portfólio $P_x$. Para tal é conveniente  introduzir a covariância entre os ativos e também a matriz da covariância.
\begin{definition}  \label{def:cov}
    Definimos a covariância entre dois ativos $A_j$ e $A_{\ell}$ como $\sigma_{j\ell}:=Cov(R_{* j}, R_{* \ell})$, onde  $R_{* j}=\left(r_{1j}, \ldots, r_{mj} \right)$ e  $R_{* \ell} =\left(r_{1\ell}, \ldots, r_{m\ell} \right)$  são os  vetores dos retornos dos ativos $A_j$ e $A_\ell$, respectivamente. Assim, usando primeira igualdade em \eqref{eq:defcovcor} temos
    \begin{equation} \label{eq:cov}
        \sigma_{j\ell}:=\frac{1}{m}\sum_{i=1}^m(r_{ij}-\mu_j)(r_{i\ell}-\mu_{\ell}); \qquad 1\leq j,\ell \leq n.
    \end{equation}
    Além disso, usando segunda igualdade em \eqref{eq:muij} e  \eqref{eq:cov} concluímos   que
    $$
        \sigma_{jj}=\sigma_j^2.
    $$
    A matriz da covariância associada aos  ativos $A_1,\ldots,A_n$ é definida por  $V:=(\sigma_{j\ell}) \in \mathbb{R}^{n\times n}$.
\end{definition}
No próximo lemma vamos   expressar a  matriz da covariância  $V:=(\sigma_{j\ell}) \in \mathbb{R}^{n\times n}$ introduzida na Definição~\ref{def:cov} usando  o vetor dos retornos médios \eqref{eq:rm} e  o vetor dos retornos \eqref{eq:rix2}.
\begin{lemma} \label{eq:mce}
    A matriz da covariância $V$ associada os  ativos $A_1,\ldots,A_n$ é dada por
    \begin{equation} \label{eq:covr}
        V=\frac{1}{m}\sum_{i=1}^m(R_{i* }-M)(R_{i* }-M)^T.
    \end{equation}
\end{lemma}
\begin{proof}
    Primeiramente note que $(R_{i* }-M)(R_{i* }-M)^T\in \mathbb{R}^{n\times n}$. Para simplificar a notação, denote por $v^i_{j\ell}$ a  $j\ell$-entrada da matriz $(R_{i* }-M)(R_{i* }-M)^T$ e   considere a matrix
    \begin{equation} \label{eq:covraux}
        \frac{1}{m}\sum_{i=1}^m(R_{i* }-M)(R_{i* }-M)^T\in \mathbb{R}^{n\times n}
    \end{equation}
    Levando en conta que   $v^i_{j\ell}$ a  $j\ell$-entrada da matriz $(R_{i* }-M)(R_{i* }-M)^T$, temos   a  $j\ell$-entrada da matriz \eqref{eq:covraux} é dada por
    \begin{equation} \label{eq:cmof}
        V_{j\ell}=\frac{1}{m}\sum_{i=1}^m v^i_{j\ell}
    \end{equation}
    Por outro lado,  combinando  \eqref{eq:rm} com \eqref{eq:rix2}  temos  $v^i_{j\ell}=(r_{ij}-\mu_j)(r_{i\ell}-\mu_{\ell}).$  Esta  igualdade juntamente com  \eqref{eq:cmof} implica
    $$
        V_{j\ell}= \frac{1}{m}\sum_{i=1}^m(r_{ij}-\mu_j)(r_{i\ell}-\mu_{\ell}),
    $$
    e assim, usando \eqref{eq:cov}, obtemos que $V_{j\ell}=\sigma_{j\ell}$ o que prova a igualdade \eqref{eq:covr}.
\end{proof}
A seguir  vamos expressar o desvio padrão ou risco $\sigma_x$  de um portfólio   $P_x=(x_1A_1,\ldots, x_nA_n)$ em função da matriz da covariância  $V$ associada ao ativos $A_1,\ldots,A_n$  que compõem este portfólio.
\begin{lemma} \label{le:fxmc}
    Seja  $V \in \mathbb{R}^{n\times n}$ a matriz da covariância associada aos  ativos $A_1,\ldots,A_n$ e seja também  $x=(x_1, \ldots, x_n)^T\in \mathbb{R}^{n\times 1}$  um vetor de alocação de capital. Então, o desvio padrão ou risco  $\sigma_x$ do portfólio  $P_x=(x_1A_1,\ldots, x_nA_n)$  é dado por
    \begin{equation} \label{eq:dvp}
        \sigma_x=\sqrt{x^TVx}
    \end{equation}
\end{lemma}
\begin{proof}
    Combinando \eqref{eq:rix} com \eqref{eq:rix2} temos $r_{ix }=R_{i* }^Tx$. Então, usando a definição  do desvio padrão  $\sigma_x$ em \eqref{eq:mxsx} juntamente com o   Lema~\ref{le:Mx} concluímos que
    \begin{equation} \label{eq:ss}
        \sigma_x^2=\frac{1}{m}\sum_{i=1}^m(r_{ix}-\mu_x)^2=\frac{1}{m}\sum_{i=1}^m(R_{i* }^Tx-M^Tx)^2.
    \end{equation}
    Por outro lado, como o produto de matrizes é distributivo e  a transposição  é linear, temos
    \begin{equation} \label{eq:ss1}
        (R_{i* }^Tx-M^Tx)^2=\big((R_{i* }-M)^Tx\big)^2=\big((R_{i* }-M)^Tx\big)\big((R_{i* }-M)^Tx\big).
    \end{equation}
    Agora,  o produto $((R_{i* }-M)^Tx)$ é de uma matriz  de dimensão $1\times 1$. Como toda matriz $1\times 1$ é simétrica  e  a transposição do produto de matrizes inverte sua ordem, temos
    \begin{equation} \label{eq:ss2}
        \big((R_{i* }-M)^Tx\big)\big((R_{i* }-M)^Tx\big)=\big((R_{i* }-M)^Tx\big)^T\big((R_{i* }-M)^Tx\big)=x^T(R_{i* }-M)(R_{i* }-M)^Tx.
    \end{equation}
    Portanto, combinado a igualdade \eqref{eq:ss} com as igualdades \eqref{eq:ss1} e \eqref{eq:ss2} concluímos que
    $$
        \sigma_x^2=\frac{1}{m}\sum_{i=1}^mx^T(R_{i* }-M)(R_{i* }-M)^Tx.
    $$
    Na igualdade anterior,  o vetor $x$ e  o numero $1/m$ são invariantes em relação a $i$. Então  podemos tirar $x$ de dentro do somatório, ao passo que $1/m$ entra, o que nos fornece a seguinte igualadade
    $$
        \sigma_x^2=x^T\Big(\sum_{i=1}^m\frac{1}{m}(R_{i* }-M)(R_{i* }-M)^T\Big)x.
    $$
    Usando  o Lema~\ref{eq:mce} obtemos que o termo entre os parênteses é exatamente a matriz $V$. Portanto, a igualdade anterior se reduz a $\sigma_x^2=x^TVx $, que é equivalente a igualdade desejada.
\end{proof}
Usando o Lema~\ref{le:fxmc}  juntamente com a Definição~\ref{def:cov}, obtemos o próximo resultado.
\begin{corollary}  \label{cr:exfdp}
    Seja  $V \in \mathbb{R}^{n\times n}$ a matriz da covariância associada aos  ativos $A_1,\ldots,A_n$ e seja também  $x=(x_1, \ldots, x_n)^T\in \mathbb{R}^{n\times 1}$  um vetor de alocação de capital. Então, o desvio padrão ou risco  $\sigma_x$ do portfólio  $P_x=(x_1A_1,\ldots, x_nA_n)$  é dado por

    $$
        \sigma_x=\sqrt{\sum_{j=1}^nx_j^2\sigma_j^2+\sum_{j=1, j\neq \ell}^n \sum_{\ell=1, \ell\neq j}^nx_jx_{\ell}\sigma_{j\ell}}=\sqrt{ \sum_{j=1}^n  x_j^2 \sigma_j^2 + \sum_{j=1, j\neq \ell}^n \sum_{\ell=1, \ell\neq j}^n x_j x_\ell \sigma_j \sigma{_\ell} \rho_ {j\ell} }
    $$
    onde $\rho_ {j\ell}$ é a correlação entre os ativos $A_i$ e $A_\ell$.
\end{corollary}
\begin{proof}
    A primeira igualdade   segue explicitando o produto de matrizes no Lema~\ref{le:fxmc}. Para provar a segunda igualdade primeiro note que $\sigma_{j\ell}=Cov(R_{* j}, R_{* \ell})$. Como $\rho_ {j\ell}$ é a correlação entre os ativos $A_i$ e $A_\ell$, seque da segunda igualdade em \eqref{eq:defcovcor} que $ \rho_ {j\ell}= \sigma_{j\ell}/(\sigma_j \sigma{_\ell})$. Então, a segunda igualdade do corolário segue da primeira.
\end{proof}

\begin{remark} \label{obs:cn}
    Uma importante consequência  do  Corolário~\ref{cr:exfdp}  é que,  dados ativos $A_1,A_2, \ldots, A_n$, o risco  do portfólio  $P_x=(x_1A_1, \ldots, x_nA_n)$ depende essencialmente de três coisas:  Da alocação de capital $x$, dos desvios padrões $\sigma_i$ e covariâncias entre os ativos $\sigma_{ij}$,  ou equivalentemente,  da correlação $\rho_ {j\ell}$ entre os ativos.  Assim, algumas formas naturais de reduzir o risco  de um portfólio são as seguinte:
    \begin{enumerate}
        \item[(i)]  Aumentar a fração de capital alocado em ativos com menor risco;
        \item[(ii)] Selecionar ativos que tenham pequeno desvio padrão;
        \item[(iii)] Selecionar ativos que possuam baixa correlação    uns com os outros.
    \end{enumerate}
    Pelo Teorema~\ref{th:corr} a correlação entre dois ativos vária no intervalo  $[-1, 1]$, sendo que o sinal positivo indica que os dois ativos caminham na mesma a direção, isto é, os retornos de ambos os ativos são positivos ou ambos são negativos. Já o sinal negativo indica que os ativos possuem direções opostas, isto é, quanto um ativo está com retorno positivo o outro está com retorno negativo e vice-versa. O valor da correlação indica o quão forte é a relação entre os ativos, sendo 1 a correlação perfeita e 0 a ausência de correlação. Ativos com correlação positiva, adicionam mais retorno na alta, mas também mais risco na baixa, deste modo ativos com correlação negativa diminuem consideravelmente o risco, perdendo um pouco de retorno na alta. O conjunto de ativos selecionado com o objetivo de ter o menor risco deve conter ativos com correlações negativas ou baixas, entretanto encontrar um conjunto de ativos com esta característica não é um tarefa fácil.
\end{remark}
\section{Portfólio de risco mínimo}  \label{sec:PRM}
O objetivo desta  seção é  apresentar uma formula explicita para a fronteira eficiente de conjunto de ativos,  a qual apareceu primeiramente em  \cite{Merton1972}. Como consequência  vamos calcular o portfólio de risco mínimo associado a este  conjunto de ativos. Vale a pena ressaltar que sempre devemos ter  em mente a Observação~\ref{re:ddpmr2} quando utilizamos o desvio padrão como uma medida de risco.

    {\it A fronteira eficiente} de um conjunto de ativos de risco é um ramo de hipérbole  no conjunto das possíveis alocações de ativos, a qual  {\it é caraterizada  como sendo a curva   contendo todos os portfólios satisfazendo a condição que não existe nenhum outro portfólio  com o mesmo risco e retorno maior}. Formalmente, a fronteira eficiente é o traço da curva em  espaço $risco\times retorno$ dada por:
\begin{equation} \label{eq:fef}
    0<r \mapsto \sigma(r):=\sqrt{x(r)^TVx(r)}
\end{equation}
onde  o vetor  $x(r)$ é a alocação associada ao portfólio ótimo $P_{x(r)}$  com retorno fixado $r$, isto é,  $x(r)$  é solução do seguinte {\it problema de otimização parametrizado}:
\begin{subequations}
    \begin{alignat}{2}
        \min_{x}     \qquad   & \frac{1}{2}x^{\T}Vx   \label{eq:optProbref} \\
        \text{tal que } \quad & {M}^{\T}x= {r}, \label{eq:constraintr1ef}   \\
                              & e^{\T}x=1 ,\label{eq:constraintr2ef}        \\
                              & x \geq 0. \label{eq:constraintr3ef}
    \end{alignat}
\end{subequations}
Para apresentar uma fórmula explícita  para $x(r)$, vamos simplificar o problema Problema~\eqref{eq:optProbref}-\eqref{eq:constraintr3ef}, ou seja, vamos  abandonar a última restrição, isto é, a restriaçãp \eqref{eq:constraintr3ef}. Assim, a nova formulação fica da seguinte forma:
\begin{subequations}
    \begin{alignat}{2}
        \min_{x}     \quad  & \frac{1}{2}x^{\T}Vx   \label{eq:optProbnf} \\
        \text{s. t.}  \quad & {M}^{\T}x = {r}, \label{eq:constraintr1nf} \\
                            & e^{\T}x =1 ,\label{eq:constraintr2nf}
    \end{alignat}
\end{subequations}
\begin{remark}
    Esta nova formulação, ou seja  Problema~\eqref{eq:optProbnf}-\eqref{eq:constraintr2nf},   significa que podemos ter frações de alocação negativas e maiores que $1$. Na prática, isso significa que permitimos operações com alavancagem de mercado, isto é, pode-se vender um ativo antes de comprá-lo e pode-se alocar mais capital do que se efetivamente tem em um determinado ativo. Por exemplo, isto pode ser feito fazendo empréstimo ao mercado.
\end{remark}

Para prosseguir  precisamos adicionar   três  hipóteses. Seja $P_x:=(x_1A_1, x_2A_2, \ldots, x_nA_n)$ o portfólio com os ativos $A_1, A_2, \ldots, A_n$ e $V$  matriz da covariância associada a estes  ativos.  Então assumimos as seguintes hipóteses:
\begin{enumerate}
    \item[(h1)] O  portfólio  $P_x$  é composto apenas com ativos arriscados;
    \item [(h2)]A matriz $V$ é definida positiva e, portanto, inversível;
          \item[(h3)]Não existe $\gamma \in \mathbb{R}$ tal que $M=\gamma e$.
\end{enumerate}
\begin{remark}
    As  hipóteses acimas são razoáveis. De fato, o Lema~\ref{le:fxmc} nos dá como consequência que a matriz $V$ é tal que $x^TVx\geq0$, para todo vetor $x \in \mathbb{R}^n$, isto é, é uma matriz semidefinida positiva. Para excluir a possibilidade de ser $x^TVx=0$, para algum vetor não nulo $x \in \mathbb{R}^n$. Basta observar, $x^TVx=0$ implica $\sigma_x=0$, o que significa que teríamos uma   alocação de capital  com risco nulo, o que contraria a primeira hipótese.
\end{remark}
Antes de enunciar o principal resultado desta seção precisamos de algumas constantes auxiliares, as quais estão associadas os dados do  Problema~\eqref{eq:optProbref}-\eqref{eq:constraintr2ef}.
\begin{equation} \label{eq:caux}
    a:= M^TV^{-1}e, \qquad b:=M^TV^{-1}M, \qquad c:=e^TV^{-1}e, \qquad \Delta:=bc-a^2
\end{equation}
Além disso, para simplificar as notações  definimos os seguintes vetores:
\begin{equation} \label{eq:vaux}
    [M\,\, e]:=\begin{bmatrix}
        \mu_1  & 1      \\
        \mu_2  & 1      \\
        \vdots & \vdots \\
        \mu_n  & 1
    \end{bmatrix} \in \mathbb{R}^{n\times 2},  \quad  \qquad  [r\,\, 1]:=(r \,\,1) \in \mathbb{R}^{1\times2}.
\end{equation}
\begin{theorem} \label{th:ef}
    Seja $P_{x(r)}$ um  portfólio com alocação de capital $x(r)$ associado a um retorno esperado $r>0$. Então  $x(r)$ é dada por:
    \begin{align}
        0<r \mapsto x(r) & := V^{-1}[{M}~e]\left([{M}~e]^{\T}V^{-1}[{M}~e]\right)^{-1}  [r~1]^{\T} \label{eq:pfdd1} \\
                         & =\frac{cr -a}{\Delta}V^{-1}{M}-\frac{ar -b}{\Delta}V^{-1}e. \label{eq:pfdd2}
    \end{align}
    Consequentemente,   a fronteira eficiente no espaço $risco\times retorno$  é dada como se seque
    \begin{align}
        0<r\mapsto \sigma(r) & :=\sqrt{[r ~1]\left([{M}~e]^{\T}V^{-1}[{M}~e]\right)^{-1} [r~1]^{\T}} \label{eq:efdd1} \\
                             & =\sqrt{\frac{1}{\Delta}\left(cr^2-2ar+b \right)}. \label{eq:efdd2}
    \end{align}
\end{theorem}
\begin{proof}
    Dado um retorno esperado $r>0$, temos que a  alocação de capital $x(r)$ é a solução do Problema~\eqref{eq:optProbnf}-\eqref{eq:constraintr2nf}.  Primeiramente vamos rescrever este problema com objetivo de aplicar o Teorema~\ref{th:sqp}.  Usando a notação \eqref{eq:vaux}  o  Problema~\eqref{eq:optProbnf}-\eqref{eq:constraintr2nf}  pode ser rescrito da seguinte forma equivalente
    \begin{subequations}
        \begin{alignat}{2}
            \min_{x}     \quad         \frac{1}{2}x^{\T}Vx   \label{eq:optProbref1}          \\
            \text{s. t.} \quad     [{M}~e]^{\T}x = & [r ~1]^{\T}. \label{eq:constraintr1ef1}
        \end{alignat}
    \end{subequations}
    Por hipótese  $M \neq \gamma e$, assim a matriz $[M \,\,e]$ tem posto completo. Então aplicando o Teorema~\ref{th:sqp} com $Q=V$,  $A=[M \,\,e]^T$ e $b=[r\,\,1]^T$ obtemos que a solução do Problema~\eqref{eq:optProbref1}-\eqref{eq:constraintr1ef1} é igual
    \begin{equation} \label{pr:pxr}
        x(r)=V^{-1}[M \,\,e]([M \,\,e]^TV^{-1}[M\,\, e])^{-1}[r \,\,1]^T,
    \end{equation}
    o que prova \eqref{eq:pfdd1}. Agora vamos usar \eqref{eq:caux}  para explicitar $x(r)$ na forma \eqref{eq:pfdd2}. Primeiramente, note que após alguma manipulações algébricas  e usando as igualdades em \eqref{eq:caux},  obtemos as seguintes igualdades:
    $$
        ([M \,\,e]^TV^{-1}[M\,\, e])^{-1}=\begin{bmatrix}MV^{-1}M & MV^{-1}e\\eV^{-1}M & eV^{-1}e\end{bmatrix}^{-1}=\begin{bmatrix}b&a\\a&c\end{bmatrix}^{-1
        }=\frac{1}{\Delta}\begin{bmatrix}c & -a\\-a & b\end{bmatrix}.
    $$
    Assim, combinando a última igualdade com  \eqref{pr:pxr}  e então  usando \eqref{eq:caux} concluímos que
    $$
        x(r)=\frac{1}{\Delta}V^{-1}[M \,\,e]\begin{bmatrix}c & -a\\-a & b\end{bmatrix}[r \,\, 1]^T= \frac{1}{\Delta}V^{-1}[M \,\,e]\begin{bmatrix}cr-a\\b-ar\end{bmatrix}=\frac{cr-a}{\Delta}V^{-1}M-\frac{ar-b}{\Delta}V^{-1}e.
    $$
    Que é a igual a \eqref{eq:pfdd2}. Portanto, a primeira parte do teorema esta provado.  Para provar a segunda parte, vamos inicialmente provar \eqref{eq:efdd1}.  Usando   \eqref{eq:pfdd1}   obtemos
    \begin{align*}
        x(r)^{\T}Vx(r) & =x(r)^{\T} [{M}~e]\left([{M}~e]^{\T}V^{-1}[{M}~e]\right)^{-1}  [r ~1]^{\T}                     \\
                       & =\left( [{M}~e]^{\T}x(r) \right)^{\T}\left([{M}~e]^{\T}V^{-1}[{M}~e]\right)^{-1}  [r ~1]^{\T}.
    \end{align*}
    Assim, usando  \eqref{eq:constraintr1ef1}, a igualdade  \eqref{eq:efdd1} segue. Resta provar \eqref{eq:efdd2}. Para isso, primeiro usamos  \eqref{eq:fef} juntamente com \eqref{eq:pfdd2}  para obter
    \begin{equation} \label{eq:pmfe}
        \sigma(r)=\sqrt{\left(\frac{cr-a}{\Delta}V^{-1}M-\frac{ar-b}{\Delta}V^{-1}e\right)^TV\left(\frac{cr-a}{\Delta}
            V^{-1}M-\frac{ar-b}{\Delta}V^{-1}e\right)}.
    \end{equation}
    Como $V$ é simétrica, $V^{-1}$ também é simétrica. Deste modo, algumas manipulações algébricas mostram que o radicando de \eqref{eq:pmfe} é igual a
    \begin{equation*}
        \left(\frac{cr-a}{\Delta}\right)^2M^TV^{-1}M-\\\frac{(cr-a)(ar-b)}{\Delta^2}M^TV^{-1}e-\frac{(cr-a)(ar-b)}{\Delta^2}e^TV^{-1}M+\left(\frac{ar-b}{\Delta}\right)^2e^TV^{-1}e.
    \end{equation*}
    Então combinando \eqref{eq:pmfe} com a última igualdade  e  usando \eqref{eq:caux}, obtemos após alguns cálculos que
    $$
        \sigma(r)=\sqrt{\frac{1}{\Delta^2}\left((cr-a)^2b-2(cr-a)(ar-b)a+(ar-b)^2c\right)}=\sqrt{\frac{1}{\Delta}(cr^2-2ar+b)},
    $$
    que é a igualdade desejada, o que conclui a prova.
\end{proof}
No próximo   corolário vamos  explicitar  o {\it portfólio ótimo} ou seja  o {\it portfólio de risco mínimo}, o qual vamos denotar por $P_{x_{min}}$.
\begin{corollary}  \label{cr:prm}
    Seja  $x_{min}$  a alocação de capital associada ao portfólio ótimo $P_{x_{min}}$, com retorno $r_{min}$ e menor risco possível  $\sigma_{min}$. Então
    \begin{equation} \label{eq:prmin}
        x_{min}=\frac{1}{c}V^{-1}e, \qquad r_{min}=\frac{a}{c}, \qquad \sigma_{min}=\frac{1}{\sqrt{c}}.
    \end{equation}
\end{corollary}
\begin{proof}
    Primeiro note que expressão em \eqref{eq:efdd2} implica que   a função risco $r\mapsto \sigma(r)$   é diferenciável. Então seu o mínimo global $\sigma_{min}$ é atingido em $r_{min}$, o qual é obtido resolvendo a equação  $\sigma'(r_{min})=0$, isto é,
    $$
        \sigma'(r_{min})=\frac{1}{2\sqrt{(1/\Delta)(cr_{min}^2-2ar_{min}+b)}}\frac{2cr_{min}-2a}{\Delta}=0.
    $$
    Assim, concluímos  que  $2cr_{min}-2a=0$, o que implica $ r_{min}=a/c$ e a segunda igualdade em \eqref{eq:prmin} está provada. Agora, vamos provar a primeira igualdade em \eqref{eq:prmin}. Como  por definição  $x_{min}:=x(r_{min})$,  usando   \eqref{eq:pfdd2}, $ r_{min}=a/c$ e \eqref{eq:caux}  concluímos
    $$
        x_{min}=\frac{c(a/c)-a}{\Delta}V^{-1}M-\frac{a(a/c)-b}{\Delta}V^{-1}e=-\frac{a^2-bc}{\Delta c}V^{-1}e=\frac{1}{c}V^{-1}e,
    $$
    o que prova a primeira igualdade em \eqref{eq:prmin}.  Finalmente, vamos provar a terceira igualdade. Como $\sigma_{min}=\sigma(r_{min})$, usando \eqref{eq:efdd2}, $ r_{min}=a/c$ e  \eqref{eq:caux}   obtemos
    $$
        \sigma_{min}=\sqrt{\frac{1}{\Delta}(c(a/c)^2-2a(a/c)+b)}=\sqrt{\frac{1}{\Delta}\frac{a^2-2a^2+bc}{c}}=\frac{1}{\sqrt{c}},
    $$
    que é a igualdade desejada, e assim a prova do corolário  está concluída.
\end{proof}
Finalmente, encerramos esta seção mostrando que parametrização da fronteira eficiente dada por \eqref{eq:efdd2} é um ramo de hipérbole em $\mathbb{R}^2$ nas variáveis $\sigma$ e $r$. Sabemos que $V$ é uma matriz definida positiva, assim  $V^{-1}$ também é definida positiva. Deste modo, segue de \eqref{eq:caux}  que  $b>0$ e $c>0$.  A terceira   hipótese acima implica que  não existe $\gamma \in \mathbb{R}$ tal que $M=\gamma e$,  assim o  vetor  $u:=aM-be \in \mathbb{R}^n \neq 0$. Deste modo, usando que  $V^{-1}$ também é definida positiva,  temos
$$
    0<(aM-be)^TV^{-1}(aM-be)=a^2M^TV^{-1}M-2abM^TV^{-1}e+b^2e^TV^{-1}e.
$$
Então usando \eqref{eq:caux} obtemos  $a^2b-2a^2b+b^2c>0$, ou equivalentemente, $b(bc-a^2)=b\Delta>0$. Por outro lado, \eqref{eq:efdd2} implica que $\Delta \sigma ^2=cr^2-2ar+b$. Reescrevendo esta igualdade na forma de  uma  equação geral de uma cônica  em $\mathbb{R}^2$ nas variáveis $\sigma$ e $r$  obtemos
\begin{equation}  \label{eq:hiperbole}
    \Delta \sigma^2+0\sigma r-cr^2+0\sigma+2ar-b=0.
\end{equation}
Como o discriminante desta equação é $0^2-4\Delta(-c)=4c\Delta>0$,  ela de fato  representa uma  hipérbole $\mathbb{R}^2$ nas variáveis $\sigma$ e $r$, veja \cite[Teorema 11.4.7]{Boldrini1980}. O gráfico da hipérbole \eqref{eq:hiperbole} ou fronteira eficiente  é representado na Figura~\ref{fig:mr}. Nesta figura   o  ponto $(\sigma_{min},r_{min})$   corresponde ao seu vértice.
\begin{figure}[H]
    \centering
    \includegraphics[width=0.4\linewidth]{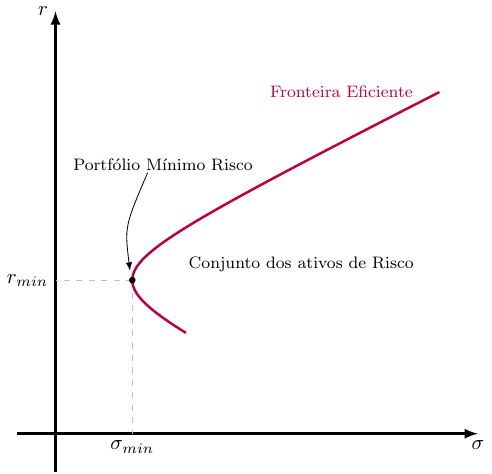}\\
    \caption{\footnotesize  Portfólio de risco mínimo.}
    \label{fig:mr}
\end{figure}
\section{Portfólio eficiente  com  um ativo livre de risco} \label{sec:pealr}

Nesta seção nosso o objetivo é apresentar uma fórmula explícita  para o {\it portfólio eficiente} no qual um dos  ativos que compõe  este portfólio  é livre de risco. Em outros palavras, vamos determinar um  portfólio  que maximiza a taxa de retorno para um determinado nível de risco fixado, com a inclusão de um ativo com risco zero ou o menor risco possível existente no marcado em consideração. Vamos primeiro introduzir algumas notações úteis, a fim de enunciar uma nova reformulação para o Problema \eqref{eq:optProbnf}-\eqref{eq:constraintr2nf} que agora irá  considerar um ativo livre de risco.

Considere $A_1, \ldots, A_n,$ ativos com risco e $A_f$ um ativo  livre de risco. Seja  $x_1, \ldots, x_n, x_{f}$   as frações dos recursos do investidor aplicadas nos ativos  $A_1, \ldots, A_n, A_{f}$, respectivamente. Denotamos por
\begin{equation} \label{eq:muo}
    x:=(x_{1}, \ldots, x_{n})\in {\mathbb R}^{n}, \qquad {\hat x}:=(x, x_{f})\in {\mathbb R}^{n+1}={\mathbb R}^{n}\times {\mathbb R}
\end{equation}
os {\it vetores dos pesos} associados ao portfólio $P_{x}$ e $P_{\hat x}$, respectivamente. Seja ${\mu}_j$ o retorno médio do ativo $A_j$, para $i=1, \ldots, n$ e $r_f$ o retorno do ativo $A_f$, respectivamente. Assim, o retorno do portifólio  $P_{\hat x}$ é dado por
\begin{equation} \label{eq:rfr}
    {\hat r}_{\hat x}={M}^{T}x + r_fx_f
\end{equation}
onde ${M}:=\left({\mu}_1, \ldots, {\mu}_n\right)$ é o vetor dos retornos médios do ativos  $A_1, \ldots, A_n$. A prova de \eqref{eq:rfr} segue a mesma idéia da prova do   Lema~\ref{le:Mx}  e será omitida.  Seja $\sigma_j$ o desvio padrão do ativo $A_j$, para $i=1, \ldots, n$ e $\sigma_f$ o desvio padrão do ativo $A_f$, respectivamente. Como $A_{f}$ é um ativo livre de risco  e não correlacionado aos ativos $A_1, \ldots, A_n$, a matriz de covariância associada aos ativos $A_1, \ldots, A_j, \ldots, A_n, A_{f} $ tem uma linha e uma coluna de zeros que correspondem ao ativo $A_{f}$. Assim, os desvios padrão $\sigma{_{\hat x}}$ do portfólio  $P_{\hat x}$ é igual ao
desvios padrão $\sigma{_{x}}$  do portfólio  $P_{x}$, ou seja, vale a seguinte iguladade
\begin{equation} \label{eq:sdfr3}
    \sigma_{\hat x}= \sigma{_x},
\end{equation}
veja Definição~\ref{def:cov} e  Lemma~\ref{le:fxmc} para mais detalhes. Portanto, usando \eqref{eq:rfr} e \eqref{eq:sdfr3}, surge uma nova reformulação para o Problema \eqref{eq:optProbnf}-\eqref{eq:constraintr2nf} que considera um ativo livre de risco. De fato, para cada retorno fixo $r >0$, o vetor peso ${\hat x}(r)$, associado ao portfólio ótimo $P_{\hat x(r)}$, é a solução do seguinte problema de otimização parametrizado

\begin{subequations}
    \begin{alignat}{2}
        \min_{x}     \quad  & \frac{1}{2}x^{\T}Vx   \label{eq:optProbrefora}      \\
        \text{s. t.}  \quad & {M}^{\T}x + r_fx_f = {r} \label{eq:constraintr1ora} \\
                            & e^{\T}x+ x_f  =1 \label{eq:constraintr2ora}
    \end{alignat}
\end{subequations}
A partir da equação \eqref{eq:constraintr2ora}, segue que $x_f = 1 - e^{\T}x$, o que, substituindo \eqref{eq:constraintr1ora}, resulta em uma nova formulação para o Problema~\eqref{eq:optProbrefora} -- \eqref{eq:constraintr2ora}, isto é, no seguinte problema de otimização parametrizado:
\begin{subequations}
    \begin{alignat}{2}
        \min_{x}     \quad      & \frac{1}{2}x^{\T}Vx   \label{eq:optProbrmb}                      \\
        \text{s. t.}      \quad & \left({M}- r_fe\right)^{\T}x  = {r}- r_f \label{eq:constraintmp}
    \end{alignat}
\end{subequations}
Na próxima proposição apresentamos a solução do Problema~\eqref{eq:optProbrmb} - \eqref{eq:constraintmp}, assim como a fronteira eficiente associada.
\begin{theorem} \label{th:efmp}
    Seja $P_{{\hat x}(r)}$ um  portfólio com alocação de capital ${\hat x}(r)$ associado a um retorno esperado $r>0$. Então   $0<r \mapsto {\hat x}(r):=({x}(r), x_f(r))$ é dada por:
    \begin{equation} \label{eq:pfddmp}
        {x}(r):= \frac{r -r_f}{\left({M}- r_fe\right)^{\T}V^{-1}\left({M}- r_fe\right)}V^{-1}\left({M}- r_fe\right), \qquad \qquad x_f(r):=1-e^{\T}x(r).
    \end{equation}
    Consequentemente, a fronteira eficiente associada ao portfólio $P_{{\hat x}(r)}$ são dois segmentos de reta no espaço risco-retorno parametrizada por $r $ como segue
    \begin{equation} \label{eq:efddmp}
        0<r\mapsto {\sigma}(r):=\frac{\left|r -r_f\right|}{\sqrt{\left({M}- r_fe\right)^{\T}V^{-1}\left({M}- r_fe\right)}}
    \end{equation}
\end{theorem}
\begin{proof}
    Segue-se de (h2) que $({M}- r_fe )^{\T}$ tem posto de linha completo. A hipótese (h2) implica que $V$ é invertível. Assim, \eqref{eq:pfddmp} segue aplicando Teorema~\ref{th:sqp} com $Q=V$, $A= ({M}- r_fe )^{\T}$ e $b={r }-r_f$. Prosseguimos para provar \eqref{eq:efddmp}. Segue de \eqref{eq:pfddmp} que
    $$
        x(r)^{\T}Vx(r)=\frac{\left(r -r_f\right)^2}{\left({M}- r_fe\right)^{\T}V^{- 1}\left({M}- r_fe\right)}
    $$
    Considerando que $V^{-1}$ é  positiva definida, a combinação de \eqref{eq:fef} e \eqref{eq:sdfr3} resulta em \eqref{eq:efddmp}.
\end{proof}
\begin{corollary} \label{cr:opt}
    Seja   $r^*>0$ tal que $x_f(r^*)=0$, ou seja, todos os recursos do investidor estão alocados em ativos de risco.  Então
    \begin{equation} \label{eq:eqpi}
        r^*=  \frac{{M}^{\T}V^{-1}\left({M}- r_fe\right)}{e^{\T}V^{-1}\left({M}- r_fe\right)}, \qquad \qquad {x}(r^*)= \frac{1}{e^{\T}V^{-1}\left({M}- r_fe\right)}V^{-1}\left({M}- r_fe\right)
    \end{equation}
\end{corollary}
\begin{proof}
    Se $x_f(r^*)=0$, então segue da segunda igualdade em \eqref{eq:pfddmp} que $e^{\T}x(r^*)=1$. Agora, usando a primeira desigualdade em \eqref{eq:pfddmp} temos
    $$
        1= \frac{r^* -r_f}{\left({M}- r_fe\right)^{\T}V^{-1}\left({M}- r_fe\right)} e^{\ T}V^{-1}\left({M}- r_fe\right)
    $$
    Após algumas manipulações algébricas na última desigualdade obtemos a primeira igualdade em \eqref{eq:eqpi}. Usando novamente as primeiras desigualdades em \eqref{eq:pfddmp} e \eqref{eq:eqpi}, obtemos
    \begin{align*}
        {x}(r^*) & = \frac{\displaystyle \frac{{M}^{\T}V^{-1}\left({M}- r_fe\right)}{e^{\T} V^{-1}\left({M}- r_fe\right)} -r_f}{\left({M}- r_fe\right)^{\T}V^{-1}\left({M} - r_fe\right)}V^{-1}\left({M}- r_fe\right)               \\
                 & = \frac{\displaystyle \frac{\left({M}- r_fe\right)^{\T}V^ {-1}\left({M}- r_fe\right)}{e^{\T}V^{-1}\left({M}- r_fe\right)} }{\left({M}- r_fe\right)^{\T}V^{-1}\left({M}- r_fe\right)}V^{-1}\left({M}- r_fe\right)
    \end{align*}
    que é equivalente à segunda igualdade em \eqref{eq:eqpi}.
\end{proof}
\subsection{Índice sharpe e linha de alocação de capital}  \label{sec:CAL}
O {\it raio sharpe} ou {\it índice sharpe} dos retornos do portfólio  mede a taxa de retorno por risco, caracterizando o quão bem o retorno de um ativo compensa o investidor pelo risco assumido. O índice sharpe é  formalmente definido  como se segue
\begin{equation} \label{eq:siap}
    S_x=\frac{R_x-r_f}{\sigma_x}
\end{equation}
onde $r_f$ é a taxa livre de risco, $R_x$ é o retorno do portfólio, $\sigma_x$ é o desvio padrão do portfólio  $P_x$ e $x$ é o vetor de peso associado ao portfólio  $P_x$, ver \cite{Sharpe1994}. O índice de sharpe está intimamente relacionado com a {\it linha de alocação de capital}, pois sua inclinação é o índice de sharpe do portfólio  eficiente. De fato, a linha de alocação de capital é a linha tangente a fronteira eficiente passando pelo  ponto do ativo livre risco $r_f$, veja Figura~\ref{fig:CapLineAloc}. Do ponto de vista do investidor, o ponto de tangência pode ser visto como o portfólio ideal. A definição analítica da linha de alocação de capital é a seguinte
\begin{equation} \label{eq:cal}
    R_x:=r_f+ \frac{r_*-r_f}{\sigma_*}\sigma_x
\end{equation}
onde $r_{*}$ é o retorno e $\sigma_{*}$ é o desvio padrão do portfólio  eficiente $P_{x_{*}}$, onde $x_{*}$ é o vetor peso associado, ver Figura~\ref{fig:CapLineAloc}.
\begin{figure}[H]
    \centering
    \includegraphics[width=0.4\linewidth]{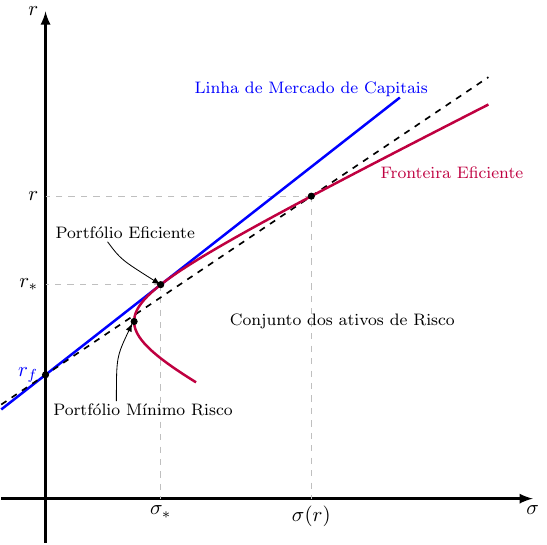}\\
    \caption{Linha de alocação de capital}
    \label{fig:CapLineAloc}
\end{figure}
\begin{remark}
    Segue-se de \eqref{eq:cal} que todos os portfólios sobre  linha de alocação de capital têm o mesmo índice de sharpe que o portfólio eficiente $P_{x_{*}}$, ou seja,
    $$
        \frac{ R_x-r_f}{\sigma_x}= \frac{r_*-r_f}{\sigma_*}
    $$
    Por outro lado, a inclinação da linha de alocação de capital é a que tem a maior inclinação comparado com  qualquer linha que une o ativo livre de risco $r_f$ a um ponto da fronteira eficiente, ou seja, o índice de sharpe de todos os portfólios  na linha de alocação de capital é maior ou igual a qualquer índice de sharpe dos portfólios na fronteira eficiente, ou seja,
    $$
        \frac{ r-r_f}{\sigma(r)}\leq \frac{r_*-r_f}{\sigma_*},
    $$
    veja Figure~\ref{fig:CapLineAloc}.
\end{remark}
Na seguinte proposição, calculamos o portfólio eficiente $P_{{x}(r^*)}$, veja a Figura~\ref{fig:CapLineAloc}. Em particular, recuperamos os mesmos resultados do Corolário~\ref{cr:opt}.
\begin{proposition}\label{eq:ep}
    O portfólio eficiente $P_{{x}(r^*)}$ com retorno $r_{*}$,  vetor de alocação de capital ${x}(r^*)$  e risco $\sigma_{*}$ são dados, respectivamente, por:
    \begin{equation} \label{eq:opttau}
        r^*=  \frac{{M}^{\T}V^{-1}\left({M}- r_fe\right)}{e^{\T}V^{-1}\left({M}- r_fe\right)}, \qquad \qquad {x}(r^*)= \frac{1}{e^{\T}V^{-1}\left({M}- r_fe\right)}V^{-1}\left({M}- r_fe\right)
    \end{equation}
    \begin{equation} \label{eq:optrisk}
        \sigma_*=  \frac{\sqrt{\left({M}- r_fe\right)^{\T}V^{-1}\left({M}- r_fe\right)}}{e^{\T}V^{-1}\left({M}- r_fe\right)}
    \end{equation}
    see   Figure~\ref{fig:CapLineAloc}.
\end{proposition}
\begin{proof}
    Seja $ 0<r\mapsto \sigma(r)$ a fronteira eficiente no espaço $risco\times retorno$ a qual é dada no  Teorema~\ref{th:ef}. Então,  no ponto de tangência na fronteira eficiente temos
    \begin{equation} \label{eq:tc}
        \sigma'(r_*)=\frac{\sigma_*}{r_*-r_f}
    \end{equation}
    Nossa tarefa é resolver a equação \eqref{eq:tc}. Primeiro, usando \eqref{eq:efdd1} obtemos $\sigma'(r_*)=(cr_*-a)/(\Delta \sigma_* )$. Assim, segue de \eqref{eq:tc} que
    \begin{equation} \label{eq:tcd}
        \sigma^2_*= \frac{(cr_*-a)(r_*-r_f)}{\Delta}
    \end{equation}
    Usando novamente \eqref{eq:efdd1} temos $ \sigma^2_*=(cr_*^2-2ar_*+b)/\Delta$, que combinado com \eqref{eq:tcd}  nos fornece
    $$
        (cr_*-a)(r_*-r_f)=cr_*^2-2ar_*+b.
    $$
    Explicitando $r_*$  na  igualdade obtemos  que
    \begin{equation} \label{eq:epe1}
        r_*=\frac{b-ar_f}{a-cr_f}
    \end{equation}
    Usando \eqref{eq:caux} obtemos a segunda igualdade em \eqref{eq:opttau}. Para provar a primeira igualdade em \eqref{eq:opttau}, primeiro usamos \eqref{eq:pfdd1} para obter
    \begin{equation} \label{eq:epe}
        x_*=\frac{cr_* -a}{\Delta}V^{-1}{M}-\frac{ar_ * -b}{\Delta}V^{-1}e.
    \end{equation}
    Assim, combinando \eqref{eq:epe1} e \eqref{eq:epe} com \eqref{eq:caux} obtemos a primeira igualdade em \eqref{eq:opttau}. Assim, devido a
    $$
        \sigma_*=\sqrt{x_*^{\T}Vx_*},
    $$
    a última desigualdade em \eqref{eq:opttau} segue usando a primeira, o que conclui a prova.
\end{proof}

\subsubsection{Modelo de precificação de ativos de capital}  \label{sec:CAPM}
O {\it modelo de precificação de ativos de capital} é um conceito teórico para precificação de um ativo  ou portfólio individual. Em particular, pode ajudar os investidores a entender a relação entre o retorno esperado e o risco à medida que tomam melhores decisões sobre a adição de ativos a um portfólio. No próximo corolário apresentamos uma formulação analítica para o modelo de precificação de ativos de capital com respeito ao portfólio eficiente.
\begin{corollary}
    Seja $P_{x_{*}}$ um portfólio eficiente com retorno $r_{*}$ e risco $\sigma_{*}$. Então,
    \begin{equation} \label{eq:capv}
        {M}- r_fe= \frac{1}{\sigma^2_* } \left(r_*- r_f\right)Vx_*
    \end{equation}
    Como consequência, o retorno médio ${\mu}_j$ do ativo $A_j$ é dado por
    \begin{equation} \label{eq:capvc}
        {\mu}_j= r_f + \beta_j \left(r_*- r_f\right), \qquad \quad \beta_j = \frac{V^{\T}_jx_*}{\sigma^2_*}=cov( {\mu}_j, {M}), \qquad j=1, \ldots, n,
    \end{equation}
    onde $V_j:=(\sigma_{j1}, \ldots \sigma_{jn})^{\T} \in \mathbb{R}^{n\times 1}$  é a $j$-ésima linha da matriz de covariância  $V$.
\end{corollary}
\begin{proof}
    Segue da Proposição~\ref{eq:ep} que
    \begin{equation} \label{eq:efin}
        Vx_*= \frac{1}{e^{\T}V^{-1}\left({M}- r_fe\right)} \left({M}- r_fe\right)
    \end{equation}
    Como $r_*={M}^{\T}x_*$ e $e^{\T}x_*=1$, a última igualdade implica que
    \begin{equation} \label{eq:efop}
        x_*^{\T}Vx_*= \frac{1}{e^{\T}V^{-1}\left({M}- r_fe\right)} \left({M}^{\T }x_*- r_fe^{\T}x_*\right)= \frac{1}{e^{\T}V^{-1}\left({M}- r_fe\right)} \left(r_ *- r_f\right)
    \end{equation}
    Assim, devido a $\sigma^2_{*}=x_*^{\T}Vx_*$, concluímos que $e^{\T}V^{-1}\left({M}- r_fe\right ) = \left(r_*- r_f\right)/ \sigma^2_* $. Assim, substituindo esta igualdade em \eqref{eq:efin} resulta
    $$
        Vx_*= \frac{\sigma^2_*}{\left(r_*- r_f\right)} \left({M}- r_fe\right)
    $$
    que é equivalente a \eqref{eq:capv}. Como $M^{\T}:=(\mu_1,  \ldots, \mu_n)$ e $e^{\T}=(1, \ldots, 1)$.  As igualdades em \eqref{eq:capvc} são uma consequência imediata de \eqref{eq:capv}.
\end{proof}
\subsection{Linha do mercado de capitais}  \label{sec:SR}
Foi mostrado na seção anterior que a carteira ideal ocorre no ponto onde a linha de alocação de capital é tangente à fronteira eficiente. Nesse sentido, a linha de mercado de capitais é um caso especial de linha de alocação de capital baseada no modelo de precificação de ativos de capital. Na verdade, a linha de mercado de capitais é uma linha de alocação de capital que representa uma combinação do ativo livre de risco e do portfólio  do mercado. Mais precisamente, a linha do mercado de capitais é a linha tangente que passa pelo  ponto do ativo livre de  risco $r_f$ até a fronteira eficiente dos ativos com risco, veja a Figura~\ref{fig.CapLineAloc}. A linha do mercado de capitais é definida analiticamente por
\begin{equation} \label{eq:cml}
    R_p:=r_f+ \frac{r_T-r_f}{\sigma_T}\sigma_p
\end{equation}
onde $r_T=$ retorno do mercado e $\sigma_T=$ desvio padrão dos retornos do mercado.
\begin{figure}[H]
    \centering
    \includegraphics[width=0.4\linewidth]{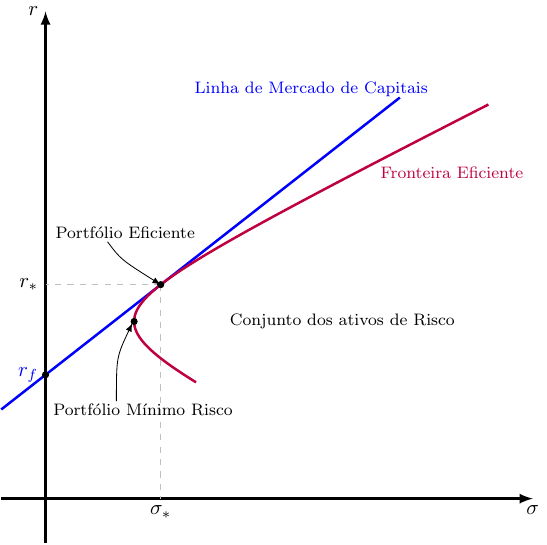}\\
    \caption{Capital market line.}
    \label{fig.CapLineAloc}
\end{figure}

\begin{remark}
    Segue-se de \eqref{eq:cml} que todas as carteiras na linha do mercado de capitais têm o mesmo índice Sharpe  do portfólio  do mercado, ou seja,
    $$
        \frac{ R_p-r_f}{\sigma_p}= \frac{r_T-r_f}{\sigma_T}
    $$
    De fato, a inclinação da linha do mercado de capitais é o índice de Sharpe do portfólio  do mercado.
\end{remark}
\section{Aplicação} \label{Sec:Aplicacao}

Nesta seção, apresentaremos uma estratégia bastante simples utilizando os resultados das Seções~\ref{sec:PRM} e \ref{sec:pealr}, mais especificamente, faremos uso do Corolário~\ref{cr:prm}, do Teorema~\ref{th:efmp} e do Corolário~\ref{cr:opt}. Para isso, vamos construir um {\it portfólio contendo apenas dois ativos de risco} e consideraremos estratégias com e sem um ativo livre de risco. Para implementar essas estratégias, desenvolvemos um programa em Python e utilizamos o ambiente do Google Colab, que através de um Notebook Jupyter, facilita a conexão entre o nosso programa e a planilha que contém os dados do Google Finance.

\subsubsection*{Seleção dos ativos} \label{Sec:SelAtivos}
Visando sucesso da estratégia, de acordo com as Observações~\ref{re:ddpmr2} e \ref{obs:cn}, é conveniente selecionar  os dois ativos  com   as seguintes características:
\begin{enumerate}
    \item[(i)]  Os ativos devem possuir  uma distribuição normal dos  retornos no período em consideração;
    \item[(ii)] Os ativos devem possuir baixa correlação (de preferência correlação negativa) um com o outro na maior parte do período em consideração.
\end{enumerate}
Selecionamos os ativos    IVVB11 e  BOVA11,  com  dados históricos entre \textcolor{blue}{$01/2017$ e $12/2023$}, segue na Figura~\ref{fig:raBOVA11IVVB11} o comportamento destes ativos no período considerado.
\begin{figure}[H]
    \centering
    \includegraphics{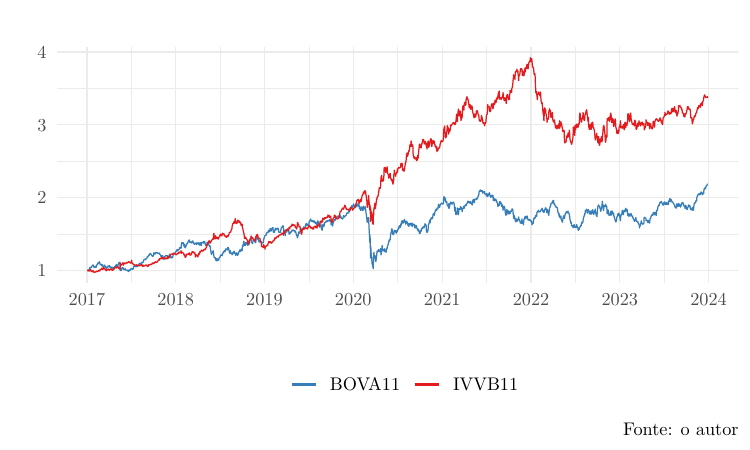}\\
    \caption{\footnotesize  Retorno acumulado.}
    \label{fig:raBOVA11IVVB11}
\end{figure}
A seguir  vamos apresentar os motivos da escolha destes dois ativo.  Estes ativos como é mostrado na Figuras~\ref{DistRetornos} possuem distribuição dos  retornos neste  período aproximadamente  normal  e   assim podemos assumir que em certo sentido   ambos satisfazem o item~(i).
\begin{figure}[H]
    \centering
    \subfloat[Distribuição retornos BOVA11.]{\label{a}\includegraphics[width=0.5\linewidth]{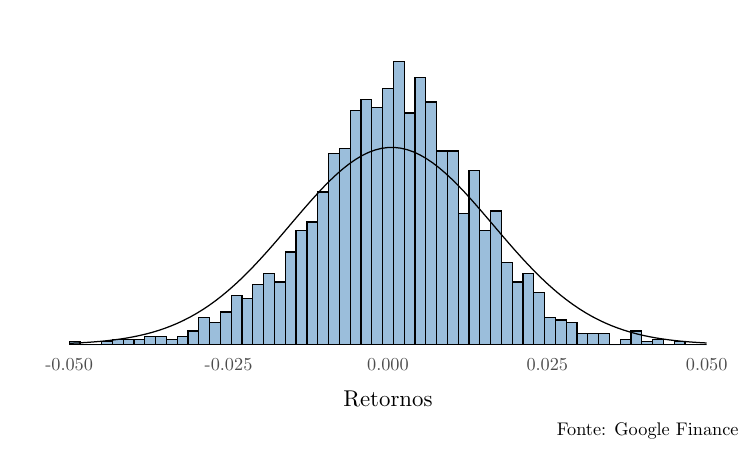}}
    \subfloat[Distribuição retornos IVVB11.]{\label{b}\includegraphics[width=0.5\linewidth]{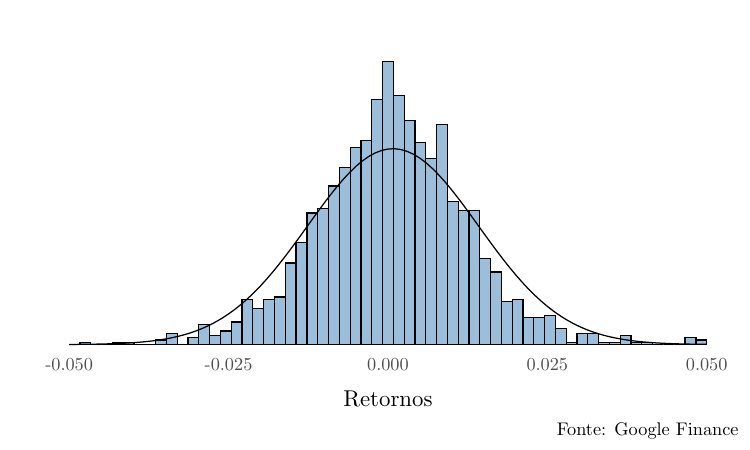}}
    \caption{\footnotesize  Distribuição dos retornos do BOVA11 e IVVB11 de $01/2017$ à $12/2023$. }
    \label{DistRetornos}
\end{figure}
Historicamente, IVVB11 e BOVA11 têm, na maior parte do tempo, baixa ou negativa correlação entre si. No entanto, podem ocorrer períodos de correlação positiva. Veja na Tabela~\ref{Tab:Corr} as correlações anuais do período considerado.
\begin{table}[H]
    \begin{center}
        \begin{tabular}{c|rrrrrrr}

                    & \text{2017} & \text{2018} & \text{2019} & \text{2020} & \text{2021} & \text{2022} & \text{2023} \\
            \hline
            $ \rho$ & $0.58$      & $-0.49$     & $0.78$      & $0.28$      & $-0.74$
                    & $-0.09$     & $0.82$                                                                            \\
        \end{tabular}
        \caption{ \footnotesize Correlação entre IVVB11 e  BOVA11 no período de  $01/2017$ e $12/2023$.}
        \label{Tab:Corr}
    \end{center}
\end{table}
Na Figura~\ref{Correlacao}  está representado graficamente os dados da Tabela~\ref{Tab:Corr}, ou seja,  a   correlação entre  BOVA11 e IVVB11 nos anos de  $2017$ a $2023$.
\begin{figure}[H]
    \centering
    \centering
    \subfloat[Correlação 2017.]{\label{a}\includegraphics[width=0.25\linewidth]{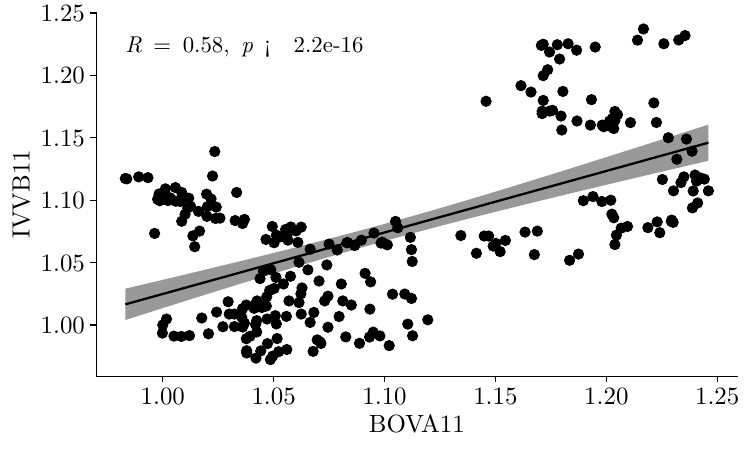}}
    \subfloat[Correlação 2018.]{\label{a}\includegraphics[width=0.25\linewidth]{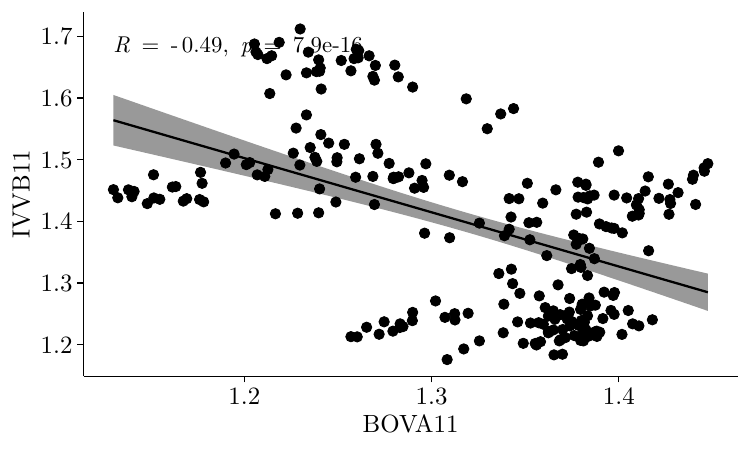}}
    \subfloat[Correlação 2019.]{\label{a}\includegraphics[width=0.25\linewidth]{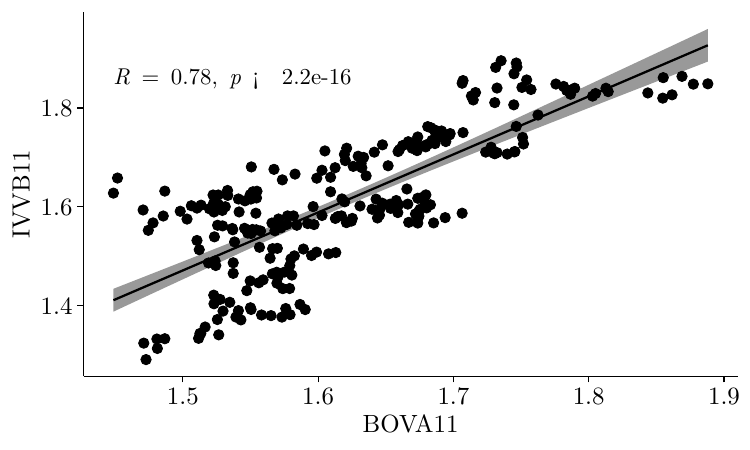}}
    \subfloat[Correlação 2020.]{\label{a}\includegraphics[width=0.25\linewidth]{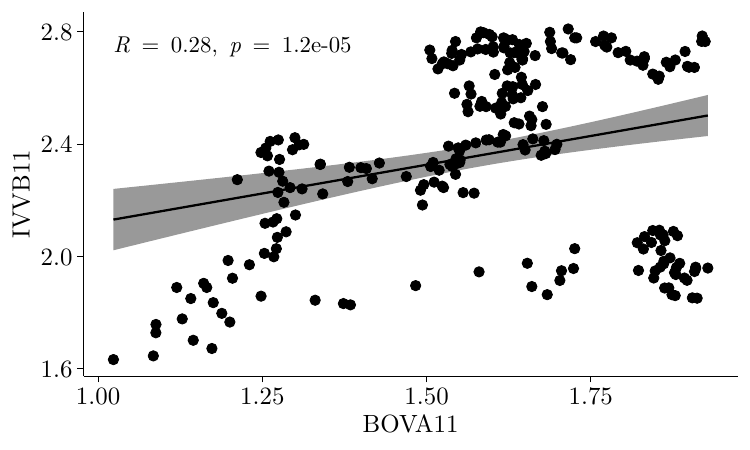}} \\
    \subfloat[Correlação 2021.]{\label{a}\includegraphics[width=0.25\linewidth]{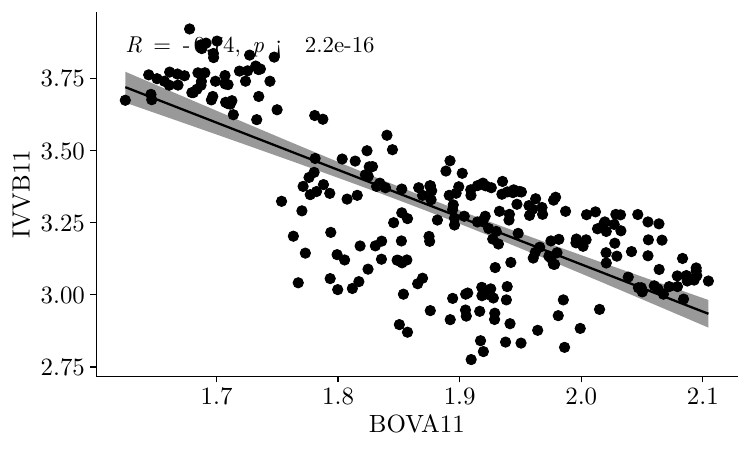}}
    \subfloat[Correlação 2022.]{\label{a}\includegraphics[width=0.25\linewidth]{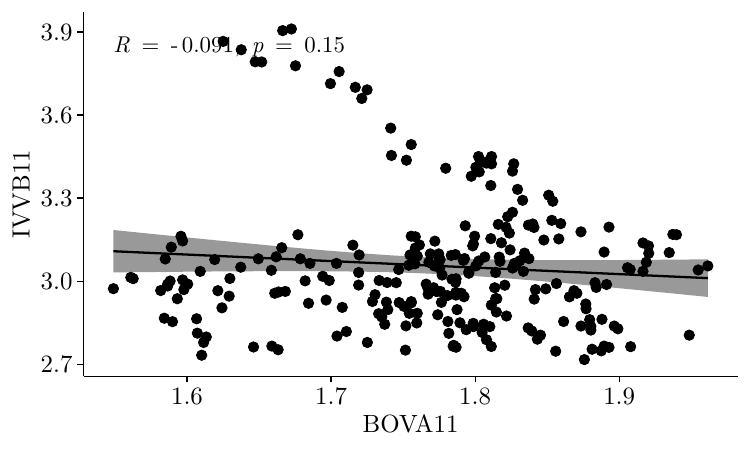}}
    \subfloat[Correlação 2023.]{\label{a}\includegraphics[width=0.25\linewidth]{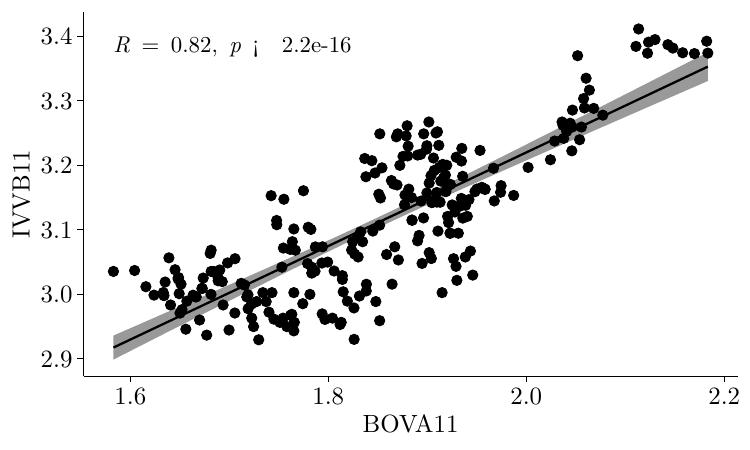}}

    \caption{\footnotesize Correlação entre  BOVA11 e IVVB11 de $01/2017$ à $12/2023$. }
    \label{Correlacao}
\end{figure}
Como mostrado na Tabela~\ref{Tab:Corr} e Figura~\ref{Correlacao}, a correlação entre os ativos varia ao longo do tempo. No entanto, é importante observar que, na maior parte do tempo, eles apresentam correlação baixa ou negativa, resultando em um baixo risco para o portfólio composto por esses dois ativos. Assim, para uma estratégia de longo prazo, consideramos que IVVB11 e BOVA11, de certo modo, também satisfazem o item~(ii) acima.
\begin{remark}
    Vale a pena destacar  que  os ativos  IVVB11 e  BOVA11 são ``Exchange Traded Fund (ETF) ou  fundo de índice. Um ETF é por definição  um ativo financeiro que tem como referência uma carteira teórica de algum índice do mercado. O IVVB11 é um EFT negociado na B3 que replica o índice  S$\&$P 500, que reúne as 500 maiores empresas em valor de mercado negociadas nos Estados Unidos.  Já o fundo de índice BOVA11  replica o índice Bovespa com as principais empresas negociadas no Brasil. Como  BOVA11 e  IVVB11 replicam  índices de  dois mercados distintos, isto é,  BOVA11 permite exposição ao mercado nacional enquanto o IVVB11 permite se expor ao mercado norte-americano, as seguintes características destes dois ativos combinados merecem destaques:
    \begin{itemize}
        \item  Os ativos BOVA11 e  IVVB11 combinados resulta em um portfólio altamente diversificado.  Assim,  investir  em uma combinação destes ativos aumenta a probabilidade do investidor  se  expor ao mercado com  menor  risco quando comparado com um portfólio menos diversificado;
        \item  Uma vez que o IVVB11 esta mais exposto ao dólar e o BOVA11 ao real, isto explica em parte o comportamento da correlação entre eles. De fato, quando o dólar se valoriza frente ao real  significa que a economia brasileira no período pode  não gozar de um bom desempenho comparada com a americana, acontecendo  o contrário  quando o dólar se desvaloriza. Esta oscilação cambial, em geral,   implica  na maioria das vezes  valorização de um dos  ativos e a  depreciação do outro.
    \end{itemize}
\end{remark}
Concluímos que,  para o investidor avesso ao risco, a combinação do BOVA11 com  IVVB11 em uma estratégia de investimento  é uma  alternativa interessante  a ser considerada.
\subsection{Descrição das  Estratégias Ingênua e  Mínimo Risco} \label{Sec:DescEstrategia-1}
Após a justificativa para a escolha dos ativos BOVA11 e IVVB11, avançamos para a descrição da estratégia que combina esses dois ativos. É importante destacar que, na estratégia apresentada, não consideramos um ativo livre de risco, pois ela será obtida como uma aplicação do Corolário~\ref{cr:prm}. Dentre várias estratégias possíveis, com o objetivo de um investimento de longo prazo, optamos pela seguinte abordagem:

\begin{enumerate}
    \item[(AP)] Aumentar  o tamanho do portifólio   realizando {\it aportes periódicos (AP)}. Neste caso,  a cada período pré-definido,  comprar  o ativo que está mais distante da proporção especificada previamente pelo balanceamento desejado.
\end{enumerate}
Neste caso, há várias opções para os aportes periódicos, por exemplo:
\begin{enumerate}
    \item[(AI)] {\it Aporte Ingênuo}: Neste caso, os aportes seriam feitos  em um único  ativo de modo que ao final do aporte a proporção de cada ativo na carteia seja a mais próxima possível $50\%$; O que conseguimos aportando no ativo que tem a menor proporção na carteira no dia do aporte.
    \item[(AMR)] {\it Aporte de Mínimo Risco}:  Neste caso, os aportes seriam feitos  em um único  ativo de modo que ao final do aporte a proporção de cada ativo na carteia seja a mais próxima possível do vetor de alocação de capital dado pelo Corolário~\ref{cr:prm}.
\end{enumerate}

\begin{remark}
    É importante observar alguns  detalhes técnicos que  poderão ocorrer  na implementação da estratégia na prática\footnote{Isso não quer dizer que estamos indicando a implementação desta estratégia.}.
    \begin{itemize}
        \item Se o investidor optar por aportes constantes na estratégia,   como não é possível comprar   fração de um  ativo, então  certamente irão sobrar resíduos a cada aporte. Neste caso, seria interessante fixar um intervalo de variação  para os aportes.
        \item  Pode ocorrer concentração em algum dos ativos quando usado AMR. Neste caso, visando controlar o risco,  seria conveniente realizar um balanceamento do portifólio  em períodos previamente estabelecidos, usando a proporção de $50\%$ em cada ativo ou  usando vetor de alocação de capital dado pelo Corolário~\ref{cr:prm}. Por exemplo, a cada seis meses ou um ano.
    \end{itemize}
\end{remark}
A opções (AI) e  (AMR)  para os aportes  serão discutidas na próxima seção.

\subsubsection{Implementação da estratégia} \label{Sec:ImpEstrategia}

Para implementar a estratégia (AP) com aportes (AI) e (AMR), nosso programa armazena o primeiro ano como fonte de dados para a Estratégia Mínimo Rísco, e o processo inicia no primeiro mês do segundo ano. Com o objetivo de manter um padrão nos testes, todas as estratégias iniciam ao mesmo tempo. Conforme mencionado anteriormente, utilizamos os dados de 2017 e iniciamos os testes em 01/2018 até 12/2023, considerando os aportes da seguinte forma:

\begin{itemize}
    \item  Aporte inicial de R\$ $1.000,00$  no início  de  janeiro de 2018;
    \item  Aportes mensais de R\$ $400,00$ iniciando em fevereiro de 2018.
\end{itemize}
Antes de mostrar as estratégias em cada ano é importante saber como os aportes são realizados.
\begin{itemize}
    \item O {\it Aporte Ingênuo} de  R\$ $400,00$  é feito ao primeiro dia útil de cada mês no ativo que possui a menor proporção na carteira no último dia do mês anterior.
    \item {\it Aporte de Mínimo Risco} de  R\$ $400,00$ é feito no primeiro dia de cada mês de modo que ao final do aporte a proporção de cada ativo na carteia seja a mais próxima possível do vetor de alocação de capital dado pelo Corolário~\ref{cr:prm} utilizando os dados dos últimos  doze meses.
\end{itemize}
\begin{remark}
    Para tornar mais simples a nossa análise vamos supor que é possível comprar frações dos ativos BOVA11 e IVVB11.
\end{remark}
Vamos exemplificar abaixo como faremos o  Aporte de Mínimo Risco em 01/02/2018 (ou primeiro dia útil). Primeiramente, precisamos das porcentagens de BOVA11 e IVVB11 atuais na carteira em 31/01/2018, além de calcular um novo patrimônio para simular em qual ativo será mais benéfico para a estratégia realizar o aporte, como citado anteriormente. Para simplificar as notações, definimos:
\begin{itemize}
    \item QBOVA11 é a quantidade total em reais de BOVA11 na carteira;
    \item PBOVA11 é a porcentagem  total  de BOVA11 na carteira;
    \item QIVVB11 é a quantidade em reais de IVVB11 na carteira;
    \item PIVVB11 é a porcentagem  total  de IVVB11 na carteira;
    \item NPRAT é o patrimonio total mais aporte mensal.
\end{itemize}
{\it Ao longo desta seção vamos usar duas casas decimais na representação dos números.}
Assim, utilizando as notações  com  dados em 31/01/2018 temos os seguintes valores com duas casas decimais na representação dos números:
\begin{align*}
    \begin{aligned}
        \mbox{PBOVA11} & = \frac{\mbox{QBOVA11}}{\mbox{QBOVA11+ QIVVB11}}\times 100 = \frac{463,82}{1.052,64}\times 100  = 44,06  \\
        \mbox{PIVVB11} & = \frac{\mbox{QIVVB11}}{\mbox{QBOVA11+ QIVVB11l}}\times 100 = \frac{588,82}{1.052,64}\times 100  = 55,94 \\
    \end{aligned}
\end{align*}
Finalmente, para decidir em qual ativo faremos o  aporte de Mínimo Risco em 01/02/2018, precisamos ainda das percentagens sugeridas (PS) de pelo menos um dos ativos, as quais podem ser obtidas usando o Corolário~\ref{cr:prm} com dados dos últimos doze meses. Após isso, calcularemos como ficaria a porcentagem de determinado ativo com e sem o aporte; em seguida, sabendo como ficaria a porcentagem após o aporte, saberemos em qual aportar. Vamos usar as seguintes notações:
\begin{itemize}
    \item PS-IVVB11 é a porcentagem   sugerida na carteira   de IVVB11;
    \item DP-IVVB11A é a porcentual de diferença do percentual sugerido para o IVVB11 se o aporte for no IVVB11;
    \item DP-IVVB11 é a poercentual de diferença do percentual sugerido para IVVB11 se o aporte for realizado no BOVA11
\end{itemize}

Deste modo, usando as notações acima, temos que PS-BOVA11 = $41,81\%$ e PS-IVVB11 = $58,19\%$ são as porcentagens na carteira sugeridas pelo programa em 31/01/2018, as quais são obtidas usando o Corolário~\ref{cr:prm} com dados dos últimos doze meses. Assim, o  Aporte de Mínimo Risco será feito baseado no DP-IVVB11A e DP-IVVB11, ou seja, se o valor absoluto do DP-IVVB11A for o menor, significa que o aporte no ativo IVVB11 será o mais benéfico para a estratégia; caso contrário, o aporte deverá ser feito no BOVA11, o que pode ser identificado fazendo os seguintes cálculos:
\begin{align*}
    \text{DP-IVVB11A} & = \frac{\text{QIVVB11} + \text{Aporte Mensal}}{\text{NPRAT}} - \text{PS-IVVB11} = \frac{\text{588,82} + \text{400,00}}{\text{1452.64}} - \text{58,19} = 9,88 \\
    \text{DP-IVVB11}  & = \frac{\text{QIVVB11}}{\text{NPRAT}} - \text{PS-IVVB11} = \frac{\text{588,82}}{\text{1452.64}} - \text{58,19} = -17,65
\end{align*}
Como os cálculos acima mostram, DP-IVVB11A está a $9,88\%$ de distância do percentual sugerido e o valor absoluto de DP-IVVB11 está a $17,65\%$ de distância do percentual sugerido. Logo, de acordo com a estratégia, o aporte em 01/02/2018 será feito no IVVB11. Na Tabela~\ref{tb:AM2018} apresentamos todos os ativos escolhidos para os aportes de mínimo risco realizados em 2018.
\begin{table}[H]
    \begin{tiny}
        \begin{center}
            \begin{tabular}{cccccrrc}
                Data  & PS-BOVA11 & PS-IVVB11 & P-BOVA11 & P-IVVB11 & \multicolumn{2}{c}{Diferença percentual} & Ativo Escolhido          \\ \cline{6-7}
                      &           &           &          &          & DP-IVVBA11                               & DP-IVVB11       &        \\\hline
                1/18  & 42,48\%   & 57,52\%   & 42,48\%  & 57,52\%  &                                          &                 &        \\
                2/18  & 41,81\%   & 58,19\%   & 32,13\%  & 67,87\%  & 9,68\%                                   & 17,85\%         & IVVB11 \\
                3/18  & 40,94\%   & 59,06\%   & 47,39\%  & 52,61\%  & 15,42\%                                  & 6,45\%          & BOVA11 \\
                4/18  & 41,19\%   & 58,81\%   & 38,10\%  & 61,90\%  & 2,20\%                                   & 15,95\%         & IVVB11 \\
                5/18  & 42,78\%   & 57,22\%   & 46,08\%  & 53,92\%  & 11,35\%                                  & 3,30\%          & BOVA11 \\
                6/18  & 43,98\%   & 56,02\%   & 48,80\%  & 51,20\%  & 7,78\%                                   & 4,82\%          & BOVA11 \\
                7/18  & 37,50\%   & 62,50\%   & 41,24\%  & 58,76\%  & 3,74\%                                   & 15,04\%         & IVVB11 \\
                8/18  & 40,89\%   & 59,11\%   & 39,14\%  & 60,86\%  & 1,75\%                                   & 8,11\%          & IVVB11 \\
                9/18  & 39,96\%   & 60,04\%   & 40,14\%  & 59,86\%  & 8,16\%                                   & 0,18\%          & BOVA11 \\
                10/18 & 38,71\%   & 61,29\%   & 38,73\%  & 61,27\%  & 0,01\%                                   & 7,81\%          & IVVB11 \\
                11/18 & 40,31\%   & 59,69\%   & 41,78\%  & 58,22\%  & 1,46\%                                   & 8,95\%          & IVVB11 \\
                12/18 & 39,16\%   & 60,84\%   & 38,02\%  & 61,98\%  & 1,13\%                                   & 5,58\%          & IVVB11 \\
                \hline
            \end{tabular}
            \caption{\footnotesize Aportes de mínimo risco} \label{tb:AM2018}
        \end{center}
    \end{tiny}
\end{table}
Como já sabemos como funcionam os aportes em cada estratégia, vamos ver agora o desenvolvimento das Estratégias Ingênua e de Mínimo Risco. Nas seis tabelas a seguir, mostramos mensalmente a quantidade de ativos em cada estratégia e o patrimônio acumulado ao final de cada mês. {\it Enfatizamos que, devido aos dados serem obtidos via a nossa implementação, haverão pequenos erros de arredondamento nas tabelas}. Usaremos as seguintes notações na confecção das tabelas:
\begin{itemize}
    \item  C-BOVA11 é a cotação do ativo BOVA11 ao final de cada mês;
    \item  C-IVVB11 é a cotação do ativo  IVVB11 ao final de cada mês;
    \item  N-BOVA11 é a quantidade do  ativo BOVA11 na carteira  ao final de cada mês;
    \item  N-IVVB11 é a quantidade do  ativo  IVVB11 na carteira  ao final de cada mês;
    \item Patrimônio-I é a patrimônio acumulado ao final de cada mês usando a estratégia Ingênua;
    \item Patrimônio-M é a patrimônio acumulado ao final de cada mês usando a Estratégia  Mínimo Risco.
\end{itemize}

\begin{table}[H]
    \begin{tiny}
        \begin{center}
            \begin{tabular}{crrrrrrrr}
                Data  & C-BOVA11  & C-IVVB11   & \multicolumn{3}{c}{Ingênua} & \multicolumn{3}{c}{Mínimo risco}                                                      \\ \cline{4-6} \cline{7-9}
                      &           &            & N-BOVA11                    & N-IVVB11                         & Patrimônio-I & N-BOVA11 & N-IVVB11 & Patrimônio-M  \\\hline
                1/18  & R\$ 82,00 & R\$ 94,30  & 6,77                        & 5,34                             & R\$ 1.059,41 & 5,66     & 6,24     & R\$  1.052,64 \\
                2/18  & R\$ 82,50 & R\$ 92,60  & 6,77                        & 9,58                             & R\$ 1.446,50 & 5,66     & 10,50    & R\$  1.439,27 \\
                3/18  & R\$ 82,43 & R\$ 91,66  & 11,62                       & 9,58                             & R\$ 1.836,67 & 10,50    & 10,50    & R\$  1.828,57 \\
                4/18  & R\$ 83,18 & R\$ 97,60  & 11,62                       & 13,95                            & R\$ 2.328,26 & 10,50    & 14,94    & R\$  2.332,81 \\
                5/18  & R\$ 74,07 & R\$ 107,37 & 16,43                       & 13,95                            & R\$ 2.714,87 & 15,39    & 14,94    & R\$  2.745,78 \\
                6/18  & R\$ 70,35 & R\$ 111,00 & 21,83                       & 13,95                            & R\$ 3.084,30 & 20,75    & 14,94    & R\$  3.119,87 \\
                7/18  & R\$ 76,45 & R\$ 111,58 & 27,51                       & 13,95                            & R\$ 3.660,24 & 20,75    & 18,51    & R\$  3.652,44 \\
                8/18  & R\$ 73,74 & R\$ 127,15 & 27,51                       & 17,53                            & R\$ 4.234,96 & 20,75    & 21,93    & R\$  4.304,32 \\
                9/18  & R\$ 76,61 & R\$ 124,10 & 32,91                       & 17,53                            & R\$ 4.697,62 & 26,20    & 21,93    & R\$  4.748,75 \\
                10/18 & R\$ 84,08 & R\$ 106,95 & 32,91                       & 20,75                            & R\$ 4.987,48 & 26,20    & 25,15    & R\$  4.910,21 \\
                11/18 & R\$ 86,37 & R\$ 112,90 & 32,91                       & 24,49                            & R\$ 5.608,62 & 26,20    & 28,89    & R\$  5.542,12 \\
                12/18 & R\$ 84,60 & R\$ 102,70 & 32,91                       & 28,04                            & R\$ 5.664,34 & 26,20    & 32,57    & R\$  5.561,70 \\
                \hline
            \end{tabular}
            \caption{ \scriptsize Testes em 2018.}
        \end{center}
    \end{tiny}
\end{table}

\begin{table}[H]
    \begin{tiny}
        \begin{center}
            \begin{tabular}{crrccrccr}
                Data  & C-BOVA11   & C-IVVB11   & \multicolumn{3}{c}{Ingênua} & \multicolumn{3}{c}{Mínimo risco}                                                        \\ \cline{4-6} \cline{7-9}
                      &            &            & N-BOVA11                    & N- IVVB11                        & Patrimônio-I  & N-BOVA11 & N-IVVB11 & Patrimônio-M   \\\hline
                1/19  & R\$ 94,00  & R\$ 104,52 & 37,64                       & 28,04                            & R\$ 6.469,20  & 30,94    & 32,43    & R\$  6.297,95  \\
                2/19  & R\$ 91,85  & R\$ 111,25 & 37,64                       & 31,86                            & R\$ 7.002,75  & 30,94    & 36,26    & R\$  6.875,76  \\
                3/19  & R\$ 91,81  & R\$ 118,15 & 41,99                       & 31,86                            & R\$ 7.620,96  & 35,29    & 36,26    & R\$  7.524,09  \\
                4/19  & R\$ 92,88  & R\$ 122,85 & 41,99                       & 35,25                            & R\$ 8.231,59  & 39,65    & 36,26    & R\$  8.137,23  \\
                5/19  & R\$ 93,59  & R\$ 115,05 & 46,30                       & 35,25                            & R\$ 8.389,49  & 39,65    & 39,52    & R\$  8.257,62  \\
                6/19  & R\$ 97,11  & R\$ 119,70 & 46,30                       & 38,73                            & R\$ 9.132,57  & 39,65    & 43,00    & R\$  8.997,51  \\
                7/19  & R\$ 97,98  & R\$ 121,11 & 50,42                       & 38,73                            & R\$ 9.631,05  & 43,77    & 43,00    & R\$  9.496,31  \\
                8/19  & R\$ 97,61  & R\$ 129,08 & 50,42                       & 42,03                            & R\$ 10.347,40 & 43,77    & 46,30    & R\$  10.248,79 \\
                9/19  & R\$ 101,01 & R\$ 132,32 & 54,52                       & 42,03                            & R\$ 11.068,96 & 43,77    & 49,40    & R\$  10.957,81 \\
                10/19 & R\$ 103,23 & R\$ 130,00 & 58,48                       & 42,03                            & R\$ 11.501,27 & 47,73    & 49,40    & R\$  11.349,16 \\
                11/19 & R\$ 104,35 & R\$ 143,60 & 58,48                       & 45,11                            & R\$ 12.580,27 & 47,73    & 52,48    & R\$  12.516,75 \\
                12/19 & R\$ 111,23 & R\$ 138,71 & 62,31                       & 45,11                            & R\$ 13.188,39 & 51,56    & 52,48    & R\$  13.014,52 \\
                \hline
            \end{tabular}
            \caption{\scriptsize  Testes em 2019}
        \end{center}
    \end{tiny}
\end{table}

\begin{table}[H]
    \begin{tiny}
        \begin{center}
            \begin{tabular}{crrccrccr}
                Data  & C-BOVA11   & C-IVVB11   & \multicolumn{3}{c}{Ingênua} & \multicolumn{3}{c}{Mínimo risco}                                                        \\ \cline{4-6} \cline{7-9}
                      &            &            & N-BOVA11                    & N-IVVB11                         & Patrimônio-I  & N-BOVA11 & N-IVVB11 & Patrimônio-M   \\\hline
                1/20  & R\$ 108,90 & R\$ 148,10 & 62,31                       & 47,99                            & R\$ 13.893,87 & 51,22    & 55,59    & R\$  13.810,61 \\
                2/20  & R\$ 100,60 & R\$ 141,56 & 65,98                       & 47,99                            & R\$ 13.432,30 & 54,84    & 55,59    & R\$  13.386,35 \\
                3/20  & R\$ 69,35  & R\$ 144,62 & 69,96                       & 47,99                            & R\$ 11.792,83 & 54,84    & 58,28    & R\$  12.231,82 \\
                4/20  & R\$ 77,21  & R\$ 170,50 & 75,73                       & 47,99                            & R\$ 14.030,16 & 60,71    & 58,28    & R\$  14.624,15 \\
                5/20  & R\$ 84,15  & R\$ 175,65 & 80,91                       & 47,99                            & R\$ 15.238,86 & 60,71    & 60,64    & R\$  15.760,87 \\
                6/20  & R\$ 91,62  & R\$ 181,72 & 85,66                       & 47,99                            & R\$ 1.6570,10 & 60,71    & 62,90    & R\$  16.992,80 \\
                7/20  & R\$ 99,29  & R\$ 183,68 & 90,03                       & 47,99                            & R\$ 17.754,70 & 60,71    & 65,15    & R\$  17.994,52 \\
                8/20  & R\$ 95,70  & R\$ 207,15 & 90,03                       & 50,17                            & R\$ 19.009,03 & 60,71    & 67,27    & R\$  19.744,27 \\
                9/20  & R\$ 91,05  & R\$ 203,98 & 94,21                       & 50,17                            & R\$ 18.811,90 & 60,71    & 69,22    & R\$  19.646,74 \\
                10/20 & R\$ 90,66  & R\$ 202,60 & 98,60                       & 50,17                            & R\$ 19.104,21 & 60,71    & 71,15    & R\$  19.919,42 \\
                11/20 & R\$ 105,00 & R\$ 210,00 & 103,01                      & 50,17                            & R\$ 21.352,72 & 60,71    & 73,06    & R\$  21.716,53 \\
                12/20 & R\$ 114,65 & R\$ 210,04 & 103,01                      & 52,07                            & R\$ 22.748,90 & 60,71    & 74,99    & R\$  22.711,17 \\
                \hline
            \end{tabular}
            \caption{\footnotesize  Testes em 2020}
        \end{center}
    \end{tiny}
\end{table}

\begin{table}[H]
    \begin{tiny}
        \begin{center}
            \begin{tabular}{crrccrccr}
                Data  & C-BOVA11   & C-IVVB11   & \multicolumn{3}{c}{Ingênua} & \multicolumn{3}{c}{Mínimo risco}                                                        \\ \cline{4-6} \cline{7-9}
                      &            &            & N-BOVA11                    & N-IVVB11                         & Patrimônio-I  & N-BOVA11 & N-IVVB11 & Patrimônio-M   \\\hline
                1/21  & R\$ 110,56 & R\$ 220,00 & 103,01                      & 53,98                            & R\$ 23.265,22 & 60,71    & 76,89    & R\$  23.627,22 \\
                2/21  & R\$ 105,59 & R\$ 231,00 & 106,63                      & 53,98                            & R\$ 23.729,05 & 60,71    & 78,69    & R\$  24.588,79 \\
                3/21  & R\$ 112,02 & R\$ 243,20 & 110,42                      & 53,98                            & R\$ 25.497,63 & 60,71    & 80,37    & R\$  26.346,59 \\
                4/21  & R\$ 114,40 & R\$ 247,00 & 113,99                      & 53,98                            & R\$ 26.374,06 & 60,71    & 81,97    & R\$  27.192,64 \\
                5/21  & R\$ 121,25 & R\$ 238,00 & 117,48                      & 53,98                            & R\$ 27.093,03 & 64,20    & 81,97    & R\$  27.293,66 \\
                6/21  & R\$ 121,81 & R\$ 232,50 & 117,48                      & 55,66                            & R\$ 27.252,68 & 67,44    & 81,97    & R\$  27.273,82 \\
                7/21  & R\$ 117,11 & R\$ 250,35 & 117,48                      & 57,36                            & R\$ 28.124,76 & 70,75    & 81,97    & R\$  28.808,08 \\
                8/21  & R\$ 114,08 & R\$ 254,00 & 120,90                      & 57,36                            & R\$ 28.367,86 & 74,16    & 81,97    & R\$  29.280,93 \\
                9/21  & R\$ 106,49 & R\$ 255,00 & 124,41                      & 57,36                            & R\$ 27.880,97 & 77,64    & 81,97    & R\$  29.171,10 \\
                10/21 & R\$ 99,68  & R\$ 283,00 & 128,16                      & 57,36                            & R\$ 29.014,85 & 81,33    & 81,97    & R\$  31.305,15 \\
                11/21 & R\$ 98,35  & R\$ 280,55 & 132,17                      & 57,36                            & R\$ 29.098,47 & 85,27    & 81,97    & R\$  31.383,66 \\
                12/21 & R\$ 100,80 & R\$ 293,56 & 136,24                      & 57,36                            & R\$ 30.578,82 & 89,39    & 81,97    & R\$  33.074,32 \\
                \hline
            \end{tabular}
            \caption{\footnotesize  Testes em 2021}
        \end{center}
    \end{tiny}
\end{table}

\begin{table}[H]
    \begin{tiny}
        \begin{center}
            \begin{tabular}{crrccrccr}
                Data  & C-BOVA11   & C-IVVB11   & \multicolumn{3}{c}{Ingênua} & \multicolumn{3}{c}{Mínimo risco}                                                        \\ \cline{4-6} \cline{7-9}
                      &            &            & N-BOVA11                    & N-IVVB11                         & Patrimônio-I  & N-BOVA11 & N-IVVB11 & Patrimônio-M   \\\hline
                1/22  & R\$ 107,98 & R\$ 260,40 & 140,24                      & 57,36                            & R\$ 30.079,65 & 93,39    & 81,97    & R\$  31.430,10 \\
                2/22  & R\$ 108,90 & R\$ 245,50 & 140,24                      & 58,92                            & R\$ 29.732,08 & 97,05    & 81,97    & R\$  30.693,70 \\
                3/22  & R\$ 115,60 & R\$ 235,65 & 140,24                      & 60,55                            & R\$ 30.475,64 & 100,65   & 81,97    & R\$  30.952,73 \\
                4/22  & R\$ 103,40 & R\$ 224,10 & 140,24                      & 62,25                            & R\$ 28.446,56 & 104,07   & 81,97    & R\$  29.130,85 \\
                5/22  & R\$ 107,18 & R\$ 214,50 & 140,24                      & 64,03                            & R\$ 28.761,07 & 107,96   & 81,97    & R\$  29.154,65 \\
                6/22  & R\$ 95,00  & R\$ 216,70 & 140,24                      & 65,89                            & R\$ 27.596,83 & 111,68   & 81,97    & R\$  28.373,60 \\
                7/22  & R\$ 99,30  & R\$ 234,50 & 144,45                      & 65,89                            & R\$ 29.790,40 & 115,87   & 81,97    & R\$  30.728,32 \\
                8/22  & R\$ 105,83 & R\$ 226,00 & 148,45                      & 65,89                            & R\$ 30.600,25 & 119,92   & 81,97    & R\$  31.217,66 \\
                9/22  & R\$ 106,46 & R\$ 213,50 & 148,45                      & 67,66                            & R\$ 30.248,32 & 123,68   & 81,97    & R\$  30.668,06 \\
                10/22 & R\$ 112,15 & R\$ 220,24 & 148,45                      & 69,54                            & R\$ 31.960,91 & 127,24   & 81,97    & R\$  32.323,39 \\
                11/22 & R\$ 108,58 & R\$ 233,25 & 148,45                      & 71,35                            & R\$ 32.759,68 & 130,78   & 81,97    & R\$  33.320,64 \\
                12/22 & R\$ 105,95 & R\$ 225,00 & 152.14                      & 71,35                            & R\$ 32.170,60 & 134,52   & 81,97    & R\$  32.696,25 \\
                \hline
            \end{tabular}
            \caption{\footnotesize  Testes em 2022}
        \end{center}
    \end{tiny}
\end{table}

\begin{table}[H]
    \begin{tiny}
        \begin{center}
            \begin{tabular}{crrccrccr}
                Data  & C-BOVA11   & C-IVVB11   & \multicolumn{3}{c}{Ingênua} & \multicolumn{3}{c}{Mínimo risco}                                                        \\ \cline{4-6} \cline{7-9}
                      &            &            & N-BOVA11                    & N-IVVB11                         & Patrimônio-I  & N-BOVA11 & N-IVVB11 & Patrimônio-M   \\\hline
                1/23  & R\$ 108,7  & R\$ 228,2  & 152,16                      & 73,11                            & R\$ 33.396,95 & 138,42   & 81,97    & R\$  33.909,94 \\
                2/23  & R\$ 102.29 & R\$ 229.3  & 152,16                      & 74,86                            & R\$ 32.556,38 & 142,11   & 81,97    & R\$  33.170,53 \\
                3/23  & R\$ 100.47 & R\$ 229.75 & 156,11                      & 74,86                            & R\$ 32.574,36 & 146,06   & 81,97    & R\$  33.218,70 \\
                4/23  & R\$ 99.53  & R\$ 229.75 & 160,17                      & 74,86                            & R\$ 33.360,24 & 150,13   & 81,97    & R\$  33.981,28 \\
                5/23  & R\$ 105.38 & R\$ 233.25 & 164,13                      & 74,86                            & R\$ 34.668,48 & 154,19   & 81,97    & R\$  35.285,48 \\
                6/23  & R\$ 114.91 & R\$ 235.25 & 167,95                      & 74,86                            & R\$ 36.785,66 & 157,92   & 81,97    & R\$  37.314,02 \\
                7/23  & R\$ 116.52 & R\$ 239.82 & 167,95                      & 76,56                            & R\$ 38.170,32 & 161,37   & 81,97    & R\$  38.692,51 \\
                8/23  & R\$ 112.31 & R\$ 247.65 & 167,95                      & 78,23                            & R\$ 38.236,12 & 164,77   & 81,97    & R\$  38.806,50 \\
                9/23  & R\$ 113.15 & R\$ 238.79 & 171,51                      & 78,23                            & R\$ 38.086,89 & 168,28   & 81,97    & R\$  38.615,47 \\
                10/23 & R\$ 109.62 & R\$ 234,00 & 171,51                      & 79,91                            & R\$ 37.499,86 & 171,87   & 81,97    & R\$  38.022,08 \\
                11/23 & R\$ 123.57 & R\$ 249.7  & 171,51                      & 81,62                            & R\$ 41.574,00 & 171,87   & 83,69    & R\$  42.134,57 \\
                12/23 & R\$ 130.39 & R\$ 257.61 & 171,51                      & 83,22                            & R\$ 43.801,49 & 171,87   & 85,29    & R\$  44.382,68 \\
                \hline
            \end{tabular}
            \caption{\footnotesize  Testes em 2023}
        \end{center}
    \end{tiny}
\end{table}

As  tabelas acima mostram  que durante os dois primeiros anos a estratégia Ingênua se saiu ligeiramente melhor que a de Mínimo Risco e ficaram praticamente empatadas no  terceiro ano  e nos anos posteriores a estratégia de Mínimo Risco se saiu melhor. Considerando todo o período,  o patrimônio  final da estratégia de de mínimo risco ficou $1,32\%$ acima  do patrimônio  final da estratégia ingênua.

\subsubsection{Cotização} \label{cotizacao}
Uma das maneiras mais populares de se avaliar o rendimento de um portfólio é o {\it sistema de cotas}, pois ele permite comparar o rendimento do portfólio com índices de referência e também com o rendimento de outros portfólios. O sistema de cotas, também chamado de {\it cotização} do portfólio analisa a variação do valor patrimonial diário do portfólio. Para obter o valor da cota define-se um valor inicial, geralmente igual a 1 e sua {\it rentabilidade acumulada} é obtida pela {\it multiplicação dos retornos diários} do portfólio. Note que o montante acumulado não entra no cálculo da cota, apenas o retorno. Outro ponto importante é que aportes feitos ao longo do tempo tampouco afetam o valor da cota pois a mesma só considera o rendimento. Vejamos um exemplo simples. Considere o seguinte:
\begin{enumerate}
    \item Faz-se um aporte inicial de R\$1.000,00;
    \item Após 1 mês o rendimento acumulado foi de 10\%;
    \item O investidor então faz um novo aporte de R\$300,00;
    \item No mês seguinte o rendimento foi de -5\%.
\end{enumerate}
Vamos calcular o valor da cota. Começamos com a cota valendo 1 e após 1 mês o capital é R\$ 1.100,00 e a cota 1,1. O investidor então faz um aporte de R\$300,00 o que altera o valor do capital para R\$ 1.400,00 mas não afeta o valor da cota. No mês seguinte houve uma queda de 5\% logo o valor da cota é de 1.045 e o valor do capital R\$ 1.330,00.

\begin{align*}
    \begin{aligned}
        \mbox{valor da cota}    & = (1+0.1)\times(1-0.05) = 1.045                    \\
        \mbox{valor do capital} & = [1.000\times(1+0.1) + 300]\times(1-0.05) = 1.330 \\
    \end{aligned}
\end{align*}
Note que o rendimento da cota é diferente do rendimento do capital.
\begin{align*}
    \begin{aligned}
        \mbox{rentabilidade da cota}    & = \frac{\mbox{valor final}}{\mbox{valor inicial}}-1=\frac{1.045}{1}-1 = 0.045            \\
        \mbox{rentabilidade do capital} & = \frac{\mbox{valor final}}{\mbox{valor aplicado}}-1 = \frac{1.330}{1000+300}-1 = 0.023. \\
    \end{aligned}
\end{align*}
Isso acontece porque o aporte afeta o rendimento do capital, mas desconsidera o valor do dinheiro no tempo. O sistema de cotas, entretanto, não é afetado pelos aportes. Se tivéssemos apenas o fluxo do dinheiro ainda assim seríamos capazes de calcular o valor da cota. O fluxo do dinheiro foi o seguinte
\begin{enumerate}
    \item Depósito inicial de R\$1.000,00;
    \item Saldo no final do mês de R\$1.100,00;
    \item Depósito no valor de R\$300,00;
    \item Saldo no fim do mês de R\$1.330,00;
\end{enumerate}
Como o saldo no final do primeiro mês foi de R\$1.100,00 o rendimento do período foi de 0,1 ou 10\% e a cota agora vale 1.1. Realizando o depósito, o saldo atual fica em R\$1400,00. Após o segundo mês o saldo é de R\$ 1330,00 onde podemos calcular o retorno $1330/1440 -1=-0.05$, atualizamos agora o valor da cota para $ (1+0.1)\times(1-0.05) = 1.045$.

Vamos considerar agora a movimentação de um portfólio que possui um único ativo, ações da Petrobras-SA, cujo código de negociação na B3 é PETR4. A rentabilidade do portfólio é, neste caso, a rentabilidade da ação, o valor da cota é a rentabilidade acumulada e o capital é medido em número de cotas.

\begin{table}[H]
    \begin{tiny}
        \begin{center}
            \begin{tabular}{cccrccr}
                Data    & PETR4     & \multicolumn{1}{c}{Retorno} & \multicolumn{1}{c}{Movimentação} & Valor da Cota & Qtde de Cotas & \multicolumn{1}{c}{Capital} \\\hline
                01/2/22 & R\$ 33,00 &                             & R\$ 1.000,00                     & 1,0000        & 1000,00       & R\$ 1.000,00                \\
                02/2/22 & R\$ 32,52 & -0,0145                     &                                  & 0,9855        & 1000,00       & R\$ 985,45                  \\
                03/2/22 & R\$ 32,07 & -0,0138                     &                                  & 0,9718        & 1000,00       & R\$ 971,82                  \\
                04/2/22 & R\$ 32,63 & 0,0175                      & R\$ 200,00                       & 0,9888        & 1202,27       & R\$ 1.188,79                \\
                07/2/22 & R\$ 32,15 & -0,0147                     &                                  & 0,9742        & 1202,27       & R\$ 1.171,30                \\
                08/2/22 & R\$ 31,83 & -0,0100                     &                                  & 0,9645        & 1202,27       & R\$ 1.159,64                \\
                09/2/22 & R\$ 31,95 & 0,0038                      &                                  & 0,9682        & 1202,27       & R\$ 1.164,01                \\
                10/2/22 & R\$ 32,44 & 0,0153                      & R\$ 500,00                       & 0,9830        & 1710,90       & R\$ 1.681,87                \\
                11/2/22 & R\$ 33,76 & 0,0407                      &                                  & 1,0230        & 1710,90       & R\$ 1.750,30                \\
                14/2/22 & R\$ 33,00 & -0,0225                     &                                  & 1,0000        & 1710,90       & R\$ 1.710,90                \\
                15/2/22 & R\$ 32,48 & -0,0158                     &                                  & 0,9842        & 1710,90       & R\$ 1.683,94                \\
                16/2/22 & R\$ 32,93 & 0,0139                      & R\$ 100,00                       & 0,9979        & 1811,11       & R\$ 1.807,27                \\
                17/2/22 & R\$ 32,80 & -0,0039                     & R\$ 100,00                       & 0,9939        & 1911,72       & R\$ 1.900,14                \\
                18/2/22 & R\$ 33,00 & 0,0061                      & \textcolor{purple}{-R\$ 700,00}  & 1,0000        & 1211,72       & R\$ 1.211,72                \\
                21/2/22 & R\$ 33,85 & 0,0258                      &                                  & 1,0258        & 1211,72       & R\$ 1.242,93                \\
                22/2/22 & R\$ 33,74 & -0,0032                     &                                  & 1,0224        & 1211,72       & R\$ 1.238,89                \\
                23/2/22 & R\$ 34,22 & 0,0142                      &                                  & 1,0370        & 1211,72       & R\$ 1.256,52                \\
                24/2/22 & R\$ 33,39 & -0,0243                     &                                  & 1,0118        & 1211,72       & R\$ 1.226,04                \\
                25/2/22 & R\$ 34,00 & 0,0183                      &                                  & 1,0303        & 1211,72       & R\$ 1.248,44                \\\hline
            \end{tabular}
            \caption{\footnotesize Exemplo de Cotização}
        \end{center}
    \end{tiny}
\end{table}

Note as movimentações (depósitos e retiradas) afetam a quantidade de cotas e consequentemente o patrimônio, mas não afetam o valor da cota em si, pois não altera a rentabilidade.
\subsection{Descrição  da Estratégia Portfólio Eficiente} \label{Sec:DescEstrategia-3.2}
Na estratégia proposta nesta seção, analisamos os ativos BOVA11 e IVVB11, juntamente com um ativo considerado livre de risco. Embora tenhamos baseado o desenvolvimento da nossa teoria no ativo livre de risco, culminando no Corolário~\ref{cr:opt}, a tentativa inicial de formular uma estratégia consistia em investir exclusivamente em ativos de risco, seguindo o vetor de alocação proposto pelo corolário mencionado. Contudo, essa abordagem revela-se puramente teórica, uma vez que a presença de coordenadas negativas no vetor de alocação implica na adoção de "short selling". Tal prática, além de inviável para o investidor individual devido a limitações regulatórias e ao elevado aumento de risco associado, mostra-se impraticável na realidade. Diante dessas limitações, a estratégia   utilizada  se desdobra da seguinte forma:
\begin{itemize}
    \item[(AF)] {\it Aporte Renda Fixa}: Se o vetor de alocação sugerido ou percentagens sugeridas (PS) pela equação $x(r^*)$ \ref{eq:opttau} contiver alguma coordenada negativa, o aporte é  integralmente direcionado para renda fixa.
    \item[(AV)] {\it Aporte Renda Variável}: Se, por outro lado, o vetor de alocação apresentar todas as coordenadas positivas, então o capital total — abrangendo novos aportes e o acumulado em renda fixa — deve ser distribuído segundo as proporções indicadas pelo vetor de alocação $x(r^*)$ \ref{eq:opttau}.
\end{itemize}
Essa estratégia será chamada de {\it Estratégia do Portfólio Eficiente}. Para uma compreensão mais clara dos métodos de aporte previamente descritos, temos uma sequência de tabelas a seguir que representam a evolução  do capital durante o tempo investido nessa estratégia. É importante notar que existem duas abordagens principais para a distribuição de aportes: o {\it aporte em renda fixa} e o {\it aporte em renda variável}. No aporte em renda fixa, os recursos são integralmente destinados ao CDI, considerando-o como um ativo de reserva livre de risco. Por outro lado, o aporte em renda variável consiste em realocar o valor total disponível — incluindo os montantes anteriormente investidos no CDI — para o BOVA11 e o IVVB11, seguindo as proporções recomendadas pela estratégia de investimento em questão. Esta estruturação facilita o entendimento dos diferentes impactos e execuções das estratégias de aporte.
Para facilitar a compreensão dessa estratégia seguem as novas notações que serão utilizadas a seguir:
\begin{itemize}
    \item S-BOVA11:= Valor investido atualmente em BOVA11;
    \item S-IVVB11:= Valor investido atualmente em IVVB11;
    \item S-CDI:= Valor investido atualmente no CDI;
    \item S-TOTAL:= Valor total do patrimonio
\end{itemize}

\begin{table}[H]
    \begin{tiny}
        \begin{center}
            \begin{tabular}{cccccccc}
                Data  & PS-BOVA11 & PS-IVVB11 & S-BOVA11     & S-IVVB11     & S-CDI      & S-Total      & Modo de Distribuição  \\\hline
                1/18  & $45,09\%$ & $54,91\%$ & R\$ 492,00   & R\$ 471,50   & R\$ 89,29  & R\$ 1.052,79 & Aporte Renda Variável \\
                2/18  & $51,13\%$ & $48,87\%$ & R\$ 742,50   & R\$ 648,20   & R\$ 54,10  & R\$ 1.444,80 & Aporte Renda Variável \\
                3/18  & $41,97\%$ & $58,03\%$ & R\$ 741,87   & R\$ 1.008,26 & R\$ 88,18  & R\$ 1.838,31 & Aporte Renda Variável \\
                4/18  & $49,28\%$ & $50,72\%$ & R\$ 1.081,34 & R\$ 1.171,20 & R\$ 71,21  & R\$ 2.323,75 & Aporte Renda Variável \\
                5/18  & $53,11\%$ & $46,89\%$ & R\$ 1.259,19 & R\$ 1.288,44 & R\$ 145,10 & R\$ 2.692,73 & Aporte Renda Variável \\
                6/18  & $48,19\%$ & $51,81\%$ & R\$ 1.407,00 & R\$ 1.554,00 & R\$ 104,41 & R\$ 3.065,41 & Aporte Renda Variável \\
                7/18  & $31,49\%$ & $68,51\%$ & R\$ 1.146,75 & R\$ 2.343,18 & R\$ 69,84  & R\$ 3.559,77 & Aporte Renda Variável \\
                8/18  & $24,0\%$  & $76,0\%$  & R\$ 889,44   & R\$ 3.255,20 & R\$ 141,29 & R\$ 4.285,93 & Aporte Renda Variável \\
                9/18  & $29,51\%$ & $70,49\%$ & R\$ 1.455,59 & R\$ 3.226,60 & R\$ 27,09  & R\$ 4.709,28 & Aporte Renda Variável \\
                10/18 & $12,21\%$ & $87,79\%$ & R\$ 672,64   & R\$ 3.850,20 & R\$ 19,40  & R\$ 4.542,24 & Aporte Renda Variável \\
                11/18 & $13,26\%$ & $86,74\%$ & R\$ 604,59   & R\$ 4.516,00 & R\$ 76,25  & R\$ 5.196,84 & Aporte Renda Variável \\
                12/18 & $36,84\%$ & $63,16\%$ & R\$ 1.945,80 & R\$ 3.183,70 & R\$ 113,71 & R\$ 5.243,21 & Aporte Renda Variável \\
                \hline
            \end{tabular}
            \caption{\footnotesize Testes da Estratégia do Portfólio Eficiente em 2018} \label{tb:ET22018}
        \end{center}
    \end{tiny}
\end{table}

\begin{table}[H]
    \begin{tiny}
        \begin{center}
            \begin{tabular}{cccccccc}
                Data  & PS-BOVA11  & PS-IVVB11  & S-BOVA11     & S-IVVB11     & S-CDI      & S-Total       & Modo de Distribuição  \\\hline
                1/19  & $39,74\%$  & $60,26\%$  & R\$ 2.350,00 & R\$ 3.449,16 & R\$ 137,61 & R\$ 5.936,77  & Aporte Renda Variável \\
                2/19  & $63,63\%$  & $36,37\%$  & R\$ 3.857,70 & R\$ 2.447,50 & R\$ 87,85  & R\$ 6.393,05  & Aporte Renda Variável \\
                3/19  & $57,65\%$  & $42,35\%$  & R\$ 3.856,02 & R\$ 2.953,75 & R\$ 151,20 & R\$ 6.960,97  & Aporte Renda Variável \\
                4/19  & $35,31\%$  & $64,69\%$  & R\$ 2.600,64 & R\$ 4.914,00 & R\$ 71,81  & R\$ 7.586,45  & Aporte Renda Variável \\
                5/19  & $26,90\%$  & $73,10\%$  & R\$ 2.152,57 & R\$ 5.407,35 & R\$ 73,79  & R\$ 7.633,71  & Aporte Renda Variável \\
                6/19  & $26,94\%$  & $73,06\%$  & R\$ 2.136,42 & R\$ 6.104,70 & R\$ 113,50 & R\$ 8.354,62  & Aporte Renda Variável \\
                7/19  & $68,70\%$  & $31,30\%$  & R\$ 6.074,76 & R\$ 2.664,42 & R\$ 143,28 & R\$ 8.882,46  & Aporte Renda Variável \\
                8/19  & $71,93\%$  & $28,07\%$  & R\$ 6.637,48 & R\$ 2.710,68 & R\$ 75,16  & R\$ 9.423,32  & Aporte Renda Variável \\
                9/19  & $66,16\%$  & $33,84\%$  & R\$ 6.666,66 & R\$ 3.308,00 & R\$ 151,19 & R\$ 10.125,85 & Aporte Renda Variável \\
                10/19 & $110,82\%$ & $-10,82\%$ & R\$ 6.813,18 & R\$ 3.250,00 & R\$ 553,84 & R\$ 10.617,02 & Aporte Renda Fixa     \\
                11/19 & $124,70\%$ & $-24,70\%$ & R\$ 6.887,10 & R\$ 3.590,00 & R\$ 957,28 & R\$ 11.434,38 & Aporte Renda Fixa     \\
                12/19 & $45,77\%$  & $54,23\%$  & R\$ 5.672,73 & R\$ 6.241,95 & R\$ 109,99 & R\$ 12.024,67 & Aporte Renda Variável \\
                \hline
            \end{tabular}
            \caption{\footnotesize Testes da Estratégia do Portfólio Eficiente em 2019} \label{tb:ET22019}
        \end{center}
    \end{tiny}
\end{table}

\begin{table}[H]
    \begin{tiny}
        \begin{center}
            \begin{tabular}{cccccccc}
                Data  & PS-BOVA11   & PS-IVVB11  & S-BOVA11     & S-IVVB11      & S-CDI        & S-Total       & Modo de Distribuição  \\\hline
                1/20  & $35,65\%$   & $64,35\%$  & R\$ 4.247,10 & R\$ 8.441,70  & R\$ 194,35   & R\$ 12.883,15 & Aporte Renda Variável \\
                2/20  & $33,68\%$   & $66,32\%$  & R\$ 4.024,00 & R\$ 8.352,04  & R\$ 188,57   & R\$ 12.564,61 & Aporte Renda Variável \\
                3/20  & $20,16\%$   & $79,84\%$  & R\$ 1.803,10 & R\$ 10.412,64 & R\$ 98,61    & R\$ 12.314,35 & Aporte Renda Variável \\
                4/20  & $-6,09\%$   & $106,09\%$ & R\$ 2.007,46 & R\$ 12.276,00 & R\$ 500,03   & R\$ 14.783,49 & Aporte Renda Fixa     \\
                5/20  & $-275,66\%$ & $375,66\%$ & R\$ 2.187,90 & R\$ 12.646,80 & R\$ 902,15   & R\$ 15.736,85 & Aporte Renda Fixa     \\
                6/20  & $-143,80\%$ & $243,80\%$ & R\$ 2.382,12 & R\$ 13.083,84 & R\$ 1.304,92 & R\$ 16.770,88 & Aporte Renda Fixa     \\
                7/20  & $-101,22\%$ & $201,22\%$ & R\$ 2.581,54 & R\$ 13.224,96 & R\$ 1.708,08 & R\$ 17.514,58 & Aporte Renda Fixa     \\
                8/20  & $-81,01\%$  & $181,01\%$ & R\$ 2.488,20 & R\$ 14.914,80 & R\$ 2.111,45 & R\$ 19.514,45 & Aporte Renda Fixa     \\
                9/20  & $-66,97\%$  & $166,97\%$ & R\$ 2.367,30 & R\$ 14.686,56 & R\$ 2.515,40 & R\$ 19.569,26 & Aporte Renda Fixa     \\
                10/20 & $-58,48\%$  & $158,48\%$ & R\$ 2.357,16 & R\$ 14.587,20 & R\$ 2.919,97 & R\$ 19.864,33 & Aporte Renda Fixa     \\
                11/20 & $-72,43\%$  & $172,43\%$ & R\$ 2.730,00 & R\$ 15.120,00 & R\$ 3.324,94 & R\$ 21.174,94 & Aporte Renda Fixa     \\
                12/20 & $-91,45\%$  & $191,45\%$ & R\$ 2.980,90 & R\$ 15.122,88 & R\$ 3.730,50 & R\$ 21.834,28 & Aporte Renda Fixa     \\
                \hline
            \end{tabular}
            \caption{\footnotesize Testes da Estratégia do Portfólio Eficiente em 2020} \label{tb:ET22020}
        \end{center}
    \end{tiny}
\end{table}

\begin{table}[H]
    \begin{tiny}
        \begin{center}
            \begin{tabular}{cccccccc}
                Data  & PS-BOVA11  & PS-IVVB11  & S-BOVA11      & S-IVVB11      & S-CDI        & S-Total       & Modo de Distribuição  \\\hline
                1/21  & $-52,97\%$ & $152,97\%$ & R\$ 2.874,56  & R\$ 15.840,00 & R\$ 4.136,37 & R\$ 22.850,93 & Aporte Renda Fixa     \\
                2/21  & $-46,25\%$ & $146,25\%$ & R\$ 2.745,34  & R\$ 16.632,00 & R\$ 4.542,47 & R\$ 23.919,81 & Aporte Renda Fixa     \\
                3/21  & $-51,55\%$ & $151,55\%$ & R\$ 2.912,52  & R\$ 17.510,40 & R\$ 4.952,41 & R\$ 25.375,33 & Aporte Renda Fixa     \\
                4/21  & $-38,48\%$ & $138,48\%$ & R\$ 2.974,40  & R\$ 17.784,00 & R\$ 5.363,53 & R\$ 26.121,93 & Aporte Renda Fixa     \\
                5/21  & $41,33\%$  & $58,67\%$  & R\$ 11.518,75 & R\$ 14.756,00 & R\$ 325,95   & R\$ 26.600,70 & Aporte Renda Variável \\
                6/21  & $49,04\%$  & $50,96\%$  & R\$ 13.033,67 & R\$ 13.485,00 & R\$ 184,60   & R\$ 26.703,27 & Aporte Renda Variável \\
                7/21  & $54,95\%$  & $45,05\%$  & R\$ 14.404,53 & R\$ 12.767,85 & R\$ 307,39   & R\$ 27.479,77 & Aporte Renda Variável \\
                8/21  & $47,52\%$  & $52,48\%$  & R\$ 12.776,96 & R\$ 14.732,00 & R\$ 259,14   & R\$ 27.768,10 & Aporte Renda Variável \\
                9/21  & $36,26\%$  & $63,74\%$  & R\$ 9.477,61  & R\$ 17.850,00 & R\$ 222,69   & R\$ 27.550,30 & Aporte Renda Variável \\
                10/21 & $37,15\%$  & $62,85\%$  & R\$ 9.569,28  & R\$ 19.527,00 & R\$ 117,89   & R\$ 29.214,17 & Aporte Renda Variável \\
                11/21 & $33,62\%$  & $66,38\%$  & R\$ 9.638,30  & R\$ 19.357,95 & R\$ 316,73   & R\$ 29.313,00 & Aporte Renda Variável \\
                12/21 & $19,58\%$  & $80,42\%$  & R\$ 5.947,20  & R\$ 24.659,04 & R\$ 320,56   & R\$ 30.926,80 & Aporte Renda Variável \\
                \hline
            \end{tabular}
            \caption{\footnotesize Testes da Estratégia do Portfólio Eficiente em 2021} \label{tb:ET22021}
        \end{center}
    \end{tiny}
\end{table}

\begin{table}[H]
    \begin{tiny}
        \begin{center}
            \begin{tabular}{cccccccc}
                Data  & PS-BOVA11  & PS-IVVB11  & S-BOVA11      & S-IVVB11      & S-CDI        & S-Total       & Modo de Distribuição  \\\hline
                1/22  & $-19,81\%$ & $119,81\%$ & R\$ 6.370,82  & R\$ 21.873,60 & R\$ 725,83   & R\$ 28.970,25 & Aporte Renda Fixa     \\
                2/22  & $-41,26\%$ & $141,26\%$ & R\$ 6.425,10  & R\$ 20.622,00 & R\$ 1.134,33 & R\$ 28.181,43 & Aporte Renda Fixa     \\
                3/22  & $-20,95\%$ & $120,95\%$ & R\$ 6.820,40  & R\$ 19.794,60 & R\$ 1.548,56 & R\$ 28.163,56 & Aporte Renda Fixa     \\
                4/22  & $11,55\%$  & $88,45\%$  & R\$ 2.791,80  & R\$ 24.202,80 & R\$ 132,97   & R\$ 27.127,57 & Aporte Renda Variável \\
                5/22  & $17,15\%$  & $82,85\%$  & R\$ 5.037,46  & R\$ 21.664,50 & R\$ 90,96    & R\$ 26.792,92 & Aporte Renda Variável \\
                6/22  & $48,08\%$  & $51,92\%$  & R\$ 11.590,00 & R\$ 14.085,50 & R\$ 217,95   & R\$ 25.893,45 & Aporte Renda Variável \\
                7/22  & $53,63\%$  & $46,37\%$  & R\$ 14.994,30 & R\$ 12.897,50 & R\$ 115,97   & R\$ 28.007,77 & Aporte Renda Variável \\
                8/22  & $24,52\%$  & $75,48\%$  & R\$ 22.859,28 & R\$ 6.554,00  & R\$ 194,58   & R\$ 29.607,86 & Aporte Renda Variável \\
                9/22  & $28,43\%$  & $71,57\%$  & R\$ 21.611,38 & R\$ 7.899,50  & R\$ 155,19   & R\$ 29.666,07 & Aporte Renda Variável \\
                10/22 & $41,76\%$  & $58,24\%$  & R\$ 12.897,25 & R\$ 18.940,64 & R\$ 200,64   & R\$ 32.038,53 & Aporte Renda Variável \\
                11/22 & $19,28\%$  & $80,72\%$  & R\$ 5.971,90  & R\$ 27.756,75 & R\$ 198,80   & R\$ 33.927,45 & Aporte Renda Variável \\
                12/22 & $-26,19\%$ & $126,19\%$ & R\$ 5.827,25  & R\$ 26.775,00 & R\$ 605,22   & R\$ 33.207,47 & Aporte Renda Fixa     \\
                \hline
            \end{tabular}
            \caption{\footnotesize Testes da Estratégia do Portfólio Eficiente em 2022} \label{tb:ET22022}
        \end{center}
    \end{tiny}
\end{table}

\begin{table}[H]
    \begin{tiny}
        \begin{center}
            \begin{tabular}{cccccccc}
                Data  & PS-BOVA11  & PS-IVVB11  & S-BOVA11      & S-IVVB11      & S-CDI        & S-Total       & Modo de Distribuição  \\\hline
                1/23  & $-5,01\%$  & $105,01\%$ & R\$ 6.041,20  & R\$ 27.155,80 & R\$ 1.016,51 & R\$ 34.213,51 & Aporte Renda Fixa     \\
                2/23  & $26,67\%$  & $73,33\%$  & R\$ 8.597,75  & R\$ 25.223,00 & R\$ 237,18   & R\$ 34.057,93 & Aporte Renda Variável \\
                3/23  & $35,68\%$  & $64,32\%$  & R\$ 11.818,80 & R\$ 22.056,00 & R\$ 270,14   & R\$ 34.144,94 & Aporte Renda Variável \\
                4/23  & $63,86\%$  & $36,14\%$  & R\$ 22.601,60 & R\$ 12.406,50 & R\$ 137,73   & R\$ 35.145,83 & Aporte Renda Variável \\
                5/23  & $64,66\%$  & $35,34\%$  & R\$ 24.008,36 & R\$ 12.362,25 & R\$ 278,04   & R\$ 36.648,65 & Aporte Renda Variável \\
                6/23  & $58,10\%$  & $41,90\%$  & R\$ 23.176,51 & R\$ 15.761,75 & R\$ 199,90   & R\$ 39.138,16 & Aporte Renda Variável \\
                7/23  & $74,12\%$  & $25,88\%$  & R\$ 30.077,25 & R\$ 10.312,26 & R\$ 261,42   & R\$ 40.650,93 & Aporte Renda Variável \\
                8/23  & $186,89\%$ & $-86,89\%$ & R\$ 28.639,05 & R\$ 10.648,95 & R\$ 668,94   & R\$ 39.956,94 & Aporte Renda Fixa     \\
                9/23  & $-57,27\%$ & $157,27\%$ & R\$ 28.853,25 & R\$ 10.267,97 & R\$ 1.079,34 & R\$ 40.200,56 & Aporte Renda Fixa     \\
                10/23 & $39,82\%$  & $60,18\%$  & R\$ 23.787,54 & R\$ 15.444,00 & R\$ 179,85   & R\$ 39.411,39 & Aporte Renda Variável \\
                11/23 & $98,93\%$  & $1,07\%$   & R\$ 43.990,92 & R\$ 249,70    & R\$ 246,33   & R\$ 44.486,95 & Aporte Renda Variável \\
                12/23 & $62,77\%$  & $37,23\%$  & R\$ 29.598,53 & R\$ 17.259,87 & R\$ 232,48   & R\$ 47.090,88 & Aporte Renda Variável \\
                \hline
            \end{tabular}
            \caption{\footnotesize Testes da Estratégia do Portfólio Eficiente em 2023} \label{tb:ET22023}
        \end{center}
    \end{tiny}
\end{table}

Esta estratégia procura equilibrar a segurança proporcionada pelo investimento em renda fixa com a potencial otimização dos retornos, aproveitando as condições favoráveis do mercado conforme sugerido pelo vetor de alocação. Analisando a sequência de tabelas anteriores, podemos ver claramente que em momentos de grande volatilidade do mercado a Estratégia do Portfólio Eficiente diminuiu seu risco usando o CDI como meio de investir durante as instabilidades e, em seguida, redirecionou todo o capital para BOVA11 e IVVB11, conseguindo um resultado muito significativo e que supera em muito um investimento puramente no CDI.
\subsection{Descrição  da  Estratégia Conservadora} \label{Sec:DescEstrategia-3.0}
Na estratégia apresentada nesta seção, levamos em consideração os ativos BOVA11 e IVVB11, bem como um ativo livre de risco.  Para isso, na estratégia apresentada nesta seção, consideramos os ativos BOVA11 e IVVB11, juntamente com um ativo livre de risco, o qual tomamos como sendo o CDI.  Inicialmente, a Estratégia  Mínimo Risco é usada para encontrar o portfólio de menor risco, enquanto a Estratégia do Portfólio Eficiente como o nome sugere busca encontrar o portfólio eficiente. Em seguida, combinamos essas duas estratégias, alocando parte do capital no CDI, resultando em uma estratégia que visa um portfólio com o mesmo risco do portfólio obtido na Estratégia  Mínimo Risco e com a eficiência do portfólio obtido na Estratégia do Portfólio Eficiente. Agora, passaremos à descrição detalhada da estratégia a ser utilizada. Seja ${\hat x}$ o vetor de alocação de capital para  que será considerado na definição da  {\it Estratégia Conservadora}, o qual é  dado  por:
\begin{equation} \label{eq:vac3 }
    {\hat x}:=(x, x_{f})\in  {\mathbb R}^{2}\times  {\mathbb R}
\end{equation}
onde $x=(x_1, x_2) \in {\mathbb R}^{2}$ é o vetor de alocação de capital nos ativos BOVA11 e IVVB11,  e $x_{f}$ é a quantidade alocada no CDI. Então a  Estratégia Conservadora  será definida considerando o seguinte vetor de alocação de capital:
\begin{equation}   \label{eq:vac3e}
    x=\frac{\sigma_{min}}{\sigma_*}x^*, \qquad \quad x_{f}=1-\frac{\sigma_{min}}{\sigma_*}(x^*_1+x^*_2)=1-\frac{\sigma_{min}}{\sigma_*}.
\end{equation}
onde  $x^*$ é   vetor de alocação e $\sigma_*$ é o risco  do portfólio  eficiente dado pela Proposição~\ref{eq:ep}. Aqui, $\sigma_{min}$ é o risco mínimo alcançado na Estratégia  Mínimo Risco. A equação \eqref{eq:vac3e} indica como distribuir os recursos entre os ativos de acordo com seus riscos para alcançar a meta da Estratégia Conservadora. Note que, combinando o Lema~\ref{le:fxmc} com as equações \eqref{eq:sdfr3} e \eqref{eq:vac3e}, concluímos que
$$
    {\hat \sigma} = \sigma_{min},
$$
ou seja, o portfólio $P_{\hat x}$ tem o mesmo risco do portfólio da Estratégia  Mínimo Risco. Isso evidencia o sucesso da estratégia em alcançar o objetivo desejado.

Avançamos para descrever detalhadamente  como serão feitos os aportes   na estratégia que combina esses três ativos. Com o objetivo de um investimento de longo prazo, optamos pela seguinte abordagem:

\begin{enumerate}
    \item[(AP)] {\it Aportes Periódicos}: Aumentar o tamanho do portfólio por meio de aportes periódicos de acordo com \eqref{eq:vac3e}.
\end{enumerate}
Neste caso,  os aportes periódicos são feitos da seguinte forma:

\begin{itemize}
    \item[(AF)] {\it Aportes Renda Fixa}: Se o vetor de alocação sugerido ou percentagens sugeridas (PS) pela Proposição~\ref{eq:opttau}  contiver alguma coordenada negativa, o aporte é  integralmente direcionado para renda fixa.
    \item[(RT)] {\it Redistribuição Total}: Se, por outro lado, o vetor de alocação apresentar todas as coordenadas positivas, então o capital total juntamente com novos aportes deve ser distribuído segundo \ref{eq:vac3e}.
\end{itemize}

\begin{remark}
    É importante observar alguns detalhes técnicos que podem ocorrer na implementação prática da estratégia\footnote{Ressaltamos que não estamos necessariamente indicando a implementação desta estratégia. São apenas considerações técnicas}:

    \begin{itemize}
        \item Como não é possível comprar frações de um ativo de risco, as sobras de recursos a cada aporte  devem ser direcionadas ao CDI.
    \end{itemize}
\end{remark}

As tabelas a seguir ilustram a evolução do patrimônio mês a mês, detalhando os retornos obtidos e a redistribuição dos aportes ao longo do período analisado. Estas tabelas abrangem o periodo de janeiro de 2018 até o último dia de 2023 e são fundamentais para entender o desempenho da estratégia conservadora, permitindo uma análise comparativa com outras estratégias de investimento em termos de risco e retorno.

\begin{table}[H]
    \begin{tiny}
        \begin{center}
            \begin{tabular}{cccccccccc}
                Data  & PS-BOVA11 & PS-IVVB11 & PS-CDI    & S-BOVA11     & S-IVVB11     & S-CDI      & S-Total      & Modo de Distribuição \\\hline
                01/18 & $44,94\%$ & $54,73\%$ & $0,33\%$  & R\$ 490,69   & R\$ 560,27   & R\$ 3,29   & R\$ 1.054,26 & Redistribuição Total \\
                02/18 & $49,08\%$ & $46,92\%$ & $4,00\%$  & R\$ 660,00   & R\$ 648,20   & R\$ 138,10 & R\$ 1.446,30 & Redistribuição Total \\
                03/18 & $41,95\%$ & $58,00\%$ & $0,05\%$  & R\$ 741,87   & R\$ 1.008,26 & R\$ 89,69  & R\$ 1.839,82 & Redistribuição Total \\
                04/18 & $48,06\%$ & $49,46\%$ & $2,48\%$  & R\$ 1.081,34 & R\$ 1.171,20 & R\$ 72,73  & R\$ 2.325,27 & Redistribuição Total \\
                05/18 & $51,21\%$ & $45,21\%$ & $3,58\%$  & R\$ 1.259,19 & R\$ 1.288,44 & R\$ 146,63 & R\$ 2.694,26 & Redistribuição Total \\
                06/18 & $47,89\%$ & $51,50\%$ & $0,61\%$  & R\$ 1.407,00 & R\$ 1.554,00 & R\$ 105,95 & R\$ 3.066,95 & Redistribuição Total \\
                07/18 & $31,22\%$ & $67,91\%$ & $0,87\%$  & R\$ 1.146,75 & R\$ 2.343,18 & R\$ 71,39  & R\$ 3.561,32 & Redistribuição Total \\
                08/18 & $22,55\%$ & $71,41\%$ & $6,04\%$  & R\$ 815,32   & R\$ 3.130,00 & R\$ 332,12 & R\$ 4.277,44 & Redistribuição Total \\
                09/18 & $28,73\%$ & $68,64\%$ & $2,63\%$  & R\$ 1.378,98 & R\$ 3.102,50 & R\$ 218,86 & R\$ 4.700,34 & Redistribuição Total \\
                10/18 & $10,41\%$ & $74,80\%$ & $14,79\%$ & R\$ 504,48   & R\$ 3.208,50 & R\$ 912,42 & R\$ 4.625,40 & Redistribuição Total \\
                11/18 & $11,13\%$ & $72,85\%$ & $16,01\%$ & R\$ 518,22   & R\$ 3.838,60 & R\$ 888,03 & R\$ 5.244,85 & Redistribuição Total \\
                12/18 & $36,79\%$ & $63,06\%$ & $0,15\%$  & R\$ 2.030,40 & R\$ 3.183,70 & R\$ 72,01  & R\$ 5.286,11 & Redistribuição Total \\
                \hline
            \end{tabular}
            \caption{\footnotesize Testes da Estratégia Conservadora em 2018} \label{tb:ET32018}
        \end{center}
    \end{tiny}
\end{table}

\begin{table}[H]
    \begin{tiny}
        \begin{center}
            \begin{tabular}{cccccccccc}
                Data  & PS-BOVA11 & PS-IVVB11  & PS-CDI    & S-BOVA11     & S-IVVB11     & S-CDI        & S-Total       & Modo de Distribuição \\\hline
                01/19 & $39,71\%$ & $60,23\%$  & $0,06\%$  & R\$ 2.350,00 & R\$ 3.553,68 & R\$ 82,99    & R\$ 5.986,67  & Redistribuição Total \\
                02/19 & $58,30\%$ & $33,32\%$  & $8,39\%$  & R\$ 3.582,15 & R\$ 2.225,00 & R\$ 633,69   & R\$ 6.440,84  & Redistribuição Total \\
                03/19 & $55,40\%$ & $40,69\%$  & $3,91\%$  & R\$ 3.764,21 & R\$ 2.835,60 & R\$ 403,73   & R\$ 7.003,54  & Redistribuição Total \\
                04/19 & $34,64\%$ & $63,45\%$  & $1,90\%$  & R\$ 2.507,76 & R\$ 4.791,15 & R\$ 325,70   & R\$ 7.624,61  & Redistribuição Total \\
                05/19 & $25,39\%$ & $68,99\%$  & $5,62\%$  & R\$ 2.058,98 & R\$ 5.062,20 & R\$ 575,85   & R\$ 7.697,03  & Redistribuição Total \\
                06/19 & $25,50\%$ & $69,16\%$  & $5,35\%$  & R\$ 2.039,31 & R\$ 5.745,60 & R\$ 618,03   & R\$ 8.402,94  & Redistribuição Total \\
                07/19 & $58,76\%$ & $26,77\%$  & $14,46\%$ & R\$ 5.192,94 & R\$ 2.301,09 & R\$ 1.434,51 & R\$ 8.928,54  & Redistribuição Total \\
                08/19 & $58,27\%$ & $22,74\%$  & $18,99\%$ & R\$ 5.368,55 & R\$ 2.194,36 & R\$ 1.891,52 & R\$ 9.454,43  & Redistribuição Total \\
                09/19 & $57,39\%$ & $29,36\%$  & $13,25\%$ & R\$ 5.858,58 & R\$ 2.911,04 & R\$ 1.360,20 & R\$ 10.129,82 & Redistribuição Total \\
                10/19 & $62,98\%$ & $-6,15\%$  & $43,17\%$ & R\$ 5.987,34 & R\$ 2.860,00 & R\$ 1.768,64 & R\$ 10.617,02 & Aporte Renda Fixa    \\
                11/19 & $64,18\%$ & $-12,71\%$ & $48,53\%$ & R\$ 6.052,30 & R\$ 3.159,20 & R\$ 2.176,47 & R\$ 11.434,38 & Aporte Renda Fixa    \\
                12/19 & $45,76\%$ & $54,21\%$  & $0,03\%$  & R\$ 5.672,73 & R\$ 6.241,95 & R\$ 67,78    & R\$ 11.982,46 & Redistribuição Total \\
                \hline
            \end{tabular}
            \caption{\footnotesize Testes da Estratégia Conservadora em 2019} \label{tb:ET32019}
        \end{center}
    \end{tiny}
\end{table}

\begin{table}[H]
    \begin{tiny}
        \begin{center}
            \begin{tabular}{cccccccccc}
                Data  & PS-BOVA11  & PS-IVVB11  & PS-CDI    & S-BOVA11     & S-IVVB11      & S-CDI        & S-Total       & Modo de Distribuição \\\hline
                01/20 & $35,27\%$  & $63,65\%$  & $1,08\%$  & R\$ 4.138,20 & R\$ 8.441,70  & R\$ 266,63   & R\$ 12.846,53 & Redistribuição Total \\
                02/20 & $33,19\%$  & $65,34\%$  & $1,48\%$  & R\$ 3.923,40 & R\$ 8.210,48  & R\$ 409,48   & R\$ 12.543,36 & Redistribuição Total \\
                03/20 & $18,73\%$  & $74,19\%$  & $7,08\%$  & R\$ 1.664,40 & R\$ 9.689,54  & R\$ 1.020,17 & R\$ 12.374,11 & Redistribuição Total \\
                04/20 & $-4,83\%$  & $84,07\%$  & $20,76\%$ & R\$ 1.853,04 & R\$ 11.423,50 & R\$ 1.424,21 & R\$ 14.783,49 & Aporte Renda Fixa    \\
                05/20 & $-95,26\%$ & $129,81\%$ & $65,44\%$ & R\$ 2.019,60 & R\$ 11.768,55 & R\$ 1.828,52 & R\$ 15.736,85 & Aporte Renda Fixa    \\
                06/20 & $-76,58\%$ & $129,83\%$ & $46,75\%$ & R\$ 2.198,88 & R\$ 12.175,24 & R\$ 2.233,25 & R\$ 16.770,88 & Aporte Renda Fixa    \\
                07/20 & $-64,34\%$ & $127,91\%$ & $36,43\%$ & R\$ 2.382,96 & R\$ 12.306,56 & R\$ 2.638,14 & R\$ 17.514,58 & Aporte Renda Fixa    \\
                08/20 & $-55,76\%$ & $124,58\%$ & $31,17\%$ & R\$ 2.296,80 & R\$ 13.879,05 & R\$ 3.043,00 & R\$ 19.514,45 & Aporte Renda Fixa    \\
                09/20 & $-49,53\%$ & $123,47\%$ & $26,05\%$ & R\$ 2.185,20 & R\$ 13.666,66 & R\$ 3.448,40 & R\$ 19.569,26 & Aporte Renda Fixa    \\
                10/20 & $-44,13\%$ & $119,58\%$ & $24,55\%$ & R\$ 2.175,84 & R\$ 13.574,20 & R\$ 3.854,44 & R\$ 19.864,33 & Aporte Renda Fixa    \\
                11/20 & $-49,12\%$ & $116,94\%$ & $32,18\%$ & R\$ 2.520,00 & R\$ 14.070,00 & R\$ 4.260,80 & R\$ 21.174,94 & Aporte Renda Fixa    \\
                12/20 & $-55,84\%$ & $116,90\%$ & $38,94\%$ & R\$ 2.751,60 & R\$ 14.072,68 & R\$ 4.667,77 & R\$ 21.834,28 & Aporte Renda Fixa    \\
                \hline
            \end{tabular}
            \caption{\footnotesize Testes da Estratégia Conservadora em 2020} \label{tb:ET32020}
        \end{center}
    \end{tiny}
\end{table}

\begin{table}[H]
    \begin{tiny}
        \begin{center}
            \begin{tabular}{cccccccccc}
                Data  & PS-BOVA11  & PS-IVVB11  & PS-CDI    & S-BOVA11      & S-IVVB11      & S-CDI        & S-Total       & Modo de Distribuição \\\hline
                01/21 & $-40,47\%$ & $116,85\%$ & $23,61\%$ & R\$ 2.653,44  & R\$ 14.740,00 & R\$ 5.074,97 & R\$ 22.850,93 & Aporte Renda Fixa    \\
                02/21 & $-36,56\%$ & $115,60\%$ & $20,95\%$ & R\$ 2.534,16  & R\$ 15.477,00 & R\$ 5.482,33 & R\$ 23.919,81 & Aporte Renda Fixa    \\
                03/21 & $-39,09\%$ & $114,92\%$ & $24,17\%$ & R\$ 2.688,48  & R\$ 16.294,40 & R\$ 5.894,16 & R\$ 25.375,33 & Aporte Renda Fixa    \\
                04/21 & $-31,12\%$ & $111,98\%$ & $19,12\%$ & R\$ 2.745,60  & R\$ 16.549,00 & R\$ 6.307,24 & R\$ 26.121,93 & Aporte Renda Fixa    \\
                05/21 & $41,31\%$  & $58,64\%$  & $0,05\%$  & R\$ 11.276,25 & R\$ 14.518,00 & R\$ 280,07   & R\$ 26.074,32 & Redistribuição Total \\
                06/21 & $49,01\%$  & $50,93\%$  & $0,06\%$  & R\$ 12.790,05 & R\$ 13.252,50 & R\$ 138,57   & R\$ 26.181,12 & Redistribuição Total \\
                07/21 & $54,38\%$  & $44,59\%$  & $1,03\%$  & R\$ 14.053,20 & R\$ 12.517,50 & R\$ 382,34   & R\$ 26.953,04 & Redistribuição Total \\
                08/21 & $47,51\%$  & $52,48\%$  & $0,00\%$  & R\$ 12.548,80 & R\$ 14.478,00 & R\$ 216,32   & R\$ 27.243,12 & Redistribuição Total \\
                09/21 & $35,46\%$  & $62,32\%$  & $2,22\%$  & R\$ 9.158,14  & R\$ 17.085,00 & R\$ 810,23   & R\$ 27.053,37 & Redistribuição Total \\
                10/21 & $36,52\%$  & $61,79\%$  & $1,69\%$  & R\$ 9.170,56  & R\$ 18.961,00 & R\$ 562,20   & R\$ 28.693,76 & Redistribuição Total \\
                11/21 & $33,17\%$  & $65,49\%$  & $1,34\%$  & R\$ 9.343,25  & R\$ 18.796,85 & R\$ 661,66   & R\$ 28.801,76 & Redistribuição Total \\
                12/21 & $17,54\%$  & $72,04\%$  & $10,42\%$ & R\$ 5.241,60  & R\$ 21.723,44 & R\$ 3.306,68 & R\$ 30.271,72 & Redistribuição Total \\
                \hline
            \end{tabular}
            \caption{\footnotesize Testes da Estratégia Conservadora em 2021} \label{tb:ET32021}
        \end{center}
    \end{tiny}
\end{table}

\begin{table}[H]
    \begin{tiny}
        \begin{center}
            \begin{tabular}{cccccccccc}
                Data  & PS-BOVA11  & PS-IVVB11 & PS-CDI    & S-BOVA11      & S-IVVB11      & S-CDI        & S-Total       & Modo de Distribuição \\\hline
                01/22 & $-11,54\%$ & $69,81\%$ & $41,74\%$ & R\$ 5.614,96  & R\$ 19.269,60 & R\$ 3.733,83 & R\$ 28.970,25 & Aporte Renda Fixa    \\
                02/22 & $-20,00\%$ & $68,49\%$ & $51,51\%$ & R\$ 5.662,80  & R\$ 18.167,00 & R\$ 4.165,04 & R\$ 28.181,43 & Aporte Renda Fixa    \\
                03/22 & $-12,02\%$ & $69,41\%$ & $42,61\%$ & R\$ 6.011,20  & R\$ 17.438,10 & R\$ 4.607,36 & R\$ 28.163,56 & Aporte Renda Fixa    \\
                04/22 & $9,11\%$   & $69,71\%$ & $21,19\%$ & R\$ 2.171,40  & R\$ 18.824,40 & R\$ 6.375,43 & R\$ 27.371,23 & Redistribuição Total \\
                05/22 & $14,59\%$  & $70,49\%$ & $14,92\%$ & R\$ 4.287,20  & R\$ 18.447,00 & R\$ 4.411,31 & R\$ 27.145,51 & Redistribuição Total \\
                06/22 & $47,99\%$  & $51,82\%$ & $0,19\%$  & R\$ 11.685,00 & R\$ 14.302,20 & R\$ 220,60   & R\$ 26.207,80 & Redistribuição Total \\
                07/22 & $52,93\%$  & $45,77\%$ & $1,30\%$  & R\$ 14.994,30 & R\$ 12.897,50 & R\$ 444,62   & R\$ 28.336,42 & Redistribuição Total \\
                08/22 & $66,87\%$  & $21,72\%$ & $11,41\%$ & R\$ 20.531,02 & R\$ 5.876,00  & R\$ 3.430,97 & R\$ 29.837,99 & Redistribuição Total \\
                09/22 & $66,02\%$  & $26,22\%$ & $7,76\%$  & R\$ 20.014,48 & R\$ 7.259,00  & R\$ 2.673,81 & R\$ 29.947,29 & Redistribuição Total \\
                10/22 & $41,37\%$  & $57,69\%$ & $0,94\%$  & R\$ 12.897,25 & R\$ 18.940,64 & R\$ 409,22   & R\$ 32.247,11 & Redistribuição Total \\
                11/22 & $16,70\%$  & $69,91\%$ & $13,38\%$ & R\$ 5.211,84  & R\$ 24.258,00 & R\$ 4.500,87 & R\$ 33.970,71 & Redistribuição Total \\
                12/22 & $-15,16\%$ & $73,04\%$ & $42,12\%$ & R\$ 5.085,60  & R\$ 23.400,00 & R\$ 4.953,41 & R\$ 33.207,47 & Aporte Renda Fixa    \\
                \hline
            \end{tabular}
            \caption{\footnotesize Testes da Estratégia Conservadora em 2022} \label{tb:ET32022}
        \end{center}
    \end{tiny}
\end{table}

\begin{table}[H]
    \begin{tiny}
        \begin{center}
            \begin{tabular}{cccccccccc}
                Data  & PS-BOVA11  & PS-IVVB11  & PS-CDI    & S-BOVA11      & S-IVVB11      & S-CDI         & S-Total       & Modo de Distribuição \\\hline
                01/23 & $-3,20\%$  & $67,04\%$  & $36,16\%$ & R\$ 5.272,32  & R\$ 23.732,80 & R\$ 5.413,55  & R\$ 34.213,51 & Aporte Renda Fixa    \\
                02/23 & $22,92\%$  & $63,04\%$  & $14,03\%$ & R\$ 7.383,95  & R\$ 21.783,50 & R\$ 5.223,82  & R\$ 34.391,27 & Redistribuição Total \\
                03/23 & $33,10\%$  & $59,67\%$  & $7,22\%$  & R\$ 11.030,88 & R\$ 20.677,50 & R\$ 2.845,46  & R\$ 34.553,84 & Redistribuição Total \\
                04/23 & $62,44\%$  & $35,34\%$  & $2,22\%$  & R\$ 22.399,80 & R\$ 12.176,75 & R\$ 978,71    & R\$ 35.555,26 & Redistribuição Total \\
                05/23 & $62,89\%$  & $34,37\%$  & $2,74\%$  & R\$ 23.693,84 & R\$ 12.362,25 & R\$ 996,31    & R\$ 37.052,40 & Redistribuição Total \\
                06/23 & $57,58\%$  & $41,53\%$  & $0,89\%$  & R\$ 23.290,68 & R\$ 15.761,75 & R\$ 492,85    & R\$ 39.545,28 & Redistribuição Total \\
                07/23 & $65,61\%$  & $22,91\%$  & $11,48\%$ & R\$ 26.892,60 & R\$ 9.352,98  & R\$ 4.796,85  & R\$ 41.042,43 & Redistribuição Total \\
                08/23 & $51,76\%$  & $-24,07\%$ & $72,30\%$ & R\$ 25.606,68 & R\$ 9.658,35  & R\$ 5.255,97  & R\$ 39.956,94 & Aporte Renda Fixa    \\
                09/23 & $-19,16\%$ & $52,61\%$  & $66,55\%$ & R\$ 25.798,20 & R\$ 9.312,81  & R\$ 5.710,99  & R\$ 40.200,56 & Aporte Renda Fixa    \\
                10/23 & $57,84\%$  & $38,28\%$  & $3,88\%$  & R\$ 23.239,44 & R\$ 15.210,00 & R\$ 1.653,06  & R\$ 40.102,50 & Redistribuição Total \\
                11/23 & $59,80\%$  & $0,65\%$   & $39,56\%$ & R\$ 27.061,83 & R\$ 249,70    & R\$ 16.359,01 & R\$ 43.670,54 & Redistribuição Total \\
                12/23 & $59,19\%$  & $35,11\%$  & $5,70\%$  & R\$ 27.381,90 & R\$ 15.971,82 & R\$ 2.723,60  & R\$ 46.077,32 & Redistribuição Total \\
                \hline
            \end{tabular}
            \caption{\footnotesize Testes da Estratégia Conservadora em 2023} \label{tb:ET32023}
        \end{center}
    \end{tiny}
\end{table}

Podemos constatar que claramente a Estratégia Conservadora acabou sendo um pouco inferior a Estratégia do Portfólio Eficiente mas ainda teve um desempenho consideravel e acima do CDI, um fato interessante de observar é que apartir de dezembro de 2021 até o final dos nossos investimentos a estratégia optou por fazer todos seus aportes no CDI e isso se deve ao fato do aumento consideravel do retorno do CDI nesse periodo em contrapartida com o valor do BOVA11 e IVVB11, na proxima seção apresentaremos mais detalhes sobre os resultados obtidos em todas as estratégias apresentadas.


\subsection{Análise da Rentabilidade}

Vamos agora analisar a rentabilidade das estratégias, para isso, inicialmente, vamos utilizar a evolução do patrimônio nas estratégias apresentadas.

Ao desenvolver as estratégias nosso objetivo era alcançar, a longo prazo, um retorno superior ao CDI  aceitando um risco ligeiramente maior.

Inicialmente, a Estratégia do Portfólio Eficiente utilizava o CDI apenas para o cálculo das porcentagens, no entanto, enfrentávamos alta volatilidade com a realização de operaçoes de venda a descoberto. Para mitigar isso, passamos a incorporar o CDI de maneira ativa como um ativo de reserva, o que nos ajudou a evitar essas vendas.

Por sua vez, a  Estratégia Conservadora incorpora cálculos matemáticos que já preveem o investimento direto no CDI, tornando-o uma parte ativa da estratégia e não apenas um meio para evitar as vendas a descoberto. A essência da Estratégia conservadora é utilizar a distribuição da Estratégia Eficiente mantendo o risco próximo do risco da Estratégia de Minimo Risco.

A seguir, apresentamos na figura \ref{fig:PatrimonioTotal} do patrimônio total ao longo do período máximo analisado, que vai de 01 dejaneiro de 2018 até 31 de dezembro de 2023 tendo aporte inicial de R\$ $1.000,00$ em janeiro de 2018 e aportes mensais de R\$ $400,00$.

\begin{figure}[H]
    \centering
    \includegraphics[width=0.9\linewidth]{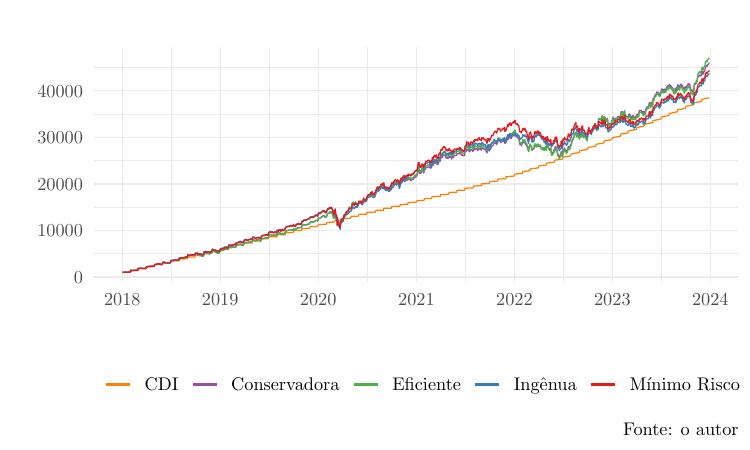}\\
    \caption{ \footnotesize Gráfico do Patrimônio Final}
    \label{fig:PatrimonioTotal}
\end{figure}

Nessa figura fica evidente que todas as estratégias propostas superaram o CDI, sendo o maior patrimônio alcaçado pela Estratégia do Portfólio Eficiente. Como discutido anteriormente todas as estratégias propostas foram feitas com objetivo de um investimento a longo prazo.

Durante o periodo de 2018 até o final de 2023 ocorreu uma grande ocilaçao no CDI que passou de cerca de 2\% para mais de 13\% de retorno. Ainda assim o investimento puramente no CDI foi a estratégia que resultou no menor patrimônio.

Na Tabela~\ref{tab:juros} e Figura~\ref{fig:grafju}  ilustramos  o percentual entre capital referente aos juros e o aportado nas estratégias durante o período de seis  anos, o que nos possibilitará no decorrer desse estudo analisar melhor a importancia do investimento a longo prazo.

\begin{table}[h]
    \begin{footnotesize}
        \begin{center}
            \begin{tabular}{lrrcc}

                \textbf{Ativos }                  & \textbf{Capital Aportado} & \textbf{Capital Final} & \textbf{Percentual aportes} & \textbf{Percentual juros} \\
                \hline\rule{-3pt}{1.1\normalbaselineskip}
                CDI                               & R\$ $29.400,00$           & R\$ $38.489,70$        & $76,38\%$                   & $23,62\%$                 \\

                Estratégia Ingênua                & R\$ $29.400,00$           & R\$ $43.801,49$        & $67,12\%$                   & $32,88\%$                 \\

                Estratégia de Mínimo Risco        & R\$ $29.400,00$           & R\$ $44.382,68$        & $66,24\%$                   & $33,76\%$                 \\
                Estratégia Conservadora           & R\$ $29.400,00$           & R\$ $46.077,32$        & $63,81\%$                   & $36,19\%$                 \\
                Estratégia do Portfólio Eficiente & R\$ $29.400,00$           & R\$ $47.090,88$        & $62,43\%$                   & $37,57\%$
                \\
                \hline
            \end{tabular}
            \caption{\footnotesize Valor do aporte e juros no patrimônio} \label{tab:juros}
        \end{center}
    \end{footnotesize}
\end{table}

\begin{figure}[H]
    \centering
    \includegraphics[width=0.7\linewidth]{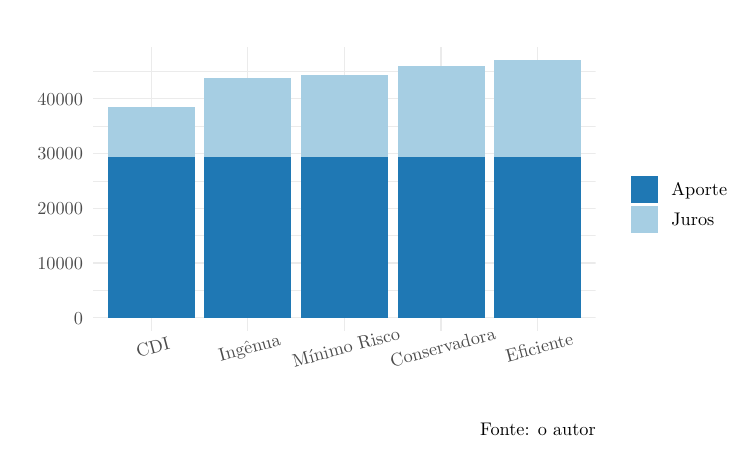}\\
    \caption{\footnotesize Valor do aporte e juros no patrimônio}
    \label{fig:grafju}
\end{figure}

Destacamos aqui a importância que um bom investimento faz na construção do patrimônio. Note que no pior dos casos, $23\%$ do patrimônio é composto de juros e no melhor dos casos esse percentual foi de $36\%$. O retorno nos investimentos deve ser encarado como um bom ajudante que ``contribui ativametne'' na construção do patrimônio. Quem investiu apenas no CDI teve um bom ``ajudante'', mas quem investiu na Estratégia do Portfólio Eficiente teve um ajudante melhor. Lembramos ainda que aquele que poupa e investe parte do seu ganho todo mês está deixando de usufruir do dinheiro no presente para usufruir do capital no futuro e, com um bom investimento ele será recompensado e consumirá mais no futuro pois o capital total crescerá junto com seu esforço e com a ``ajuda'' dos juros compostos.

É possível notar claramente a importância dos juros no patrimônio quando a estrategia tem um tempo de duração considerável. Ao analisar a Tabela \ref{tab:juros}, observamos que, conforme o tempo avança, o percentual de juros em todas as estratégias tende a se expandir de forma que, em um horizonte temporal mais extenso, o aporte inicial de R\$ $400,00$ se tornará ínfimo comparativamente ao volume total acumulado. Esse fenômeno ressalta a dinâmica do efeito dos juros compostos, onde os retornos sobre os ganhos anteriores começam a constituir uma fração substancial do montante, destacando a importância do tempo e da regularidade dos aportes no crescimento do patrimônio a longo prazo.

Além disso, é notável que durante o período analisado, presenciamos uma série de eventos significativos que fornecem ricas oportunidades de análise. Notavelmente, a pandemia global de 2020 destaca-se como um evento transformador, provocando uma queda generalizada nos mercados de ações. Essa tendência negativa é visível na performance de todas as estratégias em nosso gráfico. Em momentos críticos como esse, é crucial observar qual estratégia apresenta a menor retração, visto que a recuperação se torna mais desafiadora quanto maior a perda de patrimônio; recuperar e, por conseguinte, incrementar o capital inicial pode ser um processo árduo. Sob essa ótica, a  Estratégia Conservadora, com sua alocação ativa em CDI, demonstrou maior resiliência, sofrendo a menor queda e exibindo uma recuperação mais ágil.

Em contrapartida, o comportamento do CDI no intervalo entre 2018 e final de 2023 foi notavelmente atípico. Observamos períodos como o da pandemia, onde o CDI teve retornos minimalistas, seguidos por uma ascensão pós-pandêmica, atingindo níveis desafiadores para o investidor individual, considerando o perfil de risco quase nulo. Embora esse surto de crescimento seja um fenômeno fora do comum e possivelmente um prenúncio de uma eventual estabilização, o que desejo ressaltar é que, mesmo considerando as flutuações do CDI, todas as estratégias tiveram desempenho superior ao CDI, como pode ser discernido no gráfico fornecido.

Apesar da análise do patrimônio final nos fornecer alguma visão sobre as estratégias, a análise padrão para comparar investimentos distintos é a cotização (veja a Seção \ref{cotizacao}) dos mesmos. A Tabela \ref{Tab:Ret} contém o retorno acumulado no período analisado, podemos ver que todas as estratégias superaram de longe o retorno alcançado pelo CDI.

\begin{table}[H]
    \begin{center}
        \begin{tabular}{ccccc}

            \textbf{CDI} & \textbf{Conservadora} & \textbf{Ingênua} & \textbf{Eficiente} & \textbf{Mínimo Risco} \\
            \hline\rule{-3pt}{1.1\normalbaselineskip}
            $53,68\%$    & $129,82\%$            & $135,46\%$       & $136,00\%$         & $140,41\%$
        \end{tabular}
        \caption{ \footnotesize Retorno acumulado no período  $01/2018$ e $12/2023$.}
        \label{Tab:Ret}
    \end{center}
\end{table}

Olhando para os dados anualizados vemos, na Tabela \ref{Tab:AnaRet}, que o retorno anual médio de todas as estratégias superaram $15\%$ a.a. algo que podemos considerar um retorno excelente. Na mesma tabela, encontramos o Índice Sharpe de cada estratégia. Esse índice mede o retorno ajustado ao risco, isto é, o retono que a estratégia teve (acima do CDI) para cada real investido. Por exemplo a Estratégia de Mínimo Risco teve um Índice Sharpe de $0,42$ isso significa que essa estratégia ganhou R\$0,42 acima do CDI para cada real colocado em risco, o que configura um bom investimento. Naturalmente, quanto maior o Índice Sharpe mais a estratégia retornou considerando o risco que correu. Por essa análise as estratégias que buscavam menor risco, a Estratégia Conservadora e a de Minimo Risco, foram as que tiveram os melhores índices. Ainda podemos observar que as Estratégias Îngênuas e Eficiente tiveram retornos melhores que a Estratégia Conservadora, mas correram mais risco para alcançar tais retornos logo tiveram índices menores.

\begin{table}[ht]
    \centering
    \begin{tabular}{lccccc}
                           & {\textbf{CDI}} & {\textbf{Ingênua}} & {\textbf{Eficiente}} & {\textbf{Conservadora}} & {\textbf{Mínimo Risco}} \\
        \hline\rule{-3pt}{1.1\normalbaselineskip}
        Retorno Anualizado & 7,47\%         & 15,70\%            & 15,70\%              & 15,20\%                 & 16,10\%                 \\
        Desvio Padrão      &                & 0.185              & 0.186                & 0.171                   & 0.184                   \\
        Índice Sharpe      &                & 0.398              & 0.398                & 0.404                   & 0.420                   \\
        \hline
    \end{tabular}
    \caption{ \footnotesize Análise dos retornos anualizados.}
    \label{Tab:AnaRet}
\end{table}

O Gráfico \ref{fig:grafret} mostra a rentabilidade durante todo o período analisado.

\begin{figure}[H]
    \centering
    \includegraphics[width=0.9\linewidth]{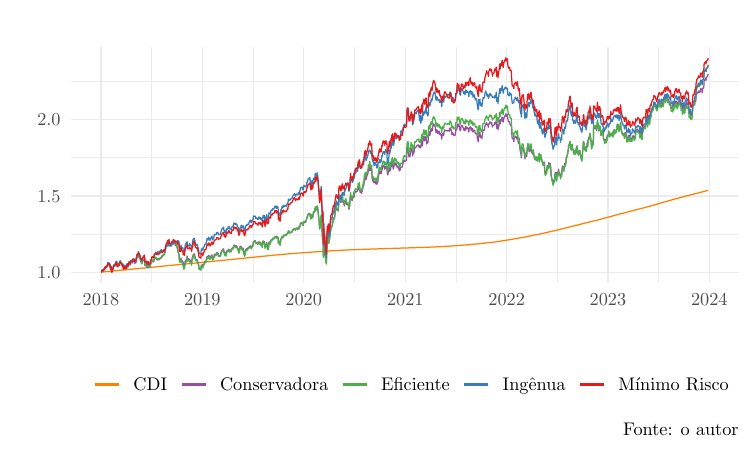}\\
    \caption{\footnotesize Gráfico dos Retornos}
    \label{fig:grafret}
\end{figure}

\section{Conclusão} \label{sec:conclusion}
Enfatizamos que não estamos recomendando o uso da estratégia discutida na Seção~\ref{Sec:Aplicacao}. Nosso objetivo é apenas demonstrar que, utilizando alguns conceitos básicos de finanças, podemos melhorar o desempenho de um determinado portfólio. De fato, as estratégias discutidas na Seção~\ref{Sec:Aplicacao}  possuem considerações importantes para serem aplicadas na prática. Fizemos algumas hipóteses e utilizamos dois ativos muito específicos, e mesmo assim, alguns ajustes são necessários para a aplicação prática da estratégia. Por exemplo, nas Estratégias Ingênua e de Mínimo Risco, seria mais natural considerar a seguinte variação da estratégia mencionada:

\begin{enumerate}
    \item[(APR)] Aumentar o tamanho do portfólio realizando {\it aportes periódicos (AP)} com rebalanceamento preestabelecido. Neste caso, a cada período pré-definido, comprar os ativos com proporções especificadas previamente pelo balanceamento desejado.
\end{enumerate}
Há várias opções para os aportes periódicos, por exemplo:
\begin{enumerate}
    \item[(AI)] {\it Aporte Ingênuo}: Neste caso, os aportes seriam feitos nos dois ativos de modo que, ao final do aporte, a proporção de cada ativo na carteira seja a mais próxima possível de $50\%$;
    \item[(AMR)] {\it Aporte de Mínimo Risco}: Neste caso, os aportes seriam feitos nos dois ativos de modo que, ao final do aporte, a proporção de cada ativo na carteira seja a mais próxima possível do vetor de alocação de capital dado pelo Corolário~\ref{cr:prm}.
\end{enumerate}
Note que, na estratégia acima, estaremos aportando nos dois ativos, o que implica em um custo de transação maior do que na estratégia da Seção~\ref{Sec:Aplicacao}. Naturalmente, se os aportes forem significativos, o custo de transação pode ser considerado desprezível. Embora esta estratégia possa ser superior à estratégia da Seção~\ref{Sec:Aplicacao}, ela não é recomendada para quaisquer dois ativos.

Esperamos que, com este artigo, possamos estimular um maior interesse e aprofundamento no estudo das finanças, incentivando a exploração de estratégias inovadoras, a análise crítica de métodos tradicionais e a aplicação prática de conceitos financeiros para a otimização de portfólios e a gestão de investimentos. Além disso, almejamos contribuir para o desenvolvimento de uma base sólida de conhecimento que possa apoiar tanto acadêmicos quanto profissionais na tomada de decisões financeiras mais informadas e eficazes.


\end{document}